\documentclass[preprint,12pt]{elsarticle}

\usepackage{amsfonts,amsmath,amssymb,amsthm}
\usepackage[linesnumbered, ruled]{algorithm2e}
\usepackage{times}
\usepackage{bm}
\usepackage{balance}
\usepackage{color}
\usepackage{graphicx}
\usepackage{setspace}
\usepackage{comment}
\usepackage{cases}
\usepackage{here}
\usepackage{array}
\usepackage{float}
\usepackage{svg}
\newtheorem{mythm}{Theorem}
\newtheorem{mylem}{Lemma}
\newtheorem{myprop}{Proposition}
\newtheorem{mycor}{Corollary}
\newtheorem{myrem}{Remark}
\newtheorem{myas}{Assumption}
\usepackage{cite}
\usepackage{url}
\usepackage{subfigure}
\newcommand{\rfig}[1]{Fig.\,\ref{#1}} 
\newcommand{\req}[1]{\eqref{#1}}

\newcommand{\rlem}[1]{Lemma\,\ref{#1}}

\newcommand{\rrem}[1]{Remark\,\ref{#1}}
\newcommand{\rsec}[1]{Section\,\ref{#1}}

\newcommand{\ras}[1]{Assumption\,\ref{#1}}

\newcommand{\rprop}[1]{Proposition\,\ref{#1}}
\newcommand{\rthm}[1]{Theorem\,\ref{#1}}

\newcommand{\qedwhite}{\hfill \ensuremath{\Box}}
\usepackage{graphicx}
\usepackage{amssymb}
\usepackage{lineno}
\usepackage{mathrsfs}
\usepackage{accents}
\newcommand{\ubar}[1]{\underaccent{\bar}{#1}}

\usepackage[whole]{bxcjkjatype}
\usepackage{color}
\definecolor{forestgreen}{rgb}{0.33,0.61,0.34}
\definecolor{deepmagenta}{rgb}{0.8, 0.0, 0.8}
\definecolor{harvardcrimson}{rgb}{0.79, 0.0, 0.09}
\newcommand{\masaki}[1]{\textcolor{forestgreen}{\bf [Masaki: #1]}}
\usepackage[normalem]{ulem}
\newcommand{\add}[1]{\textcolor{black}{#1}}

\definecolor{airforceblue}{rgb}{0.36, 0.54, 0.66}

\journal{Annual Reviews in Control}

\begin{document}

\begin{frontmatter}


\title{Event-Triggered Control for \\ Mitigating SIS Spreading Processes}


\author[label3]{Kazumune Hashimoto\fnref{myfootnote}}
\ead{hashimoto@eei.eng.osaka-u.ac.jp}
\author[label1]{Yuga Onoue\fnref{myfootnote}}
\ead{onoue@hopf.sys.es.osaka-u.ac.jp}
\author[label2]{Masaki Ogura}
\ead{m-ogura@ist.osaka-u.ac.jp}
\author[label1]{Toshimitsu Ushio}
\ead{ushio@sys.es.osaka-u.ac.jp}

\address[label3]{Graduate School of Engineering, Osaka University, Suita, Japan}
\address[label1]{Graduate School of Engineering Science, Osaka University, Toyonaka, Japan}
\address[label2]{Graduate School of Information Science and Technology, Osaka University, Suita, Japan}

\fntext[myfootnote]{These authors are contributed equally to this work.}
%\address[label1]{Osaka university, Toyonaka, Japan}

\begin{abstract}
In this paper, we investigate the problem of designing event-triggered controllers for containing epidemic processes in complex networks. We focus on a deterministic susceptible-infected-susceptible (SIS) model, which is one of the well-known, fundamental models that capture the epidemic spreading. The event-triggered control is particularly formulated in the context of viral spreading, in which control inputs (e.g., the amount of medical treatments, a level of traffic regulations) for each subpopulation are updated \textit{only when} the fraction of the infected people in the subpopulation exceeds a prescribed threshold. 
We analyze the stability of the proposed event-triggered controller and derive a sufficient condition for a prescribed control objective to be achieved. 
Moreover, we propose a novel emulation-based approach towards the design of the event-triggered controller, and show that the problem of designing the event-triggered controller can be solved in polynomial time using a geometric programming.
We illustrate the effectiveness of the proposed approach through numerical simulations using an air transportation network. 
\end{abstract}


\end{frontmatter}

%\masaki{みなさま：日本語で書き込みしたかったので\texttt{bxcjkjatype}パッケージを読み込んでいます．不具合あれば削除してください．}

\section{Introduction}

%\masaki{尾上さん：この文献 \citep{Casella2021} は大事なのでどこかで引用してください．}
Analysis and control of epidemic {processes} in complex networks {\citep{RevModPhys}} have been studied over the past decades in {several} research {fields} including epidemiology, computer science, social science, and control {engineering}, {with applications in} epidemic spreading of infectious diseases over human contact networks \citep{virusspread}, malware spreading over computer networks \citep{malwareepidemic}, rumor propagation \citep{socialepidemic}, {and} cascading failures or blackouts in electrical networks \citep{cascadingfailure}. {This research trend has been further strengthened by the current} COVID-19 {pandemic}, which {is posing a significant threat to} {humanity} and economy worldwide. 
%\masaki{橋本先生：この文へうまくつながらないように思えます．というのもここまではどちらかというと理論的な記述なのに，この文では実際に起ったことが述べられています．次の段落とのつながりを意識するなら，制御に限らない分野での研究について述べるのが適当でしょうか．}It has been shown that, as we have seen in the situation of the COVID-19, a small outbreak of an infectious disease can cause the disease spread over many parts of area quickly and easily. On the other hand, the spread of the infectious disease can be mitigated by several social strategies, such as providing medical treatments (e.g., prescribing medicines or vaccines), traffic regulations or a lockdown aiming to reduce interactions among humans. Hence, a critical problem is to analyze how the infectious disease spread and design suitable mitigation strategies, such that the disease spread can be mitigated effectively and efficiently. 
In recent years, an increasing attention has been paid in the systems and control theory toward the analysis and control of epidemic spreading processes, where the amount of medical treatments or a level of traffic regulations are treated as control inputs to be designed, as in \citep{RevModPhys,preciado2013traffic,preciado2014traffic,epidemiccontrolsurvey}. 
% applications of \textit{control theory} to the {analysis and control of} epidemic {processes} in complex networks have received an increased attention \citep{epidemiccontrolsurvey}. In epidemic processes of an infectious disease, the amount of medical treatments or a level of traffic regulations are regarded as {control inputs} that are allowable to be designed. 
Specifically, 
% feedforward or feedback 
control strategies are designed  
% under certain performance guarantees on the epidemic 
for the mitigation of the epidemic spreading processes (e.g., asymptotic or exponential convergence towards the disease-free state), based on the analysis of network structure and dynamical systems capturing the spreading processes. In this context, most of the early works consider {feedforward} control strategies, in which suitable control inputs that are applied \textit{constantly} for all the times will be designed 
% based on convex optimizations
% \citep{preciado2013,preciado2014,han2015,nowzari2017,mai2018}
\citep{preciado2014,han2015,nowzari2017,mai2018}.  %\masaki{ここにあったpreciado2013は除きました．TCNSの方だけを引用してほしいはずなので}. 
{More recently,  
% {feedback} or time-varying
dynamical and feedback
control strategies have been investigated, in which control inputs are determined to adapt the number of infected people and the time-varying nature of the dynamics \citep{ogura2016b,kohler18,outputfeedback2018,watkins,liu2019,gracy2020,jhohanson}.}

In this paper, we are particularly interested in designing a novel {feedback} control strategy for a deterministic susceptible-infected-susceptible (SIS) model, which is one of the well-known and fundamental models for the epidemic spreading processes (see, e.g., \citep{sistutorial,Mei2017}).  %\masaki{引用文献の多様性のためにMei2017を追加しました．}). 
The deterministic SIS model can address epidemic spreading processes in the following two different contexts. The first one is the \textit{individual}-based context, where the model consists of $n \geq 2$ individuals interacting with each other, and the state of each node represents the probability that the individual is infected. 
%In other words, we consider a network or a graph, in which nodes and edges represent individuals and interactions among them, respectively, and investigate how the infections are 
 %that consists of $n$ nodes that represent 
%In other words, the interactions among the individuals 
%The first one is given in an \textit{individual} context, where each node in the graph indicates 
The other one is the \textit{metapopulation} context, where the model consists of $n \geq 2$ subpopulations containing a group of individuals, or called \textit{subpopulations}, and the state of each node indicates the fraction of the infected individuals in each subpopulation. This paper focuses on the metapopulation context, as the state of each subpopulation can be measured or estimated in real time 
% \onoue{under certain assumptions on the model戻ってくる} 
(see, e.g., \citep{epidemiccontrolsurvey}), so that the state-feedback controller can be reasonably implemented in practice. %\masaki{この段落なのですが，「フィードフォワードとフィードバック」と「SISモデル」という二つのテーマがあってわかりづらく感じます．更に次の段落で続けてSISが述べられるので，一つのテーマについて複数の段落でまたがって述べられている感もあります．次々段落につなげるために，ここでは前者のテーマにとどめてはどうでしょうか．SISをやるということは，この段階で述べなくても構わないかと思います．}

Various dynamical and feedback control strategies for the SIS models have been proposed, see, e.g., \citep{kohler18,liu2019,gracy2020}. 
For example, \citep{kohler18} proposed to employ a model predictive control (MPC), in which control inputs are computed by solving a finite horizon optimal control problem online. The authors in \citep{kohler18} also showed that the asymptotic stability of the disease-free state is guaranteed by the proposed controller. Moreover \citep{liu2019} investigated the instability of the disease-free state under a linear state-feedback controller for a bi-virus epidemic spreading model. In addition, \citep{gracy2020} proposed a method for designing a time-varying controller 
% to ensure asymptotic stability of the disease-free state by making use of a periodic nature 
for the asymptotic stabilization of the epidemic processes. 

%For example, \citep{hoge} proposed to employ a model predictive control (MPC), in which control inputs are computed by solving a finite horizon optimal control problem online. They also shows that 
%So far, control strategies for the spreading processes have been investigated for various types of dynamical systems, such as susceptible-infected-susceptible (SIS) models \citep{hoge}, susceptible-infected-recovered (SIR) models \citep{hoge}, and among others \citep{hoge}. In this paper, we are particularly interested in designing a suitable control strategy for the deterministic SIS model. The deterministic SIS model captures the spreading processes in the following two ways. The first one is given in an individual context, where the model consists of $n \geq 2$ individuals interacting with each other, and each state, say $x_i, i \in \{1, \ldots, n\}$, indicates an expected probability that the individual $i$ is infected. The second one is given in a metapopulation context, where the model consists of $n \geq 2$ groups of individuals, and each state $x_i, i \in \{1, \ldots, n\}$ indicates a fraction of the infected individuals in group $i$. This paper focus on the metapopulation context, as the state can be measured or approximated in real life according to the number of infected people in the group, so that the state-feedback controller can be implemented. 

Feedback control strategies presented in the aforementioned papers \citep{kohler18,liu2019,gracy2020} assume that control inputs are updated \textit{continuously} (for the continuous-time case) or per unit of time (for the discrete-time case). {However, such frequent control updates are not {necessarily} suitable in practice, since a even small fluctuation of the processes forces us to update the control inputs.} For example, the level of traffic regulation, which can be regarded as one of the control inputs for mitigating the epidemic spreading processes, is not necessarily realistic to update continuously because even a small adjustment could require enormous efforts. 
Thus, a more practical approach will be to update the control inputs \textit{only when} they are 
% needed instead of continuously. 
necessary, rather than updating the inputs at every time instants. 
For example, the level of traffic regulation may be preferable to be changed \textcolor{black}{\textit{only when} the fraction of the infected people in each subpopulation increases or decreases by the prescribed thresholds.} 

Motivated by the above {observation}, %\masaki{the above は曖昧で特定できないので名詞をつけたほうが}, 
in this paper we propose to employ an \textit{event-triggered} control-based framework \citep{heemels2012a} for {containing} epidemic spreading processes. In the {proposed framework}, %\del{event-triggered control}, 
control inputs for each subpopulation are updated only when the fraction of the infected people in the subpopulation increases or decreases by a given threshold. 
As can be seen in the {current} situation of the COVID-19, employing the event-triggered control for the epidemic processes is reasonable and useful in practice, as many countries have been carrying out their mitigation strategies in an \textit{event-triggered} manner. 
%have established different mitigation strategies according to the level of infections.  %\masaki{ここ，よくわからないです．連続的に政策を変化させることにはならないからevent-triggeredは有用である，とシンプルに書いたほうが良いように思えます．それか，現在各自治体がやっていることがevent-triggeredなので，その実情に即した制御理論の構築が望ましい，とかいう言い方もあると思います．}. 
In Japan, for example, each prefecture has created its own cautionary levels according to the number of COVID-19 cases, in which, for each cautionary level, the contents of requests to prefectural residents are provided (see, e.g., \citep{corona}). In other words, each prefecture dynamically updates its own mitigation strategies \textit{only when} the number of the COVID-19 cases increases or decreases to some extent. Therefore, the event-triggered control-based framework {could serve} as a useful {decision-making system} to {inform us of when to update our mitigation strategies in response to the dynamic change in the number of infected people} {in subpopulations}.

The main contribution of this paper is to formulate {a framework for} the event-triggered {containment of} the deterministic SIS model. We furthermore formulate the event-triggered control in a \textit{distributed} manner{, in which} control inputs for each subpopulation (or node in the graph) are updated based on the well-designed {local event-triggered} conditions. \textcolor{black}{The SIS model may not be appropriate to concretely capture the dynamical behavior for the current COVID-19 pandemic. Nevertheless, this paper can be viewed as a first step towards a rigorous, mathematical formulation of the event-triggered control for mitigating the epidemic spreading. In particular, we provide both theoretical analysis on stability and design procedure of the event-triggered control in a computationally efficient way.} 
%{Although} most of the up-to-date results {for the containment of} epidemic {spreading} processes {focuses on the} asymptotic stabilization of the disease-free state, such control objective {could} be too conservative and restrictive {because} it may not be realistic to achieve the disease-free state {while maintaining societal functions}. Therefore, {in} the event-triggered control {framework proposed in this paper, we adopt a different} control objective {in which} the fraction of the infected people must be eventually suppressed {\textit{below a prescribed threshold}} instead of {requiring its convergence to} zero.
Specifically, {the technical contribution}  of this paper {is twofold}: 
\begin{enumerate}[\textrm{(}i\textrm{)}]
\item  We derive a sufficient condition {for the event-triggered controller to achieve} the prescribed control objective. We further show that the condition can be checked by solving a convex program; for details, see \rsec{stabilitysec}. 
\item Based on the {analysis} given in (i), we {then} propose a novel framework to \textit{design} the event-triggered control for mitigating the SIS spreading processes. {In particular}, we propose to leverage an \textit{emulation-based} approach (see, e.g., \citep{heemels2013a,heemels2013b}) as {a} two-step procedure to design parameters {of} the event-triggered control. As we will see later, the main advantage of employing the emulation-based approach is that the problem of designing the event-triggered control can be formulated by a \textit{geometric programming} {(see, e.g., \citep{boyd,Ogura2019c})}, which can be translated into the convex program and thus can be solved in polynomial time; for details, see \rsec{eventtrigdesignsec}. %\rprop{gpproposition} and \rprop{gppropositione}. 
%As we will see later, the emulation-based approach is useful to formulate the design problem as a geometric programming. 
%the parameters to characterize the event-triggered controller. 
\end{enumerate}
%Such control objective may be related to \citep{hoge}, in which they showed that, under a linear state-feedback controller, the origin becomes unstable and  
%In this paper, the event-triggered controller 
%A typical control objective in the previous works of controlling the epidemic processes is to achieve an asymptotic stabilization (or exponential stabilization) of the origin, meaning that the infection should eventually die out for all nodes. In view of the current situation of COVID-19, however, such control objective may be too conservative and restrictive, and it may be more preferable to suppress the number of infected people \textit{within a certain threshold} (instead of zeroing it), aiming at activating social economy and allowing people to make their living. 

Our approach is related to applications of event-triggered control for \textit{multi-agent systems}, (see, e.g., \citep{dimos2012a,dimos2013a,fan2013,nowzari2019}). 
Note that our {result} %\masaki{approachについて語ってはいないような気が？} 
differs from these previous results in terms of both analysis and design in the following {aspects}. While most of the previous works consider \textit{linear} %\masaki{橋本先生：linear入れましたが正しいですか？} 
multi-agent systems with single or double integrator dynamics and the control objective is to {asymptotically} achieve {a} {consensus}, {we study} the \textit{non-linear} {dynamical system arising from the} SIS model and{, furthermore,} the control objective is not achieving the consensus but to {asymptotically} suppress the fraction of the infected people {below} prescribed thresholds. As will be seen in \rsec{stabilitysec}, this leads to the stability analysis of an event-triggered controller for %\del{\textit{non-negative}} \masaki{非負システムの引用} 
{\emph{positive}} and \textit{quadratic} dynamical systems, which has not been {fully} investigated in the literature. The design procedure {presented in this paper} is the {emulation-based design using a geometric programming, which} also {differs} from the {ones in the aforementioned} works.

The remainder of this paper is organized as follows. 
%In Section~2, we \add{briefly} review \del{some} basic concepts of \del{a} geometric programming. 
In Section~2, we describe dynamics of the SIS model and the control objective to be achieved in this paper. In Section~3, we describe the details of {the proposed} event-triggered control for the SIS model. In Section~4, we investigate the stability that derives a sufficient condition for achieving the control objective under the event-triggered controller. 
In Section~5, we provide an emulation-based approach towards the design of the event-triggered controller. In Section~6, numerical simulations are given to illustrate the effectiveness of the proposed approach. Finally, conclusions and future works are given in Section~7.  

%\linenumbers
\textit{(Notation and convention):} Let $\mathbb{R}$, $\mathbb{R}_{>0}$, $\mathbb{N}$ denote the set of real numbers, positive real numbers, and non-negative integers, respectively. 
%We denote by $\mathbf{R}^n$ (resp. $\mathbb{R}_{0\ge 0}^n$) the set of $n$-dimensional vectors with the entries of real numbers (resp. nonnegative real numbers). We denote vectors using boldface letters and matrices using capital letters. 
Let $0$ denote the vector or matrix whose elements are all $0$. Let $I_n$ and $\mathsf{1}_n$ denote the $n \times n$ identity matrix and the $n$-dimensional vector whose elements are all $1$. 
% Let \del{$0_n, {1}_n \in \mathbb{R}^n$} \add{$0_n$ (${1}_n$)} denote the $n$-dimensional vector whose elements are all $0$ \del{and} \add{(}$1$, respectively\add{)}. Let $I_n$ denote the $n \times n$ identity matrix.%Given a vector $x \in \mathbb{R}^n$, let $x_i$ be the $i$-th element of $x$. Similarly, given a matrix $A\in \mathbb{R}^{n\times n}$, let $a_{ij}$ be the $(i, j)$-th element of $A$. 
%A diagonal matrix is denoted as ${\rm diag}(\cdot)$.
The transpose of vectors and matrices are denoted by $(\cdot)^{\mathsf{T}}$. 
% \del{Given ${x}\in \mathbb{R}^n$ and $A \in \mathbb{R}^{n\times m}$, their transpose are denoted by ${x}^\mathsf{T}$, $A^\mathsf{T}$, and} 
For any real vector ${x} = [x_1, x_2, \ldots, x_n]^\mathsf{T} \in \mathbb{R}^n$, 
% \del{let $\|x\|$ denote the} 
the Euclidean norm and the $\ell_1$ norm are denoted by $\|x\|$ and $\|x\|_1$, respectively (i.e., $\|x\|=\sqrt{x^2 _1 + x^2 _2 + \cdots + x^2_n}$ and $\|x\|_1 = |x_1| + |x_2| + \cdots + |x_n|$).
%(resp. the $\ell_1$-norm) of $x$ is denoted by $\|x\|$ (resp. $\|x\|_1$). 
% \del{Given $x = [x_1, x_2, \ldots, x_n]^\mathsf{T} \in \mathbb{R}^n$,} 
\textcolor{black}{Moreover, let ${\rm diag}(x)$ denote the $n \times n$ diagonal matrix} whose $i$th diagonal 
% \del{element is} 
equals $x_i$. 
% \del{We write $A\preceq 0$ if the matrix $A \in \mathbb{R}^{n\times n}$ is negative semidefinite.}
% \del{Given $x = [x_1, x_2, \ldots, x_n]^\mathsf{T} \in \mathbb{R}^n$, denote by } 
In addition, denote by ${\rm supp}({x})\subseteq \{1,\ldots, n\}$ the 
% \del{set of indices on which} 
support of $x$,
% \del{is supported. That is, we have $i \in {\rm supp}(x)$ if and only if $x_i$ is non-zero, }
i.e., we define the set ${\rm supp}({x})$ by
% \begin{equation}
% \del{i \in {\rm supp}({x})\Leftrightarrow x_i\neq 0,\ \ i\in % \{1,\ldots, n\}. \nonumber}
% \end{equation}
\begin{equation}
{{\rm supp}({x}) = \{  i\in \{1,\ldots, n\} \mid x_i\neq 0 \}.} \nonumber
\end{equation}
% \del{Given a set ${\cal N}$, we denote by $|{\cal N}|$ the cardinality of ${\cal N}$.}
For any two real vectors $x = [x_1,\ldots, x_n] \in \mathbb{R}^n$, $y = [y_1,\ldots, y_n] \in \mathbb{R}^n$, we write $x\leq y$ if and only if $x_i \leq y_i$ for all $i \in \{1, \ldots, n\}$. 
%{For any real and square matrix~$A$, we write $A\preceq 0$ if $A$ is negative semidefinite.} 
Given a set ${\cal N}$, we denote by $|{\cal N}|$ the cardinality of ${\cal N}$.

A directed graph is defined as the pair $\mathcal{G}=(\mathcal{V}, \mathcal{E})$, where $\mathcal{V}= \{1,\ldots, n\}$ is the set of nodes and $\mathcal{E}\subseteq \mathcal{V}\times \mathcal{V}$ is the set of edges. 
The set of \textit{out}-neighbors of node $i$ is denoted by $\mathcal{N}_i^{\rm out}$, i.e., $\mathcal{N}_i^{\rm out} = \{j\in \mathcal{V}:(i, j)\in\mathcal{E}\}$. Similarly, the set of \textit{in}-neighbors of node $i$ (including node $i$ itself) is denoted by $\mathcal{N}_i^{\rm in}$, i.e., $\mathcal{N}^{\rm in} _i = \{j \in \mathcal{V} : (j ,i) \in \mathcal{E}\}$. 
\textcolor{black}{Note that, from the definition, if a certain node includes the self-loop, both the in- and the out-neighbors include the node itself, i.e., if $(i,i) \in \mathcal{E}$, then $i \in \mathcal{N}_i^{\rm in}$ and $i \in \mathcal{N}_i^{\rm out}$.}
%\masaki{先にout-neibhborを定義して，それにsimilarlyで続けてin-neighborsを定義してください．今の文章は複雑すぎると思います．}. 
%The directed graph $\mathcal{G}$ is called weakly connected if for every pair of nodes, when removing all the orientations in the graph, there exists a path between them. 
%$\mathcal{G}$ is called strongly connected if for every pair of nodes, there exists a path between them. 

%\section{Preliminaries of geometric programming} \label{pregp}
%\masaki{このsectionだけ浮いていると思います．他のsectionと内容の性質も量も違うので．例えば前節の"Notation"を"Notation and convention"にして，この節の内容をそこに含めるのはどうでしょうか．}
Next, we shall review some basic concepts of the geometric programming \citep{boyd}. 
Let the positive variables be given by $y = [y_1,\ldots,y_n]^\mathsf{T} \in \mathbb{R}^n _{>0}$. 
%$y_1,\ldots,y_n > 0$ denote positive variables, and define $y=[y_1,\ldots,y_n]$. 
%In the context of GP, 
A function $g : \mathbb{R}^n _{>0} \rightarrow \mathbb{R}_{>0}$ is called a \textit{monomial} if it is given of the form: $g(y)= cy_1^{a_1}\cdots y_n^{a_n}$, where $c \geq 0$ and $a_1\ldots,a_n\in \mathbb{R}$ are given constants. 
%defined if there exits $c\ge 0$ and $a_1\ldots,a_n\in \mathbb{R}$ such that $g(y)= cy_1^{a_1}\cdots y_n^{a_n}$. 
Moreover, a function $f : \mathbb{R}^n _{>0} \rightarrow \mathbb{R}_{>0}$ is called a \textit{posynomial} if it is given of the form: $f(y)= \sum_{i=1}^n c_iy_1^{a_{1i}}\cdots y_n^{a_{ni}}$, where $c_i \geq 0$ and $a_{1i}, \ldots, a_{ni} \in \mathbb{R}$ are given constants. 
Given a set of posynomial functions $f_0,f_1,\ldots,f_k : \mathbb{R}_{>0}^n \rightarrow \mathbb{R}$ and a set of monomials $g_1,\ldots,g_q$, a \textit{geometric program} is an optimization problem given of the form:
\begin{align}
  \rm{minimize}&~~~f_0(y)\nonumber\\
  \rm{subject~to}&~~~f_i(y)\le 1,~i=1,\ldots,k\nonumber\\
  &~~~g_i(y)=1,~i=1,\ldots,q\nonumber
\end{align}
%where $y_i$ are the optimization variables and have an implicit constraint that they are positive (i.e., $y_i > 0$). 
Although a geometric programming is not a convex program by itself, it can be converted into a convex problem with logarithmic changes of variables and logarithmic transformation of the objective and the constraint functions, (see, e.g., \citep{boyd2007tutorial}). %\masaki{こっちにしましょう S. Boyd, S.-J. Kim, L. Vandenberghe, and A. Hassibi, “A tutorial on geometric programming,” Optim. Eng., vol. 8, no. 1, pp. 67–127, 2007.}. 
Hence, the geometric program can be solved efficiently in polynomial time.

%a set of ordered pairs of nodes  called directed edges and $(v_j, v_i)\in \mathcal{E}$ is an edge from $v_j$ pointing toward $v_i$, and $A=[a_{ij}]_{ij}\in \mathbb{R}^{n\times n}$ is a weighted adjacency matrix defined entry-wise as $a_{ij}>0$ if and only if $(v_j, v_i)\in \mathcal{E}$, and $a_{ij}=0$ otherwise. We define the index set of neighbors of a node $v_i\in \mathcal{V}$ as $\mathcal{N}_i=\{j|(v_j, v_i)\in \mathcal{E}\}$.

%Let a weighted, directed graph (also called digraph) be given by ${\mathcal{G}}=(\mathcal{V}, \mathcal{E}, A)$, where $\mathcal{V}=\{v_1,\ldots,v_n\}$ is a set of $n$ nodes, $\mathcal{E}\subseteq \mathcal{V}\times \mathcal{V}$ is a set of ordered pairs of nodes  called directed edges and $(v_j, v_i)\in \mathcal{E}$ is an edge from $v_j$ pointing toward $v_i$, and $A=[a_{ij}]_{ij}\in \mathbb{R}^{n\times n}$ is a weighted adjacency matrix defined entry-wise as $a_{ij}>0$ if and only if $(v_j, v_i)\in \mathcal{E}$, and $a_{ij}=0$ otherwise. We define the index set of neighbors of a node $v_i\in \mathcal{V}$ as $\mathcal{N}_i=\{j|(v_j, v_i)\in \mathcal{E}\}$.

\section{System description and control objective}\label{dynamicssec}

In this section, we describe the dynamics of the epidemic spreading and the control objective in this paper. 
\subsection{Deterministic SIS model} 

As previously stated in the Introduction, we adopt a deterministic susceptible-infected-susceptible (SIS) model in terms of the \textit{metapopulation} context (see, e,g., \citep{epidemiccontrolsurvey}). 
%Metapopulation models allow groups of individuals in the network to be lumped together into fixed subpopulations. The dynamics of the metapopulation model are defined with the assumption that each subpopulation is well mixed, that is, each individual has ahomogeneous recovery and infected rate, and each individualhas equal contact with everyone else within it. 
Consider a network that consists of $n$ ($n\geq 2$) groups of individuals, which are labeled by $\{1, \ldots, n\}$. \textcolor{black}{Individuals in each group are affected by those in their own group or those in the neighboring groups.} The neighbor relationships among the groups are captured by a directed graph $\mathcal{G} = (\mathcal{V}, \mathcal{E})$, where $\mathcal{V} = \{1, \ldots, n \}$ is the set of nodes and $\mathcal{E} \subseteq \mathcal{V}\times \mathcal{V}$ is the set of edges. If $(i, j) \in \mathcal{E}$, it means that the node $i$ is affected by the node $j$. \textcolor{black}{Since individuals can be affected by those in their own group, the graph has the self-loops at all nodes, i.e., $(i, i) \in \mathcal{E}$ for all $i \in \mathcal{V}$. If $(i, j) \in \mathcal{E}$ with $i \neq j$, it means that there is a possibility of contact from individuals in node $i$ to those in node $j$.} 
%Consider a directed graph $\mathcal{G} = (\mathcal{V}, \mathcal{E})$ that forms a metapopulation, where $\mathcal{V} = \{1, \ldots, n \}$ with $n \geq 2$ is the set of nodes and $\mathcal{E} \subseteq \mathcal{V}\times \mathcal{V}$ is the set of edges. Each node $i \in \mathcal{V}$ represents a subpopulation or a group of individuals, and each edge $(i, j) \in \mathcal{E}$ represents a possibility of contact from individuals in node $i$ to the ones in node $j$. 
%Moreover, it is assumed that the graph $\mathcal G$ is weakly connected. 

In the (metapopulation) SIS model, each node has a $[0, 1]$-valued state variable representing the fraction of \emph{infected} individuals in the node.
% \del{Moreover, let}
We let the state of node $i$ at time $t \geq 0$ 
%\del{given} 
denoted by 
% \del{$x_i (t) \in [0, 1]$, which indicates the fraction of individuals that are \textit{infected} in node $i$}
$x_i (t) $. Under this notation, the scalar $1-x_i(t)\in [0, 1]$ represents the fraction of individuals in node $i$ that are not infected, which we call \textit{susceptible} subpopulation. The dynamics of the state variable of node $i$ in the SIS model is then expressed as follows: 
\begin{align}\label{dynamics}
\dot{x}_i (t) = - \delta_i x_i(t) + ( 1- x_i(t)) \sum_{j\in \mathcal{N}^{\rm in}_i} \beta_{ji} x_j(t),\quad t\geq 0,
\end{align}
% \del{for $t \geq 0$, }
where $\delta_i >0$ and $\beta_{ji} >0$ ($j \in \mathcal{N}^{\rm in} _i$) are called the recovery and the infection rates, respectively.
% \del{are the infection rates from the in-neighbors}. %to node $i$. 
\textcolor{black}{Note that $i \in \mathcal{N}^{\rm in}_i$ for all $i \in \mathcal{V}$, since every node has the self-loop.}
As shown in \req{dynamics}, the epidemic spreading in the 
% \del{graph ${\cal G}$ are} 
SIS model is captured by the following two processes: (a) \textit{infected} individuals in node $i$ recover from infection according to the recovery rate $\delta_i$, and 
\textcolor{black}{(b) \textit{susceptible} individuals in node $i$ are infected either from the node $i$ itself according to the infection rate $\beta_{ii}$ or other nodes having their edges to $i$ according to the infection rates $\beta_{ji}$, $j \in \mathcal{N}^{\rm in} _i \backslash \{i\}$.}
%\onoue{直す?:(b) \textit{susceptible} individuals in node $i$ are infected according to the infection rates $\beta_{ji}$ from the in-neighbors $j \in \mathcal{N}^{\rm in} _i$.} 

In this paper, it is supposed that we can \textit{dynamically control} the recovery and the infection rates (see, e.g., \citep{kohler18}). %\masaki{尾上さん：as in [100, 101]... みたいに続けることで，recovery rate と infection rate のdynamic controlは他の文献でもやられているよ，と言えると説得力があって良いです．逆にいうと，それがないために，勝手な仮定を置いていると思われル可能性が．Introにはあるけど．}. 
For example, the recovery rate can be controlled by increasing or decreasing the amount of medical resources and treatments. On the other hand, the infection rates from the neighbors can be controlled by traffic regulations (e.g., decreasing the number of flights) or a closure of facilities such as the sightseeing places. This {consideration} %\masaki{尾上さん：単発のthisは何を指しているか著者にしかわからないので避けましょう．"This consideration"も微妙だけど，応急処置として} 
allows us to replace the constants $\delta_i$ and $\beta_{ji}$ in \eqref{dynamics} with the time-dependent functions $\ubar{\delta}_i + u_i(t)$ and $\bar \beta_{ji}- v_{ji}(t)$, respectively, where $\ubar{\delta}_i >0$ and $\bar \beta_{ji} >0$ ($j\in\mathcal{N}_i^{\rm in}$) {represent} the natural (or, baseline) recovery and infection rates before intervention, whereas the time-dependent scalars $u_i(t)$ and $v_{ji}(t)$ ($j\in\mathcal{N}_i^{\rm in}$) for $t \geq 0$ {represent} the {effect from applying} control inputs {and are assumed to satisfy} $u_i(t) \geq 0$ and $v_{ji} (t) \in [0, \bar \beta_{ji}]$ for all $t \geq 0$. %and $u_i(t) \geq 0$ and $v_{ji}(t) \in [0, \bar \beta_{ji}]$ are the control inputs. 
Then, the {original} SIS dynamics {\eqref{dynamics}} is {rewritten} as follows: 
\begin{align}\label{controldynamics}
\dot{x}_i(t) = - (\ubar{\delta}_i + u_i(t))  x_i(t) + ( 1- x_i(t)) \sum_{j\in \mathcal{N}^{\rm in} _i}(\bar \beta_{ji}- v_{ji}(t)) x_j(t). 
\end{align}
The model \req{controldynamics} can be {further} written in a {vectorial} form as 
\begin{equation}
\dot{ {x}}(t)=-(\ubar{D}+U(t)) {x}(t) + (I_n-X(t))(\bar{B}-V(t)) {x}(t)
\label{dynamics3}
\end{equation}
where ${x}=[x_1,\ldots,x_n]^\mathsf{T}$  {is the state vector and the $n\times n$ matrices $\ubar{D}$, $\bar B$, $X(t)$, $U(t)$, and $V(t)$ are defined by}
\begin{equation}
\begin{aligned}
    \ubar{D}&=\rm{diag}(\ubar{\delta}), \\
    \bar{B} & =[\bar \beta_{ji}]_{i,j}, \\
    X(t)&={\rm diag}(x(t)), \\
    U(t)&= {\rm diag}(u(t)), \\ 
    V(t)&=[v_{ji}(t)]_{i,j}, \\ 
\end{aligned}
\end{equation}
by using the vectors $\ubar{\delta} = [\ubar{\delta}_1, \ldots, \ubar{\delta}_n]^\mathsf{T}$ and $u(t) = [u_1(t), \ldots, u_n(t)]^\mathsf{T}$. 
%In this paper, we deal with \req{dynamics3} so as to design suitable control strategies for mitigating the epidemic spreading\masaki{この文の必要性が？}. 

\subsection{Control objective}\label{ControlObjective}

%\textcolor{black}{As previously described in the Introduction, a typical control objective in the previous works is to achieve asymptotic stabilization of the origin (i.e., disease-free state), in which we require that the infection eventually dies out in all the nodes; however, such a control objective could be too conservative and restrictive. Therefore, we here adopt the following alternative control requirement:}
%Therefore, we here adopt the following alternative control requirement:} 
In this paper, we consider the following control objective: for every {initial state} $x (0) \in [0, 1]^n$, there exists $t' \geq 0$ such that the state trajectory satisfies %{the state trajectory $x(t)$, $t \geq 0$ satisfies}
\begin{align}
{w}_m ^\mathsf{T} {x}(t) \leq \bar{d}_m,\ \forall t \geq t',\ \forall m\in \{1,\ldots,M\}, 
%\limsup_{t\to \infty}\ {w}_m ^\mathsf{T} {x}(t) \leq \bar{d}_m,\ \forall m\in \{1,\ldots,M\},
\label{objective}
\end{align}
%for all $t \geq t'$, 
where $M \in \mathbb{N}_{>0}$ is the number of {given} control objectives, {${w}_1, \dotsc, w_M \in \{0,1\}^n$} are given vectors, and {$\bar{d}_1, \dotsc, \bar d_M \geq 0$} are given thresholds. Using \req{objective}, we can express various control objectives. \textcolor{black}{For example, suppose that we would like to stabilize $x_i$ below the threshold $\bar{x}_i \geq 0$ for all $i\in \{1,\ldots, n\}$ in finite time.} {This} control objective can be expressed by \req{objective} with $M = n$, $\bar{d}_m = \bar{x}_m$ for all $m \in \{1, \ldots, M\}$, and {$w_m$ ($m \in \{1, \ldots, M\}$) being the $m$th canonical basis of $\mathbb{R}^n$}. 
%Moreover, suppose that we would like to stabilize the average of the states within a prescribed threshold $\bar{x} >0$, i.e., $\limsup_{t\to \infty}\ \frac{1}{n}{\sum_{i=1}^n x_i (t)}\leq \bar{x}$. Such control objective can be expressed by \req{objective} with $M=1$ and $\bar{d} = n \bar{x}$.
%For example, suppose that we would like to stabilize each state $x_i$, $i\in \mathcal{V}$ {within the corresponding threshold} $\bar{x}_i$, $i\in \mathcal{V}$. 
%Such control objective can be expressed by \req{objective} with $M = n$ and $\bar{d}_m = \bar{x}_m$ for all $m \in \{ 1, \ldots, M\}$ and $w_m \in \{0, 1\}^n$ is defined such that only the $m$-th element of $w_m$ is $1$ (and $0$ otherwise). 
For another example, suppose that we divide the set of all the nodes into $M$ groups, i.e., 
\begin{align}\label{groups}
{\cal V}_1, \ldots, {\cal V}_M \subseteq \mathcal{V} 
\end{align}
with ${\cal V}_1 \cup {\cal V}_2 \cdots \cup {\cal V}_M = \mathcal{V}$ (for the illustration, see \rfig{groupgraph}), and %Intuitively, each group ${\cal N}_m$, $m \in \{1, \ldots, M\}$ represents a set of nodes or a region that we are interested in (for example, ${\cal N}_1$ represents a subpopulation of certain urban area in a country and ${\cal N}_2$ represents a subpopulation in the other rural area).
we would like to stabilize the average of the states in each group ${\cal V}_m$ below the threshold $\bar{x}_m \geq 0$ in finite time, i.e., 
\begin{align}\label{objectivegroup}
%\limsup_{t\to \infty}\ 
\frac{1}{|{\cal V}_m|} {\sum_{i\in {\cal V}_m} x_i (t)}\leq \bar{x}_m,\ \forall t \geq t', \ \forall m\in \{1,\ldots,M\}, 
\end{align}
for some $t' \geq 0$. 
{This} control objective can be expressed by \req{objective} {if we choose} $\bar{d}_m = |{\cal V}_m| \bar{x}_m$ and {define $w_m$ by}
\begin{equation}
    [w_m]_i = \begin{cases}
    1,&\mbox{if $i \in {\cal V}_m$}, 
    \\
    0,&\mbox{otherwise.}
    \end{cases}
\end{equation}

\begin{figure}[tbp]
  \centering
  \includegraphics[width = 0.4\linewidth]{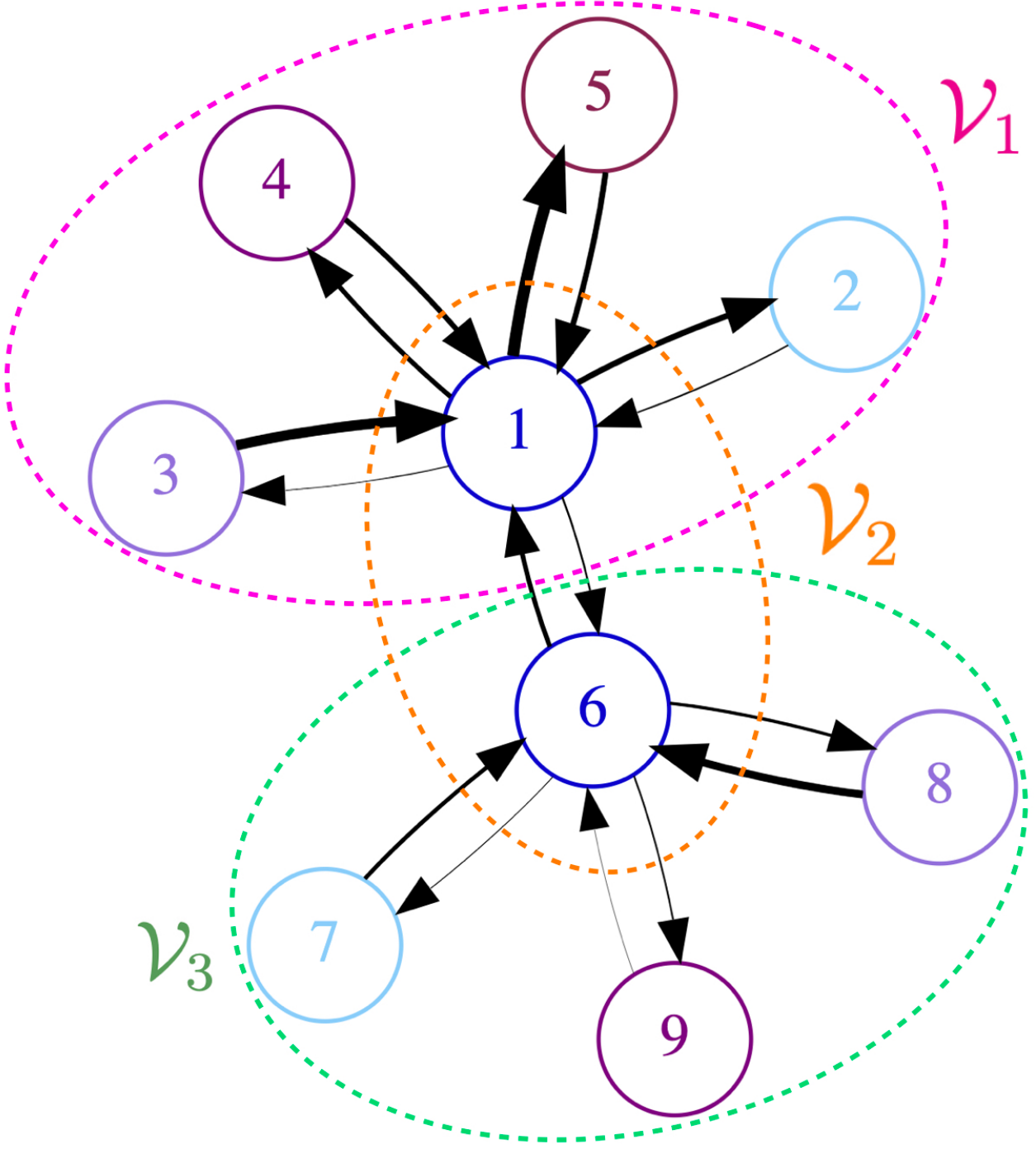}
  \caption{An example of the groups ${\cal V}_1, \ldots, {\cal V}_M$ defined in \req{groups}. The figure considers the case of $n=9$ and $M=3$.}
  \label{groupgraph}
\end{figure}

{We can confirm} that the control objective \req{objective} includes the case of achieving {asymptotic stabilization of the origin} by setting $M=n$, $\bar{d}_m = 0$ for all $m\in \{1, \ldots, M\}$, and {letting $w_m$ ($m\in \{1, \ldots, M\}$) be the $m$th vector in the canonical basis of $\mathbb{R}^m$}. Hence, \req{objective} provides the control objective in a more flexible way than the asymptotic stabilization of the origin. 

\section{Event-triggered control} \label{eventtrigsec}

As previously described in the Introduction, {conventional feedback} control strategies for the SIS models assume that the control inputs {can be} updated \textit{continuously}, i.e., the {amount of} medical resources and the {qualitative degree of} traffic regulations must be changed continuously or even per unit of time. 
However, such frequent control updates are not {necessarily} suitable in practice, since even a small fluctuation of the states (fraction of infected individuals) forces us to update the control inputs. Hence, a more suitable approach would be to update the control inputs \textit{only when} they are needed instead of continuously, i.e., the amount of medical resources and the traffic regulations are changed \textcolor{black}{only when the fraction of infected individuals increases or decreases by the prescribed thresholds.} This leads us to the usage of an \textit{event-triggered control} \citep{heemels2012a}, in which the control inputs are updated only when they are needed according to a well-designed event-triggered condition (as detailed below). 
%\masaki{ちょっとわからない．will get back later.}

To formulate the {proposed} event-triggered control {strategy}, let $t^i _0, t^i _1, t^i _2, \ldots$ with $t^i _0 < t^i _1 < t^i _2 \cdots$ be the \textit{triggering} time instants when control inputs {for} the recovery rate for node $i$, $u_i (t)$, and {the infection rates from node $i$ to its \textit{out}-neighbors (including the node $i$ itself), $v_{ij}(t)$, $(j \in \mathcal{N}^{\rm out}_i)$}, are updated. 
For simplicity, it is assumed that the initial updating time instants for all the nodes, i.e., $t^i _0$, $i\in \mathcal{V}$, are given. For any $t\in [t^i _0, \infty)$, node $i$ evaluates the following event-triggering condition %\masaki{素人的にはここは"event-triggering condition"のほうが自然なのですが，どうでしょうか．}: 
\begin{align}\label{evcond}
|e_i (t)| < \sigma_i x_i(t) + \eta_i, 
\end{align}
where $e_i (t) \in \mathbb{R}$ denotes the error between $x_i (t)$ and the state at the latest triggering time before $t$, i.e., 
\begin{align} 
e_i (t) = x_i(t) - x_i (t^i_{\ell_t}),\ {\rm with}\ \ t^i_{\ell_t} = \underset{\ell' \in \mathbb{N}}{\max} \{ t^i _{\ell'} \geq 0 : t^i _{\ell'} < t \},
\end{align}
where %\masaki{$\sigma_i$だけ属する範囲が指定されていないのが違和感．非負？実数なんでも？} 
the scalars $\sigma_i$, $\eta_i$, $i \in \mathcal{V}$ are the parameters to characterize the event-triggered condition. 
The parameters $\sigma_i$, $\eta_i$, $i \in {\cal V}$ are called the \textit{event-triggering gains}, which will be designed later in this paper. It is assumed that the event-triggering gains are chosen such that 
\begin{align}\label{conditioneventtrig}
\sigma_i \in (0, 1), \ \ \eta_i \in (0, 1),
\end{align}
for all $i\in {\cal V}$. If the condition \req{evcond} is satisfied, then node $i$ does \textit{not} update the control inputs, i.e., $t \neq t^i_{\ell+1}$. On the other hand, if \req{evcond} is violated, then node $i$ updates the control inputs, i.e., $t = t^i _{\ell+1}$. 
More specifically, the triggering time instants are given as follows: 
\begin{align}\label{triggeringtime}
t^i _{\ell+1} = {\inf} \left \{ t >t^i _{\ell}\ : \ |e_i (t)| \geq \sigma_i x_i (t) + \eta_i \right \}
\end{align}
for all $\ell \in \mathbb{N}$. 

\textcolor{black}{Our choice of the event-triggering condition in \req{evcond} (as well as the triggering time instants in \req{triggeringtime}) is motivated as follows. %Notse that in the right hand side of \req{evcond}, 
Intuitively, the term $\sigma_i x_i (t)$ in the right hand side of \req{evcond} becomes more dominant than $\eta_i$ when $x_i(t)$ is large, and $\eta_i$ is more dominant than $\sigma_i x_i (t)$ when $x_i(t)$ is very small. 
For example, when the state is very small, the control inputs are no more updated unless the error $|x_i(t) - x_i (t^i_{\ell_t})|$ exceeds $\eta_i$. In particular, when the state is decreasing and eventually satisfies $0< x_i(t) \leq \eta_i$ for all $t \geq {t}'$ (for some $t'$), the control inputs are no more updated after $t'$ since the error $|x_i(t) - x_i (t^i_{\ell_t})|$ does not exceed $\eta_i$ for all $t \geq {t}'$. Therefore, by using the event-triggering condition \req{evcond}, it can be expected that the frequency of the control updates becomes less and less as the state gets smaller and smaller.}

\begin{myrem}
\normalfont
It is necessary for the event-triggering gains $\eta_i$ to be designed as $\eta_i >0$ for all $i \in \mathcal{V}$ in our problem set-up in order to guarantee that the inter-event times are always positive, i.e., avoid the Zeno behavior or satisfy an event-separation property (see, e.g., \citep{heemels2012a,heemels2014a}). If we set $\eta_i = 0$, we cannot guarantee that the inter-event times are always positive due to the effect of the term $(1- x_i(t)) \sum_{j\in \mathcal{N}^{\rm in} _i}(\bar \beta_{ji}- v_{ji}(t)) x_j(t)$ in \req{controldynamics}; for certain values of this term, the inter-event times can eventually become zero in finite time, (see, e.g., \citep{heemels2014a}). 
If $\eta_i > 0$ for all $i\in\mathcal{V}$, it follows that, for each $t^i _\ell$, the next triggering time $t^i _{\ell+1}$ is given at least after the absolute error $|e_i (t)| = |x_i (t) - x_i(t^i _\ell)|$ reaches $\eta_i$. 
Since $x_i (t)$ is continuous for all $t$, there always exists a $\Delta >0$ such that the event-triggered condition \req{evcond} is satisfied for all $t \in [t^i _\ell, t^i _\ell + \Delta]$, and so the inter-event times are positive for all the times. \qedwhite
\end{myrem}

{For each node $i \in {\cal V}$, the} control inputs are updated according to the following linear state feedback controller: 
\begin{align}
u_i (t) &= k_i x_i (t^i _\ell),\  \forall t \in [t^i _\ell, t^i _{\ell+1}), \label{ui} \\
v_{ij} (t) &= l_{ij} x_i (t^i _\ell)\ \ \forall t \in [t^i _\ell, t^i _{\ell+1}), \ j \in \mathcal{N}^{\rm out} _i, \label{vij}
%a_{ij} (t) &= \ell_{ij} x_j (t^j _\ell),\ \forall t\in [t^j _\ell, t^j _{\ell+1}), 
\end{align}
for all $\ell \in \mathbb{N}$, where $k_i$, $l_{ij}$, $i\in \mathcal{V}$, $j \in \mathcal{N}^{\rm out} _i$ are parameters to characterize the control strategy. 
The parameters $k_i$, $l_{ij}$, $i\in \mathcal{V}$, $j \in \mathcal{N}^{\rm out} _i$ are called the \textit{control gains}, which will be designed together with the event-triggering gains later in this paper. 
%\masaki{so-calledだと，ちまたで一般的にそう呼ばれているように理解してしまいました（小蔵の英語が間違っている可能性あり）．もし意図が「control gains と呼ぶ」ということであれば，単にcalledを使うところだと思いました．} \textit{control gains }, which will be designed together with the event-triggering gains later in this paper. 
It is assumed that the control gains  must be chosen such that 
\begin{align}
k_i \in (0, \bar{k}_i],\ l_{ij} \in (0, \bar{l}_{ij}], \label{constraintcontrolparameters}
\end{align}
for all $i\in\mathcal{V}$ and $j \in {\cal N}^{\rm out} _i$, where $\bar{k}_i >0$ and $\bar{l}_{ij} \in (0, \bar \beta_{ij}]$ ($i\in\mathcal{V}$, $j \in {\cal N}^{\rm out} _i$) are given upper bounds for the control gains of $k_i$ and $l_{ij}$, respectively. 

The following {proposition establishes the invariance of the region $[0, 1]^n$ under the event-triggered controller \req{ui}, \req{vij}.}
\begin{myprop}\label{invariantprop}
\normalfont
Consider the SIS model \req{controldynamics} and the event-triggered controller \req{ui}, \req{vij}, in which the control and the event-triggering gains are, respectively, chosen such that \req{conditioneventtrig} and \req{constraintcontrolparameters} are satisfied for all $i\in\mathcal{V}$ and $j \in {\cal N}^{\rm out} _i$. 
%and assume that the control inputs are piecewise continuous and satisfy $u_i(t) \geq 0$, $v_{ji}(t) \in [0, \bar \beta_{ji}]$ for all $i \in \mathcal{V}, j \in {\cal N}^{\rm in} _i$, $t \geq 0$. 
%Moreover, assume that the graph $\mathcal{G} = (\mathcal{V}, \mathcal{E})$ is strongly connected. 
Then, if $x (0) \in [0, 1]^n$, the state trajectory satisfies $x(t) \in [0, 1]^n$ for all $t \geq 0$. \qedwhite 
%and the event-triggered controller \req{ui}, \req{vij}, in which control gains  are chosen to satisfy \req{constraintcontrolparameters} for all $i\in\mathcal{V}$, $j \in {\cal N}^{\rm out} _i$. Moreover, assume that the graph $\mathcal{G} = (\mathcal{V}, \mathcal{E})$ is strongly connected. Then, for all $x (0) \in [0, 1]^n$, it follows that $x(t) \in [0, 1]^n$, $\forall t \geq 0$. \qedwhite 
\end{myprop}

\begin{proof}
%The proof is given in a similar way to \citep{hoge} and thus we provide a sketch here. 
%The proof is given in a similar way to \citep{hoge}. 
Since we apply the event-triggered controller \req{ui}, \req{vij}, where the control gains  are chosen to satisfy \req{constraintcontrolparameters} for all $i\in\mathcal{V}$ and $j \in {\cal N}^{\rm out} _i$, it follows that the control inputs are piecewise continuous in $t$ and satisfy $u_i (t) \in [0, \bar{k}_i]$ and $v_{ij}(t) \in [0, \bar{\beta}_{ij}]$ for all $t \geq 0$. 
Let 
\begin{align}
f(x, \mu) = -(\ubar{D}+U) {x} + (I_n-X)(\bar{B}-V) {x}
\end{align}
with $\mu$ being the collection of all the control inputs, i.e., $\mu = \left [u_i, v_{ij}, i \in \mathcal{V}, \right.$ $\left. j \in {\cal N}^{\rm out} _{i} \right]^\mathsf{T}$. 
Moreover, let $\Omega = [0, 1]^n$. Since $\Omega$ is bounded and $f(x, \mu)$ is smooth in $x$, there exists an $L>0$ such that $|\partial f/\partial x| \leq L$ for all $x \in \Omega$, $u_i \in[0, \bar{k}_i]$, and $v_{ij} \in [0, \bar{\beta}_{ij}]$, satisfying the Lipschitz condition in $\Omega$ \citep{khalil}. Hence, the solution of $\dot{x} = f(x, \mu)$ exists and unique in $\Omega$. The fact that the state trajectory remains in $\Omega$ for all the times can be shown as follows. 
%for all $t \geq 0$ if the state trajectory stays in $\Omega$ for all times. 
%for all $t \geq 0$ if 
%$x(t) \in \Omega$ for all $t \geq 0$. 
%$u_i(t) \geq 0$, $v_{ji}(t) \in [0, \bar \beta_{ji}]$ for all $i \in \mathcal{V}, j \in {\cal N}^{\rm in} _i$, $t \geq 0$
%For $u_i \in[0, k_i]$, $v_{ij} \in [0, \bar{\beta}_{ij}]$, there exists an $L>0$ such that $|\partial f/\partial x| \leq L$ for all $x \in [0, 1]^n$. 
%Since $u_i, v_{ij}$ are piece-wise continuous and there exists an $L>0$ such that $|\partial f/\partial x| \leq L$ for all $x \in [0, 1]^n$. 
Let $\partial \Omega \subset \Omega$ be the boundary of $\Omega$ and let $\partial \Omega_{1, i}, \partial \Omega_{2, i} \subset \partial \Omega$ for all $i \in \mathcal{V}$ given by 
\begin{align}
&\partial \Omega_{1, i} = \{x \in [0, 1]^n : x_i = 0 \},\\ 
&\partial \Omega_{2, i} = \{x \in [0, 1]^n : x_i = 1 \}.
\end{align}
Note that the union of $\partial \Omega_{1, i}$ and $\partial \Omega_{2, i}$ for all $i \in \mathcal{V}$ comprises $\partial \Omega$. 
%The fact that the state trajectory stays in $\Omega$ can be shown
%If the vector field $f(x, \mu)$ is tangential or pointing into $\Omega$. 
In addition, let $o_{1, i}  \in \{-1, 0\}$ and $o_{2, i} \in \{0, 1\}^n$ for $i \in \mathcal{V}$ be the outer normal vectors with respect to $\partial \Omega_{1, i}$ and $\partial \Omega_{2,i}$, respectively, i.e., $o_{1,i}$ (resp. $o_{2, i}$) is the vector whose $i$-th element is $-1$ (resp. $1$) and $0$ otherwise. Then, for all $x \in \partial \Omega_{1, i}$, $u_i \in [0, \bar{k}_i]$, and $v_{ji} \in [0, \bar \beta_{ji}]$, we obtain 
\begin{align}
&o_{1, i}^\mathsf{T}\ f (x, \mu) = -\sum_{j\in \mathcal{N}^{\rm in} _i}(\bar \beta_{ji}- v_{ji}) x_j \leq 0. \label{vectorfield1} 
%\ \forall x_j \in \Omega, v_{ji} \in [0, \bar \beta_{ji}] \label{vectorfield1}
%&o_{2, i}^\mathsf{T}\ f (x, \mu) |_{x_i = 1} =  - (\ubar{\delta}_i + u_i)  x_i \leq 0,\ \forall x_i \in \Omega, u_i \in [0, \bar{k}_i],\label{vectorfield2}
\end{align}
In addition, for all $x \in \partial \Omega_{2, i}$, $u_i \in [0, \bar{k}_i]$, and $v_{ji} \in [0, \bar \beta_{ji}]$, we obtain 
\begin{align}
&o_{2, i}^\mathsf{T}\ f (x, \mu) =  - (\ubar{\delta}_i + u_i)  x_i \leq 0. \label{vectorfield2}
\end{align}
\req{vectorfield1} and \req{vectorfield2} imply that, for every $x$ on the boundary of $\Omega$ and for every $u_i \in [0, \bar{k}_i]$, $v_{ji} \in [0, \bar \beta_{ji}]$, $i\in \mathcal{V}$ and  $j \in {\cal N}^{\rm in} _i$, the vector field $f(x ,\mu)$ is tangential or pointing inwards $\Omega$, which shows that $\Omega$ is an invariant set (for the related analysis, see, e.g., \citep{LAJMANOVICH}). Therefore, if $x (0) \in \Omega$, the state trajectory satisfies $x(t) \in \Omega$ for all $t \geq 0$.
\end{proof}

\section{Stability Analysis}\label{stabilitysec}

In this section, we analyze the stability of the closed-loop system for the SIS model \req{controldynamics}. %under the event-triggered controller. 
In particular, we investigate a sufficient condition that, under appropriate selections of the control and the event-triggering gains, the control objective \req{objective} is achieved by applying the event-triggered controller \req{ui}, \req{vij}. 

{We start our stability analysis by introducing} several additional parameters and notations.
First, we denote the set of control and the event-triggering gains as $k = [k_1, \ldots, k_n]^\mathsf{T}$, $\sigma = [\sigma_1, \ldots, \sigma_n]^\mathsf{T}$, and $\eta = [\eta_1, \ldots, \eta_n]^\mathsf{T}$ {. We then define the matrices} %let $K, L \in \mathbb{R}^{n\times n}$ be given by
\begin{align}
K&={\rm diag}(k),\ L=[l_{ji}]_{i,j},\\
G &= {\rm diag}(\sigma), \ H = {\rm diag}(\eta).
\end{align} 
%For the notational simplicity in the sequel, we let the set of event-triggering gains as $\sigma = [\sigma_1, \ldots, \sigma_n]^\mathsf{T}$, $\eta = [\eta_1, \ldots, \eta_n]^\mathsf{T}$ and 
%\begin{align}
%G = {\rm diag}(\sigma), \ H = {\rm diag}(\eta).
%\end{align} 
Second, {we define} a {candidate} Lyapunov function by
\begin{equation}
V( {x})= {p}^\mathsf{T} {x},
\label{eq11}
\end{equation}
for a {given} Lyapunov parameter ${p}=[p_1,\ldots,p_n]^\mathsf{T} \in \mathbb{R}_{>0}^n$. %(for how to find an appropriate $p$, see \rrem{lyapunovrem}). 
Note that we can make use of the candidate Lyapunov function as the linear function of $x$, since the dynamics \req{controldynamics} is non-negative for every $x(0) \in [0, 1]^n$ (see \rprop{invariantprop}). 
Moreover, let $p_m^* > 0$ for all $m \in \{1, \ldots, M\}$ be given by  
%Using paremters $p_m^*\in \mathbb{R}, (m=1,\ldots,M)$, denote a minimum value of $p_i$ corresponding to the $i$-th non-negative entry-wise of $ {w}_m$ (that is, the element is $1$) as
\begin{equation}
p_m^*=\min_{i\in {\rm supp}( {w}_m)} p_i, 
\label{pmstar}
\end{equation}
where $w_m \in \{0,1\}^n$ is defined in \req{objective}. 
That is, $p_m^*$ represents the smallest value among the set of the Lyapunov parameters whose indices belong to the support of $w_m$. 
%\masaki{この$p_m^*$の定義はProp 2では使われていないように見えます．だとするとここで定義するのは面白く無いとおもいます．Theorem 1の手前に移動するか，あるいはTheorem 1の中で定義するべきです．}. 
Additionally, {define the vectors} $s, r \in \mathbb{R}^n$ by 
%$k = [k_1, \ldots, k_n]^\mathsf{T}$, $K={\rm diag}(k)$, $L=[l_{ij}]_{i,j}\in \mathbb{R}^{n\times n}$, $\sigma = [\sigma_1, \ldots, \sigma_n]^\mathsf{T}$, $\eta = [\eta_1, \ldots, \eta_n]^\mathsf{T}$, $G={\rm diag}(\sigma)$, $H={\rm diag}(\eta)$, and 
%, and let $n$-vectors $ {s}=[s_1,\ldots,s_n] \in \mathbb{R}_{>0}^n$, $ {r}=[r_1,\ldots,r_n] \in \mathbb{R}^n$, a matrix $Q=[\hat{s}_{ij}]_{i,j}\in \mathbb{R}^{n\times n}$, and a domain ${E}\subseteq \mathbb{R}^n$ be introduced by 
\begin{align}
 {s}^\mathsf{T} &=  {p}^\mathsf{T}(K+L)(I_n-G),
\label{svec}\\
 {r}^\mathsf{T} &=  {p}^\mathsf{T}\{\bar{B}-\ubar{D}+(K+L)H\}.
\label{rvec}
%Q&=S+\frac{1}{2}PL(I_n-G)(G+H)\label{eq15}\\
%\hat{{E}}&=\{ {x}\in [0,1]^n:  {x}^\mathsf{T}Q {x}- {r}^\mathsf{T} {x} \le 0\},
\end{align}
%where $K={\rm diag}(k_i)$, $L=[l_{ij}]_{i,j}\in \mathbb{R}^{n\times n}$, $G={\rm diag}(\sigma_i)$, $H={\rm diag}(\eta_i)$, $S={\rm diag}(s_i)$, and $P={\rm diag}(p_i)$, for $\forall v_i\in\mathcal{V}$. Note that $\sigma_i,\eta_i\in [0,1)$ and $\mathcal{G}$ is a connected graph, so $s_i>0$, for $i = 1,\ldots,n$. 
% \del{where $K$ and $L$ (respectively, $G$ and $H$) are the matrices that collect the control (respectively, event-triggered) parameters, which are all defined in \rsec{eventtrigsec}.}\masaki{直前の繰り返しになるので冗長だと思いました．} 
Finally, define the matrix $Q \in \mathbb{R}^{n\times n}$ and the set ${\cal W} \subset \mathbb{R}^n$ by 
\begin{align}
Q &= S+ \cfrac{1}{2} PL (I_n - G) (G+H),  \label{shat} \\
{\cal W} &=\{x\in \mathbb{R}^n: x^\mathsf{T}Q x- \textcolor{black}{(r+\epsilon \mathsf{1}_n)}^\mathsf{T}x \le 0\} \label{calw}
\end{align}
where $S = {\rm diag}(s)$ \textcolor{black}{, $P = {\rm diag}(p)$, and $\epsilon >0$ is an any positive constant}.

{The following theorem gives a sufficient condition for the control objective \req{objective} to be achieved and is} the main result {of} this section.
%\masaki{この節全体の話として，最初にTheorem 1を述べてあとは証明を行う，という構成にしたほうが理論の重厚感が出て良いのかな，と思います．こうすれば記法の準備も証明に含まれることになるので．今の構成だと，冒頭の記法の準備が煩雑で論文の読みやすさと価値に対してマイナスに働いていると感じます．あとProp 2の前で定義されたベクトル$r$がTheorem 1まで使われないのが不満です．}
\begin{mythm}\label{mainresult}
\normalfont
Consider the SIS model \req{controldynamics}, the event-triggered controller \req{ui}, \req{vij}, and the control objective \req{objective}. Assume that the control and the event-triggering gains satisfying \req{conditioneventtrig} and \req{constraintcontrolparameters} 
are chosen such that the following conditions are satisfied: 
\begin{equation}
\theta^* \le p_m^*\bar{d}_m,  
\label{thetastar}
\end{equation}
for all $m \in \{1, \ldots, M\}$, where $\theta^* \in \mathbb{R}$ is defined according to the following optimization problem: %\masaki{次の式の読み方がわからないです．推測はできるけど．正当な数学の書き方ではないと思います．}
%\masaki{例えば以下のように書き直してくれれば，理解できます．}
\begin{align}
{\theta^* = \underset{x \in {\cal W}}{\rm max}\ \ p^\top x}
\label{optimization}
\end{align}
%\begin{align}
%\theta^* = \ \underset{x\in [0, 1]^n}{\rm max}\ \{ {p}^\mathsf{T} {x} \in \mathbb{R} : {x}^\mathsf{T}Qx - r^\mathsf{T} x \leq 0\}. 
%\label{optimization}
%\end{align}
where ${\cal W}\subset \mathbb{R}^n$ is defined in \req{calw}. 
%the vector $r$ and matrix $Q$ are given by \req{rvec} and \req{shat}, respectively.
Then, for any $x(0) \in [0, 1]^n$, the control objective \req{objective} is achieved by applying the event-triggered controller \req{ui}, \req{vij}. In addition, for every selection of the control and the event-triggering gains satisfying \req{conditioneventtrig} and \req{constraintcontrolparameters}, it follows that the optimization problem \req{optimization} is strictly convex. 
\qed
\end{mythm}
In essence, \rthm{mainresult} states that if the control and the event-triggering gains are appropriately chosen such that \req{thetastar} is satisfied, then the control objective \req{objective} is achieved by applying the event-triggered controller \req{ui}, \req{vij}. \rthm{mainresult} also states that the optimization problem \req{optimization} is strictly {convex} for every selection of the control and the event-triggering gains. Hence, the condition \req{thetastar} can be efficiently checked in polynomial time. 
\begin{myrem}
\normalfont
{Note that the matrix $Q$ may not be a symmetric matrix. Without loss of generality, if the matrix $Q$ is not symmetric, we can replace the set ${\cal W}$ in \req{optimization} with ${\cal W}' = \{x\in \mathbb{R}^n: \frac{1}{2}{x}^\mathsf{T}(Q^\mathsf{T} + Q)x - (r+\epsilon \mathsf{1}_n)^\mathsf{T} x \leq 0\}$, so that $Q^\mathsf{T} + Q$ is the symmetric matrix and the convex optimization problem is given in a standard form \citep{boyd}. Such replacement is valid because ${x}^\mathsf{T}Q x = \frac{1}{2}{x}^\mathsf{T}(Q^\mathsf{T} + Q)x$ for all $x \in \mathbb{R}^n$.}
\qedwhite
\end{myrem}
%if the control gains  and the event-triggering gains are all {given}, i.e., the parameters $Q$ and $r$ in \req{optimization} are fixed. 
%Hence, \rthm{mainresult} is useful if we would like to {verify} whether the control objective is achieved with the control and the event-triggering gains that have been chosen apriori. 

%From \rprop{positiveshat}, the optimization problem \req{optimization} is a \textit{convex} program with the linear objective function and the quadratic constraint~\citep{boyd}\footnote{Note that $Q$ may not be a symmetric matrix. If $Q$ is not symmetric, we replace the constraint ${x}^\mathsf{T}Qx - r^\mathsf{T} x \leq 0$ in \req{optimization} with $\frac{1}{2}{x}^\mathsf{T}(Q^\mathsf{T} + Q)x - r^\mathsf{T} x \leq 0$, so that the convex optimization problem is given in a standard form. Such replacement is valid because ${x}^\mathsf{T}Q x = \frac{1}{2}{x}^\mathsf{T}(Q^\mathsf{T} + Q)x$ for all $x \in \mathbb{R}^n$.}, if both the control gains  and the event-triggering gains are {given}, i.e., the parameters $Q$ and $r$ in \req{optimization} are fixed. Hence, \rthm{mainresult} is useful if we would like to \textit{verify} whether the control objective is achieved with the control and the event-triggering gains that have been chosen apriori. 

\begin{proof}
Let us first show that the optimization problem \req{optimization} is strictly convex. 
From \req{calw}, the strict convexity of \req{optimization} can be shown by guaranteeing that the matrix $Q$ is positive definite for every selection of the control and the event-triggering gains satisfying \req{conditioneventtrig} and \req{constraintcontrolparameters}. From \req{shat}, it can be shown that the $i$-th diagonal element of the matrix $Q$, denoted by ${q}_{ii}$, is given by 
\begin{align}
{q}_{ii} = (1-\sigma_i)\left\{p_ik_i + \textcolor{black}{ {\frac{\sigma_i + \eta_i}{2}p_i l_{ii}} + \sum_{ j\in \mathcal{N}_i^{\rm out}\backslash \{i\}}p_jl_{ij}}\right\} >0,\label{eq19}
\end{align}
for all $i \in \mathcal{V}$. %Note that we have $\hat{s}_{ii} >0$ for all $i \in \{1, \ldots, n \}$ since $\sigma_i\in [0,1)$ for all $i \in \{1, \ldots, n \}$. 
Moreover, the $(j, i)$-th ($j\neq i$) off-diagonal element of the matrix $Q$, which is denoted as ${q}_{ji}$, is given by 
\begin{align}
{q}_{ji} =
\begin{cases}
\frac{1}{2}(1-\sigma_i)(\sigma_i+\eta_i)p_jl_{ij}>0,\ &{\rm if} \ j \in {\cal N}^{\rm out}_i, \notag \\ 
0 \ &{\rm if} \ \textcolor{black}{j \notin \mathcal{N}_i^{\rm out}}
\end{cases}
\end{align}
Hence, the difference between the $i$-th diagonal element and the sum of the other elements in the $i$-th row is given by 
\begin{align}
&{q}_{ii} - {\color{blue} \sum_{j \neq i } {q}_{ji}} \notag \\ 
&= (1-\sigma_i) \left\{ p_ik_i + {\color{blue} \frac{\sigma_i + \eta_i}{2}p_i l_{ii} + \left (1 - \frac{\sigma_i + \eta_i}{2} \right ) \sum_{j\in \mathcal{N}_i^{\rm out}\backslash \{ i\}}p_jl_{ij}} \right \} >0
\label{diffs}
\end{align}
where we used $1- (\sigma_i + \eta_i)/2 >0$ for all $\sigma_i \in (0,1)$ and $\eta_i \in (0, 1)$. %\masaki{"as"の意味が取りづらいです．素直に「仮定の不等式より$1-(\sigma_i + \eta_i)/2 >0$なので．．．」みたいに書いた方がわかりやすい．}. 
Hence, it follows that $Q$ is a strongly diagonally dominant matrix, which implies that, from the Gershgorin circle theorem~\citep{horn2012matrix}, the matrix $Q$ is positive-definite.

Next, we show that the control objective \req{objective} is achieved by applying the event-triggered controller \req{ui}, \req{vij}. Using \req{ui} and \req{vij}, the closed-loop system is given by 
\begin{align}
\dot{x}_{i}(t)=&-(\ubar{\delta}_i+k_ix_i(t_\ell^{i}))x_i(t)\nonumber\\ 
&+(1-x_i(t))\sum_{j\in \mathcal{N}_i^{\rm in}} (\bar \beta_{ji}-l_{ji}x_j(t_\ell^{j}))x_j(t)\nonumber\\
=&-\{\ubar{\delta}_i+k_i(x_i(t) - e_i(t))\} x_i(t)\nonumber\\
&+(1-x_i(t))\sum_{j\in \mathcal{N}_i^{\rm in}} \left\{\bar \beta_{ji}-l_{ji}\left(x_j(t) - e_j(t)\right)\right\}x_j(t)\label{eq23}
\end{align}
%(\ref{eq9}) implies that the following constraints are always be satisfied
Moreover, due to the event-triggered condition \req{evcond}, it follows that $|e_i(t)|\le \sigma_i x_i(t)+\eta_i$, for all $t\ge 0$ and $i\in\mathcal{V}$. Hence, we obtain
%Therefore, we obtain 
\begin{align}\label{transformdxdt}
\dot{x}_{i}(t)\le&-\{\ubar{\delta}_i+k_i(x_i(t)-|e_i(t)|)\} x_i(t)\nonumber\\
&+(1-x_i(t))\sum_{j\in \mathcal{N}_i^{\rm in}}\{\bar \beta_{ji}-l_{ji}(x_j(t)-|e_j(t)|)\}x_j(t)\nonumber\\
\le&-\{\ubar{\delta}_i + k_i((1-\sigma_i)x_i(t) - \eta_i)\}x_i(t) \nonumber\\
&+(1-x_i(t))\sum_{j\in \mathcal{N}_i^{\rm in}}\{\bar \beta_{ji} - l_{ji}( (1-\sigma_{j})x_j(t) - \eta_j)\}x_j(t)\nonumber \\
= &-(\ubar{\delta}_i-k_i\eta_i)x_i(t)+\sum_{j\in \mathcal{N}^{\rm in}_i} (\bar \beta_{ji}+l_{ji}\eta_j)x_j(t), \nonumber \\
&-k_i(1-\sigma_i)x_i^2(t)-\sum_{j\in \mathcal{N}^{\rm in} _i} l_{ji}(1-\sigma_{j})x_j^2(t)\nonumber\\
&-x_i(t)\sum_{j\in \mathcal{N}_i^{\rm in}}(\bar \beta_{ji} + l_{ji}\eta_{j})x_j(t)
+ x_i(t)\sum_{j\in \mathcal{N}_i^{\rm in}} l_{ji}(1 - \sigma_j)x_j^2(t) %\label{xdotcom}
\end{align}
Note that for every $x (0) \in [0, 1]^n$, we have $x (t) \in [0,1]^n$ for all $t >0$ (see \rprop{invariantprop}). 
Hence, the last term in \req{transformdxdt} can be computed as
\begin{align}
x_i(t)\sum_{j\in \mathcal{N}_i^{\rm in}} l_{ji}(1 - \sigma_j)x_j^2(t)\le x_i(t)\sum_{j\in \mathcal{N}_i^{\rm in}} l_{ji}(1 - \sigma_j)x_j(t). 
\end{align}
%for all $x_i\in[0,1]$, so it follows that
Thus, \req{transformdxdt} becomes
%we obtain
\begin{align}\label{xeq2}
\dot{x}_{i}(t)\le 
& -(\ubar{\delta}_i-k_i\eta_i)x_i(t)+\sum_{j\in \mathcal{N}^{\rm in}_i} (\bar \beta_{ji}+l_{ji}\eta_j)x_j(t), \nonumber \\
&-k_i(1-\sigma_i)x_i^2(t)-\sum_{j\in \mathcal{N}^{\rm in} _i} l_{ji}(1-\sigma_{j})x_j^2(t)\nonumber\\
&-x_i(t)\sum_{j\in \mathcal{N}_i^{\rm in}}\{\bar \beta_{ji} - l_{ji} + l_{ji}(\sigma_j+\eta_{j})\}x_j(t),  
%\leq &-(\ubar{\delta}_i-k_i\eta_i)x_i(t)+\sum_{j\in %\mathcal{N}^{\rm in}_i} %(\bar \beta_{ji}+l_{ji}\eta_j)x_j(t), \nonumber \\
%&-k_i(1-\sigma_i)x_i^2(t)-\sum_{j\in \mathcal{N}^{\rm in} _i} %l_{ji}(1-\sigma_{j})x_j^2(t)\nonumber\\
%&-x_i(t)\sum_{j\in \mathcal{N}_i^{\rm in}} %l_{ji}(\sigma_j+\eta_{j})x_j(t),
\end{align}
%where we used $|e_i(t)|\le \sigma_i x_i(t)+\eta_i$, for all $t\ge 0$ and $i\in\mathcal{V}$ to obtain the second inequality. 
%Therefore, we obtain 
Using $\bar \beta_{ji} \ge l_{ji}$ and $1-\sigma_i <1$, we then obtain
\begin{align}
\dot{x}_{i}(t)\le&-(\ubar{\delta}_i-k_i\eta_i)x_i(t)+\sum_{j\in \mathcal{N}^{\rm in}_i} (\bar \beta_{ji}+l_{ji}\eta_j)x_j(t)\nonumber\\ 
&-k_i(1-\sigma_i)x_i^2(t)-\sum_{j\in \mathcal{N}^{\rm in} _i} l_{ji}(1-\sigma_{j})x_j^2(t)\nonumber\\
&-\frac{1}{2}x_i(t)\sum_{j\in \mathcal{N}^{\rm in} _i}  l_{ji}(1-\sigma_j)(\sigma_j+\eta_j)x_j(t), 
\label{xeq}
\end{align}
%where we used $1-\sigma_j <1$ from \req{eq23} to \req{xeq}. 
By collecting \req{xeq} for all $i \in \mathcal{V}$, we have 
\begin{align}
\dot{x}(t)\le&\left \{\bar{B}-\ubar{D}+(K+L)H\right\} {x}(t)-(K+L)(I_n-G)\tilde{ {x}}(t)\nonumber\\
&-\frac{1}{2}X(t)L(I_n-G)(G+H) {x}(t),
\label{xdoteq}
\end{align}
where $\tilde{ {x}}(t)=[x_1^2(t),\ldots,x_n^2(t)]^\mathsf{T}$ and $X(t) = {\rm diag}(x(t))$.  
The derivative of the Lyapunov function $V( {x}) = p^\mathsf{T} x$ is then given by 
\begin{align}
\cfrac{{\rm d}}{{\rm d}t} {V}( {x}(t)) &= p^\mathsf{T} \dot{x}(t) \notag \\
&\le  {r}^\mathsf{T} {x}(t)- {x}^\mathsf{T}S {x}(t) -\frac{1}{2} {x}^\mathsf{T}(t) PL(I_n-G)(G+H) {x}(t)\nonumber\\
&=  {r}^\mathsf{T} {x}(t)- {x}^\mathsf{T}(t)Q {x}(t),
\label{eq27}
\end{align}
so that ${\rm d}{V}( {x}(t))/{\rm d}t \leq {r}^\mathsf{T} {x}(t)- {x}^\mathsf{T}(t) Q {x}(t)$. 

Now, consider $\theta^* \in \mathbb{R}$ computed from \req{optimization}. 
From \req{optimization}, it follows that $x^\mathsf{T} Q x - (r+\epsilon \mathsf{1}_n)^\mathsf{T} {x} \leq 0 \implies p^\mathsf{T} x \leq \theta^*$ for all $x \in \mathbb{R}^n$. Moreover, since \req{optimization} is strictly convex and it corresponds to the maximization of the linear function (i.e., $p^\mathsf{T} x$) over the ellipsoidal set (i.e., $\mathcal{W}$), the optimal solution of \req{optimization} is unique and lies on the 
%since $Q$ is positive definite, the optimization problem \req{optimization} corresponds to the maximization of the linear function (i.e., $p^\mathsf{T} x$) over the ellipsoidal set (i.e., $\mathcal{W}$). Hence, the optimal solution of \req{optimization} is unique and lies on the 
boundary of $\mathcal{W}$ \footnote{For this clarification, see, e.g., the solution to Exercise 4.21(b) in \citep{boyd}, in which it can be verified that the optimal solution $x^\star$ lies on the boundary of the ellipsoidal set $(x -x_c)^T A (x -x_c) \leq 1$.}.
%strict equality of the constraint holds when $x$ equals the optimal solution (i.e., $(x^\star -x_c)^T A (x^\star -x_c) =1$, where $x^\star$ is the optimal solution to ).
%i.e., $x^{*\mathsf{T}}Qx^* - (r+\epsilon \mathsf{1})^\mathsf{T} x^* =0$, where $x^* \in \mathbb{R}^n$ is the optimal solution of $x$ in \req{optimization}. 
In other words, we have $x^{*\mathsf{T}}Qx^* - (r+\epsilon \mathsf{1}_n)^\mathsf{T} x^* =0$, where $x^*$ is the optimal solution of $x$ in \req{optimization} (i.e., $p^\mathsf{T} x^* = \theta^*$). Therefore, 
%\Longleftrightarrow  p^\mathsf{T} x = \theta^* $ and, therefore, 
%From \req{optimization}, it follows that 
%it follows that 
\begin{align}\label{thetaprop}
x^\mathsf{T} Q x - (r+\epsilon \mathsf{1}_n)^\mathsf{T} {x} < 0 \implies p^\mathsf{T} x < \theta^*,
\end{align}
for all $x \in \mathbb{R}^n$. 
Hence, by taking the contrapositive of \req{thetaprop}, we obtain $p^\mathsf{T} x \geq \theta^* \implies x^\mathsf{T} Q x - (r+\epsilon \mathsf{1}_n)^\mathsf{T} {x} \geq  0$ for all $x \in \mathbb{R}^n$. 
%\begin{align}
%p^\mathsf{T} x > \max \{0, \theta^*\} \implies {x}^\mathsf{T}Q {x}-{r}^\mathsf{T} {x} >0
%\end{align}for all $x \in [0, 1]^n$. 
Therefore, from \req{eq27}, we obtain 
\begin{align}
V( {x}(t)) \geq \theta^* \implies \cfrac{{\rm d}}{{\rm d}t} {V}( {x}(t)) &\leq  {r}^\mathsf{T} {x}(t)- {x}^\mathsf{T} (t)Q {x}(t) \notag \\
&\leq - \epsilon \mathsf{1}_n ^\mathsf{T} x = - \epsilon \|x\|_1 \label{laypunovdecrease}
\end{align}
%for all $t \geq 0$. 
Eq.~\req{laypunovdecrease} implies that the derivative of the Lyapunov function $V$ along the trajectory of the SIS model satisfies 
\begin{align}
   \frac{\mathrm{d}}{\mathrm{d}t} {V}( {x}(t)) \leq - \epsilon \|x\|_1, \ \forall x(t) \in \Lambda,  
\end{align}
where $\Lambda = \{x \in [0, 1]^n : V(x) \geq \theta^*\}$. Since $\dot{V}$ is negative in $\Lambda$, any state trajectory starting in $\Lambda$ converges to the set $\Omega = \{x \in [0, 1]^n : V(x) \leq \theta^*\}$ in finite time (see e.g., Section 4.8 in \citep{khalil}). Moreover, since $\dot{V}$ is negative in $\partial \Omega = \{x \in [0, 1]^n : V(x) = \theta^*\}$, it is shown that $\Omega$ is an invariant set, i.e., once the state enters $\Omega$, it remains therein for all future times. 
%on the boundary of $\Omega$, 
%Therefore, 
%\del{Eq. \req{laypunovdecrease} implies that, as long as $V( {x}(t))$ is greater than $\theta^*$, it ensures that $V({x}(t))$ is strictly decreasing.}
%\textcolor{black}{From , the derivative of the Lyapunov function $V$ along the trajectory of the SIS model satisfies}
%\begin{align}
%  \frac{\mathrm{d}}{\mathrm{d}t} {V}( {x}(t)) < 0, \ \forall x(t) \notin \Lambda, 
%s\end{align}
%\textcolor{black}{for all $\forall t \ge 0$, where $\Lambda = \{x : V(x) \le \theta^*\}$. 
%The inequality implies the set $\Lambda$ is positively invariant since the derivative $\mathrm{d}{V}(x(t))/\mathrm{d}t$ out of the set $\Lambda$ is negative and any solution starting in the set $\Lambda$, $x(0) \in [0, 1]^n$ s.t. $V(0) \le \theta^*$, remains in it for all $t \ge 0$. On the other hand, a trajectory starting out of the set $\Lambda$, i.e., $x(0) \in [0, 1]^n$ s.t. $V(x(0)) > \theta^*$, move in a direction of strictly decreasing the Lyapunov function $V(x(t))$ since the derivative $\mathrm{d}{V}(x(t))/\mathrm{d}t$ is negative out of the set $\Lambda$ for all $t \ge 0$.}
%for all $x \in [0, 1]^n$. 
%Hence, defining the level set as $\Lambda_{\theta^*} = \{ x \in [0, 1]^n  :  V(x) \leq \theta^*\}$, it follows that ${\rm d}{V}/{\rm d}t < 0$ for all $x \in[0, 1]^n \backslash \Lambda_{\theta^*}$. 
Therefore, for every $x(0) \in [0, 1]^n$, there exists $t' \geq 0$ such that  
%we obtain 
%for every $x \in [0, 1]^n$, the state trajectory asymptotically converges to the set $\Lambda_{\theta^*}$, i.e., 
\begin{align}\label{lyapunovconverge}
%\lim\sup_{t\rightarrow \infty} 
V(x(t)) \leq \theta^*, \ \forall t \geq t'. 
\end{align}
Moreover, since $p_m^*$ is defined by \req{pmstar}, we have 
\begin{align}
p_m^* {w}_m^\mathsf{T} {x}(t) &=p_m^*\sum _{i\in {\rm supp}( {w}_m)}x_i (t) \le \sum_{i\in {\rm supp}( {w}_m)}p_ix_i (t) \notag \\
&\le \sum_{i = 1}^n p_ix_i(t) = V( {x}(t)),
\label{eq30}
\end{align}
which implies that
\begin{align}\label{lyapunovconverge2}
%\limsup_{t\to \infty}
V( {x(t)})\le \theta^*, \ \forall t \geq t' \implies %\limsup_{t\to\infty}\ 
p_m^* {w}_m^T {x}(t)\le \theta^*, \ \forall t \geq t'. 
\end{align}
 Hence, if \req{thetastar} holds for all $m \in \{1, \ldots, M\}$, we have ${w}_m^T {x}(t)\le \bar{d}_m$, for all $t \geq t'$ and $m \in \{1, \ldots, M\}$. Therefore, for any $x(0) \in [0, 1]^n$, the control objective \req{objective} is achieved by applying the event-triggered controller \req{ui}, \req{vij}. 
 %for all ${x}(0)\in [0,1]^n$, which means that the control objective \req{objective} is achieved. 
%Therefore, from \req{optimization}, it follows that 
%\begin{equation}
%V( {x}) > \theta^* \implies \dot{V}( {x})<0,
%\end{equation}
\end{proof}

\begin{comment}
The following {corollary} is an immediate consequence of \rthm{mainresult}. 
\begin{mycor}\label{originstable}
\normalfont
Consider the SIS model \req{controldynamics} and the event-triggered controller \req{ui}, \req{vij}. Assume that the control and the event-triggering gains satisfying \req{conditioneventtrig} and \req{constraintcontrolparameters} are chosen such that ${r} < 0$, where $r$ is defined in \req{rvec}. Then, the origin is asymptotically stable by applying the event-triggered controller, i.e., for every $x(0) \in [0,1]^n$, $\lim_{t\to \infty} x(t) = 0$. \qedwhite 
\end{mycor}
\begin{proof}
Suppose that the set of control and the event-triggering gains are chosen such that ${r} < 0$. Then, it follows that $x^\mathsf{T} Qx - r^\mathsf{T} x >0$ for every $x \in \mathbb{R}^n\backslash \{0\}$ since $Q$ is positive definite. This implies that the origin $x = 0$ is only a feasible point that satisfies the constraint \req{optimization}, and thus the solution to \req{optimization} is $\theta^* = 0$. 
From \req{laypunovdecrease}, this follows that the Lyapunov function is strictly decreasing except the origin, i.e., ${\rm d}{V}( {x}(t))/{\rm d}t<0$ for every $x(t) \in [0, 1]^n \backslash \{0\}$, which shows that the origin $x = 0$ is asymptotically stable. 
\end{proof}
\end{comment}
%As shown in \rthm{mainresult}, 

\section{Event-triggered controller synthesis}\label{eventtrigdesignsec}

%\masaki{節の名前について．前節がStability Analysisでスッキリなのにくらべてこの節はなんだか洗練されていないと思います．"Controller Synthesis"とかじゃダメですか？}
%In the previous section, we investigate a sufficient condition that, \textit{given} the Lyapunov, control and event-triggering gains , the control objective can be achieved by applying the event-triggered controller. As shown in \rthm{mainresult}, the optimization problem \req{optimization} becomes \textit{convex} if all parameters (Lyapunov, control and event-triggering gains ) are \textit{given}, i.e., the parameters $Q$ and $r$ in \req{optimization} are all fixed. However, if we consider \textit{designing} the control and the event-triggering gains such that the control objective can be achieved, the problem \req{optimization} becomes non-convex, since the parameters $Q$ and $r$ are now regarded as the variables to be designed. 
%\del{In this section, we investigate a way of designing the control and the event-triggering gains , such that the control objective is achieved.} 
{In this section, we investigate an event-triggered controller design.} %\masaki{もうスッキリ言ってしまって良いと思ったので．} 
As shown in \rthm{mainresult}, the control objective \req{objective} is achieved if the control and the event-triggering gains {satisfy the inequality} \req{thetastar} for all $m\in \{1, \ldots, M\}$ {, which we can efficiently check by convex optimization}. 
%\del{As previously stated, if these parameters  are given (i.e., $r$ and $Q$ are given in \req{optimization}), the conditions \req{thetastar} can be checked easily, since the optimization problem \req{optimization} becomes convex.} \masaki{reduces to a convex optimization problem とか？become convex は論理的に変だと思います．すべての最適化問題は is convex か is not convexのどちらかのはずなので} 
However, {it is not necessarily easy to directly use the inequality \req{thetastar} for} \textit{designing} the control and the event-triggering gains, {because the vector} $r$ and {the matrix} $Q$ used to define the set ${\cal W}$ in {the optimization problem} \req{optimization} {contain the parameters to be designed}. %(see \req{svec}--\req{calw}). 
{In order to overcome this difficulty, in} this section, we {present a tractable and numerically efficient method for designing the control and event-triggering gains via convex relaxation techniques for the conditions required in \rthm{mainresult}, such that both the control and the event-triggering gains are designed in polynomial time}. 

%\masaki{以下はよんでません．}
%several techniques to relax the required conditions in \rthm{mainresult}, such that the control and the event-triggering gains can be found in a tractable way. 
Specifically, we propose an \textit{emulation-based} approach to the design of the control and the event-triggering gains. 
%two approaches towards the design of the control and the event-triggering gains ; an \textit{emulation-based} approach, and a \textit{co-design} approach. In the following sub-sections, we describe the details of these approaches. 
%\subsection{Emulation-based approach}
The emulation-based approach is the well-known technique to design the event-triggered controller (see, e.g., \citep{heemels2012a,heemels2013a}), and basically it consists of the two steps. First, we find the set of the control gains  under the assumption that the \textit{continuous-time} controller is implemented. Second, using the control gains  obtained by the first step, we then design the event-triggering gains, such that the control objective is achieved. As will be shown below, both the former and the latter problems can be formulated by geometric programmings \citep{boyd}, meaning that the control and the event-triggering gains can be found efficiently in polynomial time. 

\subsection{Designing control gains}\label{designcontrolgainsec}
We start by designing the control gains  $k_i$, $l_{ij}$ for all $i \in \mathcal{V}$ and $j\in {\cal N}^{\rm out}_i$. 
As mentioned above, in the emulation-based approach, the control gains  are designed under the assumption that the continuous-time controller is implemented; that is, \req{ui} and \req{vij} are replaced by 
\begin{align}
u_i (t) &= k_i x_i (t), \label{uitimecon} \\
v_{ij} (t) &= l_{ij} x_i (t), \ j \in \mathcal{N}^{\rm out} _i, \label{vijtimecon}
\end{align}
for all $t \geq 0$. Moreover, define the constants $\tilde{r}_{c, i}$, $c_{1,i}$ for all $i\in {\cal V}$ by 
\begin{align}
%c_{1, i}&= \sum_{j\in \mathcal{N}^{\rm out}_i} p_j\bar \beta_{ij}-p_i\ubar{\delta}_i, \label{eq32}\\
\tilde{r}_{c, i} &= -p_i\ubar{\delta}_i + \sum_{j\in \mathcal{N}^{\rm out}_i} p_j\bar \beta_{ij}, \label{rtildeci} \\
c_{1,i}&= p_i\bar{k}_i+\sum_{j\in \mathcal{N}^{\rm out}_i}p_j\bar{l}_{ij} >0, \label{ctwoi}
%\mathcal{C}&= \{1,\ldots,n~| c_i^{(1)}\ge 0,~i=1,\ldots, n\}, %\\  \label{eq35}
%c_{m,3}&= 2p_m^* \bar{d}_m - \sum_{i\notin \mathcal{C}}\frac{p_ic_i^{(1)}}{c_i^{(2)}}, \label{eq34}
\end{align}
From \req{rtildeci}, we obtain $\tilde{r}^\mathsf{T} _c = p^\mathsf{T} \left (\bar{B}-\ubar{D}\right)$ with $\tilde{r} _c = [\tilde{r} _{c,1}, \ldots, \tilde{r} _{c,n}]^\mathsf{T}$. 
In addition, define the set $\mathcal{C} \subseteq \mathcal{V}$ and the constants $c_{2,m}$ for all $m \in \{1, \ldots, M\}$ by 
\begin{align}
&\mathcal{C}= \{i \in \mathcal{V}\ :\ \tilde{r}_{c, i}\ge 0\}, \label{positivec}\\ 
&c_{2,m}= 2p_m^* \bar{d}_m - \sum_{i\notin \mathcal{C}}\cfrac{p_i \tilde{r}_{c, i}}{c_{1,i}} >0. %\ m \in \{1, \ldots, M\}. 
\end{align}
%for all $m \in \{1, \ldots, M\}$. %Note that $c_{2,m}>0$ for all $m\in\{1,\ldots,M\}$.  
%As shown in \req{positivec}, ${\cal C}$ is the set of indices $i\in \mathcal{V}$, for which $\tilde{r}_{c, i}$ take the non-negative values. 

The following proposition shows that the set of the control gains  achieving the control objective \req{objective} under the continuous-time controller can be found by solving a geometric programming problem. 

%The following proposition shows that, we can find the control gains  $k_i$ and $l_{ij}$ such that the control objective \req{objective} is achieved by applying the continuous-time controller \req{uitimecon} and \req{vijtimecon} for the SIS model \req{controldynamics}. 
%that result in achieving the control objective can be found by solving the geometric programming: %by solving the geometric programming, we can find the set of the control gains  such that the control objective \req{objective} is achieved under the continuous-time controller.
\begin{myprop}\label{gpproposition}
\normalfont
Consider the SIS model \req{controldynamics}, continuous-time controller \req{uitimecon}, \req{vijtimecon}, and the control objective \req{objective}. 
Moreover, let $\tilde{k}^*_i, \tilde{l}^*_{ij}, \tilde{s}^*_{c,i} >0$ for all $i \in \mathcal{V}$, $j \in {\cal N}^{\rm out} _i$ and $\epsilon^*_1, \epsilon^*_2, \epsilon^*_3, {\xi}^*_{c} >0$ denote the optimal solution of $\tilde{k}_i, \tilde{l}_{ij}, \tilde{s}_{c,i} >0$ for all $i \in \mathcal{V}$, $j \in {\cal N}^{\rm out} _i$ and $\epsilon_1, \epsilon_2, \epsilon_3, {\xi}_{c}>0$, in the following geometric programming: 
%\tilde{k}_i, \tilde{l}_{ij}, \tilde{s}_{c,i},i \in \mathcal{V},j \in {\cal N}^{\rm out} _i, {\xi}_{c}
%\substack{\tilde{k}_i, \tilde{s}_{c,i} >0,\ {i \in \mathcal{V}} \\ \tilde{l}_{ij} >0,\  {i\in \mathcal{V}, j \in {\cal N}^{\rm out} _i}, \\ {\xi}_{c} >0}
\begin{align}
  \underset{Z_c >0}{\rm{minimize}}&\ \ g_c \left(Z_c \right), \notag \\ 
  %&~~~ f(K,L), \label{eq36}\\
 \! \rm{subject~to}\ \ \ 
  &\tilde{s}_{c,i}+ p_i\tilde{k}_i + \sum_{j\in \mathcal{N}^{\rm out}_i}p_j\tilde{l}_{ij}\le c_{1,i},\ \forall i \in \mathcal{V}\label{gpthree}\\
  &\tilde{k}_i + \epsilon_1 \leq \bar{k}_i,\ \forall i \in \mathcal{V} \label{gpfour} \\ 
  &\tilde{l}_{ij} + \epsilon_2 \leq \bar{l}_{ij},\ \forall i \in \mathcal{V}, \forall j \in {\cal N}^{\rm out} _i, \label{gpfive}\\
  &{\xi}_c ^{\frac{1}{2}}+\sum_{i\in \mathcal{C}}{p_i(\tilde{r}_{c,i}+\epsilon_3)}{\tilde{s}^{-1}_{c,i}}\le c_{2,m},\ \forall m\in \{1,\ldots, M\}\label{gpone}\\
  &\left(\sum_{i \in \mathcal{V}}{p_i^2}{\tilde{s}^{-1}_{c,i}}\right)\left(\sum_{i \in \mathcal{V}}{(\tilde{r}_{c,i}+\epsilon_3)^{2}}{\tilde{s}^{-1}_{c,i}}\right) \le {\xi}_c,\label{gptwo}%\\
 % &~~~i,j= 1,\ldots, n,~~~m = 1,\ldots,M,\nonumber
\end{align}
where $Z_c$ is the vector that collects all the decision variables in the optimization problem, i.e., $Z_c = \left[\tilde{k}_i, \tilde{l}_{ij}, \tilde{s}_{c,i},i \in \mathcal{V},j \in {\cal N}^{\rm out} _i, \epsilon_1,\epsilon_2,\epsilon_3, {\xi}_{c}\right]^\mathsf{T}$ and $g_c(\cdot)$ is a given posynomial function. %and $\epsilon >0$ is an any positive constant. 
%Assume that the above problem is feasible, and let $\tilde{k}^*_i, \tilde{l}^*_{ij}, {\xi}^*_{c}, \tilde{s}^*_{c,i} >0$ for all $i \in \mathcal{V}$ and $j \in {\cal N}^{\rm out} _i$ be an any (feasible) solution to the above problem. 
Then, the control objective is achieved by applying the continuous-time controller, in which the control gains $k_i$, $l_{ij}$ are given by 
\begin{align}\label{kilij}
k_i = \bar{k}_i-\tilde{k}^*_i,\ \ l_{ij} =  \bar{l}_{ij} -\tilde{l}^* _{ij} 
\end{align}
for all $i \in \mathcal{V}$ and $j \in {\cal N}^{\rm out} _i$. 
\qed
\end{myprop}
%\begin{proof}
\rprop{gpproposition} is shown by providing sufficient conditions for the control objective to be achieved under the continuous-time controller (see \rlem{emulationlemma} in Appendix~A), and then translate the conditions into the posynomial constraints as shown in \req{gpthree}--\req{gptwo}. 
%is shown as follows: (i) we provide sufficient conditions for the control objective to be achieved under the continuous-time controller (see \rlem{emulationlemma}), (ii) based on the conditions derived in (i), we then translate them into the posynomial constraints as shown in \req{gpthree}--\req{gptwo}. 
For the detailed proof, see Appendix~A. 
\begin{myrem}[On the selection of the cost function $g_c$]
\normalfont 
\textcolor{black}{For example, one could select the (posynomical) cost function $g_c(\cdot)$ as follows: 
\begin{align}\label{controlparamcostfunc}
g_c\left(Z_c \right) = \sum_{i\in \mathcal{V}} \frac{w_{k,i}}{\tilde{k}_i} + \sum_{i\in \mathcal{V}} \sum_{j \in {\cal N}^{\rm out} _{i} } \frac{w_{l,ij}}{\tilde{l} _{ij}},  
\end{align}
where $w_{k,i}, w_{l,ij}>0$ for all $i \in \mathcal{V}, j \in \mathcal{N}^{\rm out} _i$ are given weights. Note that $\tilde{k}_i$ and $\tilde{l}_{ij}$ for all $i\in\mathcal{V}$ and $j\in\mathcal{N}_i^{\rm out}$ are the variables in the optimization problem satisfying  $\tilde{k}_i = \bar{k}_i - k_i$ and $\tilde{l}_{ij} = \bar{l}_{ij} - l_{ij}$ (see Appendix~A). 
Moreover, $\bar{k}_i$, $\bar{l}_{ij}$ for all $i\in\mathcal{V}$ and $j\in\mathcal{N}_i^{\rm out}$ are the constants that represent the upper bounds of the control gains (see \req{constraintcontrolparameters}). 
Hence, reducing $k_i$ (resp. ${l}_{ij}$) implies to reduce the cost of $1/{\tilde{k}_i}$ (resp. $1/{\tilde{l}_{ij}}$). Therefore, minimizing \req{controlparamcostfunc} subject to the constraints \req{gpthree}-\req{gptwo} aims at obtaining small control gains while achieving the control objective.} \qedwhite 
\end{myrem}

\subsection{Designing event-triggering gains} 
Let us now design the event-triggering gains, i.e., $\sigma_i \in (0, 1), \eta_i \in (0,1)$ for all $i\in {\cal V}$. Consider the event-triggered controller \req{ui}, \req{vij}, in which the triggering time instants $t^i _0, t^i _1, t^i _2, \ldots$ are given according to \req{triggeringtime}. 
Fix the control gains by $k_i = k^* _i \in (0, \bar{k}_{ij}]$, $l_{ij} = l^* _{ij} \in (0, \bar{l}_{ij}]$ for all $i \in \mathcal{V}$, $j \in {\cal N}_i^{\rm out}$, where $k^* _i$, $l^* _{ij}$ ($i \in \mathcal{V}$, $j \in {\cal N}_i^{\rm out}$) are the optimal control gains that are designed by solving the geometric programming problem proposed in \rprop{gpproposition}. 
%Suppose that the control gains  are designed by solving the geometric programming according to \rprop{gpproposition}, and we let the corresponding control gains  by $k_i = k^* _i$ and $l_{ij} = l^* _{ij}$ for all $i \in \mathcal{V}$ and $j \in {\cal N}_i^{\rm out}$.
%In addition, let $k^*=[k^* _1, \ldots, k^*_n]^\mathsf{T}$, $K^* ={\rm diag}(k^*)$, and $L^*=[l^*_{ji}]_{i,j}$. 
%To design the event-triggering gains , 
%To design the event-triggering gains , we define the notations as follows. First, let $\tilde{s}_e = [\tilde{s}_{e,1}, \ldots, \tilde{s}_{e,n}]^\mathsf{T}$ and $\tilde{r}_e =[\tilde{r}_{e,1}, \ldots, \tilde{r}_{e, n}]^\mathsf{T} \in \mathbb{R}^n$ be the vectors %satisfying 
%\begin{align}
%\tilde{s}_e^\mathsf{T} &\leq {p}^\mathsf{T}(K^*+L^*)(I_n-G) %\label{setilde}\\
%\tilde{r}^\mathsf{T}_e & \geq %{p}^\mathsf{T}\{\bar{B}-\ubar{D}+(K^*+L^*)H\}. \label{retilde}
%\end{align}
%Note that both $\tilde{s}_e$ and $\tilde{r}_e$ will be the variables in the (event-triggered) design problem, since they both involve the event-triggering gains $G$ and $H$. 
%Moreover, let 
Moreover, define the constants $c_{3, i}$ for all $i \in {\cal V}$ by 
\begin{align}\label{const3}
  c_{3,i} =p_i k^*_i+\sum_{j\in \mathcal{N}^{\rm out} _i}p_j l^*_{ij} >0.
\end{align}
For technical reasons, we make the following assumption: 
\begin{myas}\label{assumption}
\normalfont
For all $i \notin {\cal C}$, $c_{3, i} + \tilde{r}_{c, i} >0$. 
%$\sum_{j\in \mathcal{N}^{\rm out}_i} p_j (\bar \beta_{ij}+l^*_{ij}) -p_i\ubar{\delta}_i$
\qedwhite 
\end{myas}
%From \req{rtilde}, \req{positivec} and \req{const3}, 
Recall that $\tilde{r}_{c,i}$ and ${\cal C}$ are defined in \req{rtildeci} and \req{positivec}, respectively. Hence, \ras{assumption} implies that, for all $i \in {\cal V}$ satisfying $\tilde{r}_{c,i} <0$, the following condition is satisfied: 
\begin{align}
& c_{3, i} + \tilde{r}_{c, i} = p_i(k^* _i - \ubar{\delta}_i) + \sum_{j\in \mathcal{N}^{\rm out}_i} p_j (\bar \beta_{ij}+l^*_{ij}) >0.  \label{ascondition}
\end{align}
Hence, \req{ascondition} implies that the (optimal) control gains  $k^* _i$, $l^* _{ij}$, $i \notin {\cal C}$, $j \in {\cal N}^{\rm out} _i$ should be chosen large enough such that $c_{3, i} + \tilde{r}_{c, i}$ is positive.  
%As will be seen below, \ras{assumption} is utilized to guarantee an existence of $\eta_i <1$ for all $i \notin {\cal C}$, such that that control objective is achieved. 
\begin{myrem}
\normalfont
The condition required in \ras{assumption} can be indeed satisfied by imposing an additional constraint in the geometric programming presented in \rprop{gpproposition}. From \req{ascondition}, the control gains  $k _i$, $l_{ij}$ for all $i \notin {\cal C}$ and $j \in {\cal N}^{\rm out} _i$ must be chosen such that $p_i(k _i - \ubar{\delta}_i) + \sum_{j\in \mathcal{N}^{\rm out}_i} p_j (\bar \beta_{ij}+l_{ij}) >0$ for all $i \notin {\cal C}$. This leads to the following posynomial constraint: 
\begin{align}
p_i\tilde{k}_i + p_i \ubar{\delta}_i + \sum_{j\in \mathcal{N}^{\rm out}_i}p_j\tilde{l}_{ij} + \epsilon \le c_{1,i} + \sum_{j\in \mathcal{N}^{\rm out}_i}p_j \bar \beta_{ij}, \label{ascondition2}
\end{align}
where $\epsilon >0$ is a given arbitrary small positive constant. Hence, $c_{3, i} + \tilde{r}_{c, i} >0$ is achieved by additionally imposing \req{ascondition2} for all $i \notin {\cal C}$ in the geometric programming provided in \rprop{gpproposition}. \qedwhite 
\end{myrem}

The following proposition shows that the event-triggering gains achieving the control objective can be found by solving the geometric programming problem:

%The following proposition shows that, by solving the geometric programming \onoue{geometric programing ? feasible problem ?}, we can find the event-triggering gains $\sigma_i$ and $\eta_i$ such that the control objective \req{objective} is achieved by applying the event-triggered controller \req{ui} and \req{vij} for the SIS model \req{controldynamics}. 

\begin{myprop}\label{gppropositione}
\normalfont
Consider the SIS model \req{controldynamics}, event-triggered controller \req{ui}, \req{vij}, and the control objective \req{objective}. 
%Let $k^* _i \in (0, \bar{k}_{ij}]$ and $l^* _{ij} \in (0, \bar{l}_{ij}]$ for all $i \in \mathcal{V}$ and $j \in {\cal N}_i^{\rm out}$ denote the optimal control gains  designed by solving the geometric programming in \rprop{gpproposition}.
%and suppose that the set of control gains  $k^* _i$ and $l^* _{ij}$ for all $i \in \mathcal{V}$ and $j \in {\cal N}_i^{\rm out}$ are designed. 
%Let $k^* _i \in (0, \bar{k}_{ij}]$ and $l^* _{ij} \in (0, \bar{l}_{ij}]$ for all $i \in \mathcal{V}$ and $j \in {\cal N}_i^{\rm out}$ be the optimal control gains  designed by solving the geometric programming provided in \rprop{gpproposition}. 
%Given the control gains  $k^* _i$ and $l^* _{ij}$ for all $i \in \mathcal{V}$ and $j \in {\cal N}_i^{\rm out}$ and under \ras{assumption}, 
Let \ras{assumption} hold, and %and fix the control gains  by $k_i = k^* _i$, $l _{ij} = l^* _{ij}$ for all $i \in \mathcal{V}$ and $j \in {\cal N}_i^{\rm out}$. 
let $\tilde{\sigma}^* _i$, $\tilde{s}^* _{e,i}$, $\eta^* _i$, $\tilde{r}^* _{e,i} >0$ for all $i \in \mathcal{V}$ and $\epsilon^*_1, \epsilon^*_2, \epsilon^*_3, {\xi}^* _{e}>0$ be the optimal solution of $\tilde{\sigma} _i$, $\tilde{s} _{e,i}$, $\eta _i$, $\tilde{r} _{e,i} > 0$ for all $i \in \mathcal{V}$ and $\epsilon_1, \epsilon_2, \epsilon_3, {\xi} _{e} > 0$, in the following geometric programming: 
%\left(\tilde{\sigma} _i, \tilde{s} _{e,i}, {i \in \mathcal{V}}, \eta _i, \tilde{r} _{e,i}, {i \in {\cal C}}, \xi_e \right)
%\substack{\tilde{\sigma} _i, \tilde{s} _{e,i} >0,\ {i \in \mathcal{V}} \\ \eta _i, \tilde{r} _{e,i} >0,\ {i \in {\cal C}}\\ \xi_e >0 }
\begin{align}
\underset{Z_e >0}{\rm{minimize}}&\ \ g_e (Z_e), \notag \\ 
{\rm subject\ to}\ \ &{\tilde{s}_{e,i}}{\tilde{\sigma}^{-1}_i}  \le c_{3,i}, \ \forall i\in \mathcal{V}
\label{propositionconstone}\\
&\tilde{r}_{c,i}\tilde{r}^{-1}_{e,i} + c_{3,i} \eta_i \tilde{r}^{-1}_{e,i} \le 1, \ \forall i\in {\cal C}, \label{propositionconsttwo}\\
&\tilde{\sigma}_i + \epsilon_1 \leq 1,\ \forall i\in \mathcal{V} \label{propositionconstthree}\\
&\eta_i  + \epsilon_2 \leq 1,\ \forall i\in {\cal C}, \label{eq59} \\
&{\xi}^{\frac{1}{2}}_{e}+\sum_{i \in \mathcal{V}} {p_i (\tilde{r}_{e,i}+\epsilon_3)}{\tilde{s}^{-1}_{e,i}}\le 2p_m^*\bar{d}_m, \ \forall m = \{1,\ldots,M\}, 
\label{propositionconstfour}\\
&\!\!\left(\sum_{i \in \mathcal{V}}{p_i^2}{\tilde{s}^{-1}_{e,i}}\right)\left(\sum_{i \in \mathcal{V}}{(\tilde{r}_{e,i}+\epsilon_3)^2}{\tilde{s}^{-1}_{e,i}}\right) \le {\xi}_{e}
\label{propositionconstfive}
\end{align}
where $Z_e$ is the vector that collects all the decision variables in the optimization problem, i.e., $Z_e = \left[ \tilde{\sigma} _i, \tilde{s} _{e,i}, \eta _i, \tilde{r} _{e,i}, {i \in \mathcal{V}},\epsilon_1, \epsilon_2, \epsilon_3,\xi_e\right]^\mathsf{T}$ and $g_e (\cdot)$ is a given posynomial function. %and $\epsilon >0$ is an any small positive constant. 
Then, the control objective is achieved by applying the event-triggered controller \req{ui}, \req{vij}, in which the event-triggering gains $\sigma_i, \eta_i$ are given by
\begin{align}
\sigma_i &= 1- \tilde{\sigma}^*_i, \ \forall i \in \mathcal{V}, \label{sigmai} \\
\eta_i &= 
\begin{cases}
\eta^* _i, \ \ \ \ \ \ \ \forall i \in {\cal C}, \\%\label{etainc} \\ 
\cfrac{{-\tilde{r}_{c,i}}}{c_{3, i}}\ \ \ \ \forall i \notin {\cal C}. \label{etanotinc}
\end{cases}
\end{align}
\qed
\end{myprop}
%Note that \req{propositionconstone}--\req{propositionconstfive} are the posynomial constraints. %and \req{gpfour}, \req{gpfive} are the monomial constraints. Hence, the event-triggering gains $\sigma_i$ and $\eta_i$ can be found by solving the geometric programming \onoue{geometric programing ? feasible problem ?}. 

%\begin{proof}
%See Appendix~B. 
%\end{proof}
\rprop{gppropositione} is proven by modifying the conditions required in \rthm{mainresult} (see \rlem{emulationevlemma} in Appendix~B), so that the conditions required to achieve the control objective can be translated into the posynomial constraints as shown in \req{propositionconstone}--\req{propositionconstfive}. 
%the event-triggering gains can be found via a geometric programming. 
For the detailed proof, see Appendix~B. 
\begin{myrem}[On the selection of the cost function $g_e$]
\normalfont 
\textcolor{black}{
For example, one could select the cost function $g_e(\cdot)$ as follows: 
\begin{align}\label{eventtriggerparamcostfunc}
g_e \left( Z_e \right) = \sum_{i\in \mathcal{V}} w_{\sigma, i} \tilde{\sigma}_i + \sum_{i\in \cal{C}} \frac{w_{\eta,i}}{\eta_i},
\end{align}
where $w_{\sigma, i}, w_{\eta,i}>0$ for all $i \in \mathcal{V}$ are given weight parameters. Note that $\tilde{\sigma} _i$ and $\eta_{i}$ for all $i\in{\cal V}$ are the variables satisfying $\tilde{\sigma} _i = 1- \sigma_i$ (see Appendix~B), and that $\sigma_i, \eta_{i}$ for all $i\in{\cal V}$ are the event-triggering gains. 
Hence, increasing $\sigma_i$ (resp. $\eta_{i}$) implies to reduce the cost of $\tilde{\sigma} _i$ (resp. $\eta_{i}^{-1}$). From \req{triggeringtime}, increasing $\sigma_i$ and $\eta_{i}$ allow us to reduce the number of the control updates. Therefore, minimizing \req{eventtriggerparamcostfunc} subject to the constraints \req{propositionconstone}--\req{propositionconstfive} implies to obtain large event-triggering gains so as to reduce the number of the control updates while achieving the control objective.} \qedwhite 
\end{myrem}
%\onoue{$Z_c$, $Z_e$なのですが, the set of decision variables $Z_c =\{\xi_{c}\}\cup \{\tilde{k}_i, \tilde{l}_{ij}\tilde{s}_{c,i}\}_{i\in\mathcal{V}, j\in\mathcal{N}_i^{\rm out}}$, $Z_e = \{\xi_e\}\cup \{\tilde{\sigma} _i,\tilde{s} _{e,i}\}_{i\in\mathcal{V}} \cup \{\eta _i, \tilde{r} _{e,i}\}_{i\in\mathcal{C}}$ではだめでしょうか?}
\textcolor{black}{
\section{Some discussions on the proposed approach} 
In this section, we provide some discussions on the proposed approach. In \rsec{designlyapunovsec}, we provide a way of how to design the Lyapunov parameter $p$. 
In \rsec{conservativenesssec}, we discuss a conservativeness of the geometric programming problems in \rprop{gpproposition} and 3 with respect to the condition derived in Theorem~1. 
\subsection{On designing the Lyapunov parameter} \label{designlyapunovsec}
Note that the Lyapunov parameter $p \in \mathbb{R}^n _{>0}$ should be \textit{given} in \rprop{gpproposition} (and \rprop{gppropositione}), which means that $p$ must be chosen \textit{apriori} before designing the control and the event-triggering gains. For example, one could choose $p = [p_1, \ldots, p_n]^\mathsf{T}$ by solving the following linear program: 
\begin{align}\label{designlyapunov}
    \underset{p >0,\  \|p\|_1 = c_p}{\rm minimize} \ \ \sum_{i \in \mathcal{V}}\left (-p_i\ubar{\delta}_i + \sum_{j\in \mathcal{N}^{\rm out}_i} p_j\bar{\beta}_{ij}\right ), 
\end{align}
where $c_p>0$ is a given positive constant. Since $\tilde{r}_{c, i} = -p_i \underline{\delta}_i + 
\sum_{j\in \mathcal{N}^{\rm out}_i} p_j \overline{\beta}_{ij}$
for all $i \in {\cal V}$ (see \req{rtildeci}), the optimization problem \req{designlyapunov} aims at finding $p$ such that $\sum_{i \in \mathcal{V}} \tilde{r}_{c, i}$ is minimized. Moreover, recall that the vector $\tilde{r}_{c} = [\tilde{r}_{c, 1}, \ldots, \tilde{r}_{c, n}]$ has been utilized in the derivative of the Lyapunov function under the continuous time controller: 
\begin{align}
 \cfrac{{\rm d}}{{\rm d}t} {V}({x}) \leq \tilde{r}^\mathsf{T} _c {x}- {x}^\mathsf{T}\tilde{S}_c {x}
\end{align}
(see \req{upperboundcontinuous} in the Appendix). Hence, intuitively, if we have smaller components of $\tilde{r}_{c, i}$, $i \in {\cal V}$, then the term $\tilde{r}^\mathsf{T} _c {x}$ becomes smaller and so we can obtain a larger domain of $x$ for which ${\rm d}{V}/{\rm d}t$ is ensured to be negative: $\{x\in [0, 1]^n : \tilde{r}^\mathsf{T} _c {x}- {x}^\mathsf{T}\tilde{S}_c {x} < 0\}$. Thus, designing $p$ such that $\tilde{r}_{c, i}$ becomes small may have the potential to enlarge the domain of attraction. Note that it is indeed difficult to take the matrix $\tilde{S}_c$ into account for designing $p$, since $\tilde{S}_c$ must satisfy the constraint involving the control gains (on the other hand, $\tilde{r}_{c, i}$ does not depend on the control gains).}
\textcolor{black}{In \req{designlyapunov}, the linear constraint $\|p\|_1 = p_1 + p_2 + \cdots p_n = c_p$ has been utilized to normalize the Lyapunov parameter, so that the sum of all the components of $p$ equals $c_p$. In essence, this avoids the case where the optimal solution of $p$ becomes extremely close to zero. 
%This avoids the case where the optimal solution of $p$ becomes extremely close to zero. 
%avoid the case where the optimal solution of $p$ cannot be found or extremely close to zero. For example, if the constraints $\|p\|_1 = c_p$ \textit{were} not given and $\bar{\beta}_{ii} < \ubar{\delta}_i$ for some $i \in \mathcal{V}$, the optimal finite solution of $p$ would not be found, since the term $p_i(\bar{\beta}_{ij} - \ubar{\delta}_i)$ is negative for any $p_i >0$ and thus $p_i \rightarrow \infty$ implies $\sum_{i \in \mathcal{V}} (-p_i\ubar{\delta}_i + \sum_{j\in \mathcal{N}^{\rm out}_i} p_j\bar{\beta}_{ij}) \rightarrow - \infty$.
%$\sum_{i' \in \mathcal{V}} \sum_{j\in \mathcal{N}^{\rm out}_{i'}} p_j\bar{\beta}_{ij}-p_i\ubar{\delta}_i = p_i (\bar{\beta}_{ij} - \ubar{\delta}_i) + \sum_{i \in \mathcal{V}} \sum_{j\in \mathcal{N}^{\rm out}_i\backslash \{i\}} p_j\bar{\beta}_{ij}-p_i\ubar{\delta}_i$. 
%In addition, 
%
For example, if $\bar{\beta}_{ii} > \ubar{\delta}_i$ for all $i\in \mathcal{N}$, then the cost in \req{designlyapunov} is positive for all $p>0$. Hence, if $\|p\|_1 = c_p$ \textit{were} not given, we could then obtain the optimal solution as $p^* \approx 0$, since it tries to make the cost in \req{designlyapunov} as close as possible to $0$, i.e., $\sum_{i \in \mathcal{V}} (-p_i\ubar{\delta}_i + \sum_{j\in \mathcal{N}^{\rm out}_i} p_j\bar{\beta}_{ij}) \rightarrow 0$ as $p \rightarrow 0$.}

\textcolor{black}{
One might wonder how to select $c_p$ in \req{designlyapunov}. 
Here, \textit{without loss of generality}, we can set 
$c_p = 1$; how we select $c_p$ does not affect the domain of the control gains $k_i, l_{ij}, i\in \mathcal{V}, j \in \mathcal{N}^{\rm out}_i$ (resp. the event-triggering gains $\sigma_i, \eta_i$, $i \in \mathcal{V}$) for which the geometric programming problem in Proposition~2 (resp. Proposition~3) is feasible. Specifically, we have the following result:}
\textcolor{black}{
\begin{myprop}\label{feasibilityresult}
\normalfont
Let $p^{(1)} = [p^{(1)}_{1}, p^{(1)}_{2},\ldots, p^{(1)}_n]$ and $p^{(2)} = [p^{(2)}_{1}, p^{(2)}_{2},\ldots, p^{(2)}_n]$ denote the optimal solution of \req{designlyapunov} with $c_p= \gamma^{(1)}$ and $c_p = \gamma^{(2)}$, respectively, where $\gamma^{(1)}, \gamma^{(2)} >0$ with $\gamma^{(1)} \neq \gamma^{(2)}$ are any positive constants. Let {(P.1)} and {(P.2)} (resp. (Q.1) and (Q.2)) denote the geometric programming problem in Proposition~2 (resp. Proposition~3) with the Lyapunov parameter being given by $p= p^{(1)}$ and $p= p^{(2)}$, respectively. 
%Moreover, let \textbf{(Q.1)} and \textbf{(Q.2)} denote the geometric programming problem in Proposition~3 with the Lyapunov parameter being given by $p= p^{(1)}$ and $p= p^{(2)}$, respectively. 
Then, it follows that the feasibility of {(P.1)} (resp. {(Q.1)}) implies the feasibility of {(P.2)} (resp. {(Q.2)}), and vice versa. \qedwhite 
\end{myprop}
\rprop{feasibilityresult} implies that the domain of the control gains (resp. event-triggering gains) for which the geometric programming problem in Proposition~2 (resp. Proposition~3) with $p= p^{(1)}$ is feasible equals the one with $p= p^{(2)}$. Hence, if the cost function of the geometric programming problem in Proposition~2 depends only on the control gains, i.e., $g_c (\tilde{k}_i, \tilde{l}_{ij}, i \in \mathcal{V}, j \in {\cal N}^{\rm out} _i)$, then the optimal solution with $p= p^{(1)}$ equals the one with $p= p^{(2)}$.
In general, we define the cost functions depending only on the control gains, since we would like to optimize these parameters (see \req{controlparamcostfunc} as an example of the cost function). 
Similarly, if the cost function of the geometric programming problem in Proposition~3 depends only on the event-triggering gains, i.e., $g_e (\tilde{\sigma}_i, \eta_i, i \in \mathcal{V})$, then the optimal solution with $p= p^{(1)}$ equals the one with $p= p^{(2)}$ (see \req{eventtriggerparamcostfunc} as an example of the cost function).
%In general, we define the cost functions depending only on the control/event-triggering gains, since we would like to optimize the control and the event-triggering gains. 
%In general, we define the cost functions depending only on the control gains (and the event-triggering gains), since we would like to optimize them. 
%Thus, regardless of the choice of $c_p$ in the optimization problem \req{designlyapunov}, the geometric program yields the same optimal control gains. 
For the proof of \rprop{feasibilityresult}, see Appendix~C.}

\textcolor{black}{
\subsection{On the conservativeness of \rprop{gpproposition} and 3}\label{conservativenesssec}
In this section, we discuss the potential conservativeness of Proposition~2 and~3. Note that the posynomial constraints in \rprop{gpproposition} and 3 are given as the \textit{sufficient} conditions to those in \rthm{mainresult}. 
Hence, it is worth discussing how the posynomical constraints derived in \rprop{gpproposition} and 3 are conservative with respect to those in \rthm{mainresult}. 
For deriving the posynomial constraints of the control gains (\rprop{gpproposition}), the sufficiency has arisen since the term $\sum_{i\notin\mathcal{C}}\frac{-p_i\tilde{r}_{c,i}}{\tilde{s}_{c,i}}$ in the right hand side of \req{gaincalconst} has been replaced by $\sum_{i\notin\mathcal{C}}\frac{-p_i\tilde{r}_{c,i}}{c_{1,i}}$ in \req{gaincalconst3}, using the inequalities $\frac{-p_i\tilde{r}_{c,i}}{c_{1,i}} \leq \frac{-p_i\tilde{r}_{c,i}}{\tilde{s}_{c,i}}$ for all $i \notin {\cal C}$. This conservative operation (i.e., from \req{gaincalconst} to \req{gaincalconst3}) has been taken for all $i \notin \mathcal{C}$. 
Thus, the conservativeness for designing the control gains from Proposition~2 increases as the number of the nodes $i$ satisfying $i \notin \mathcal{C}$ (i.e., the cardinality of $\mathcal{V}\backslash\mathcal{C}$) increases. In other words, the conservativeness decreases as the number of the nodes $i$ satisfying $i \in \mathcal{C}$ (i.e., the cardinality of $\mathcal{C}$) increases. Recall that $\mathcal{C}$ is defined as $\mathcal{C}= \{i \in \mathcal{V} : \tilde{r}_{c, i}\ge 0\}$, where $\tilde{r}_{c, i} = -p_i\underline{\delta}_i + \sum_{j\in \mathcal{N}^{\rm out}_i} p_j\bar{\beta}_{ij}$ (see \req{positivec}). 
Hence, $i \in \mathcal{C}$ implies 
\begin{align}
-p_i\ubar{\delta}_i+ \sum_{j\in \mathcal{N}^{\rm out}_i} p_j\bar{\beta}_{ij} = p_i(\bar{\beta}_{ii}-\ubar{\delta}_i) + \sum_{j\in \mathcal{N}^{\rm out}_i\backslash \{i\}} p_j\bar{\beta}_{ij}\geq 0. \label{cin}
\end{align}
If $\bar{\beta}_{ii} \geq \ubar{\delta}_i$ (i.e., the baseline infection rate for the node itself is larger than the natural recovery rate), then \req{cin} holds. 
%Thus, $i \notin \mathcal{C}$ implies that the natural recovery rate is \textit{at least} larger than the baseline infection rate of the node itself (to make the term $\overline{\beta}_{ii}-\underline{\delta}_i$ negative). 
%The conservative operation (i.e., from \req{gaincalconst} to \req{gaincalconst3}) has been taken for all $i \notin \mathcal{C}$. 
%Thus, the conservativeness for designing the control gains from Proposition~2 increases as the number of the nodes satisfying \req{cin} increases. 
Hence, as the number of nodes satisfying $\bar{\beta}_{ii} \geq \underline{\delta}_i$ increases, which may be the case where a huge outbreak (i.e., the infection rates are large) happens, the conservativeness for designing the control gains from Proposition~2 decreases.}
%Nevertheless, we believe that the conservative operation as described above might not be a serious issue in the practical scenario, since the case where the natural recovery rate is (at least) larger than the baseline infection rate of the node itself may be rare. 
%Indeed, $\bar{\beta}_{ii} < \ubar{\delta}_i$ would imply that, if the infection rates from the neighbors are all zero (i.e., $\bar{\beta}_{ji} = 0$ for all $j \in \mathcal{N}^{\rm in}_{i}\backslash \{i\}$), the infected people for node $i$ would exponentially decrease for every $x_i(0) \in [0, 1]$ \textit{without} any control of $u_i$ and $v_i$ (i.e., $u_i(t) = v_{ii} (t) = 0, \forall t \geq 0)$. However, in view of the current situation of the COVID-19 (as well as the other pandemics), the utilization of $u_i$ and $v_{ii}$, such as providing the medical treatments, social distances, or closure of the facilities, should be necessary even in the rural area (in which the infection rates from the other areas are relatively small) so as to reduce the number of the infected people. }

\textcolor{black}{
For deriving the posynomial constraints of the event-triggering gains (Proposition~3), the sufficiency has arisen in \req{constrelax}, since we have used the following inequality: ${x}^\mathsf{T}Qx - r^\mathsf{T} x \geq {x}^\mathsf{T} S x - r^\mathsf{T} x$, 
where $Q = S+ \frac{1}{2} PL (I_n - G) (G+H)$ (see \req{shat}). The term $\frac{1}{2} PL (I_n - G) (G+H)$ is the matrix that enumerates the coefficiencies of the cross term $x_i x_j,\ j \in \mathcal{N}^{\rm in} _i$ in the derivative of $x_i$ (see \req{xeq}). %While deriving the posynomical constraints for Proposition~3, we neglect the effect of the cross term. 
%The effect of the cross term for node $i$ will be large if it has many in-neighbors. 
%Hence, the conservativeness increases if the number of edges pointing to the node $i$ is large. 
Since we neglect this term for deriving the posynomial constraints, it implies that $\sigma_i, \eta_i$ will be more conservatively selected as the number of the in-neighbor nodes for node $i$ is larger. 
In other words, if the number of the in-neighbor nodes for node $i$ is very large, very small $\sigma_i, \eta_i$ could be obtained, which could result in frequent control updates (or the geometric programming problem in Proposition~3 may become infeasible).} 

\textcolor{black}{If the geometric programming problem of finding the control gains (\rprop{gpproposition}) is not feasible, we have no choice but could try to change the candidate Lyapunov function (i.e., modify the parameter $p$), or, if allows, try to enlarge the upper bound of the control gains $\bar{k}_i, \bar{l}_{ij}$ in \req{constraintcontrolparameters} so as to increase the feasibility domain of the geometric programming problem in \rprop{gpproposition}. 
If \rprop{gpproposition} is feasible but \rprop{gppropositione} is not feasible, we could modify the cost function in Proposition~2 (e.g., change the weight parameters of \req{controlparamcostfunc}), or, slightly tighten the constraints in \rprop{gpproposition} so as to make \rprop{gppropositione} feasible. More specifically, in \rprop{gpproposition}, we replace the constraints \req{gpthree}, \req{gpone}, \req{gptwo} with 
\begin{align}
   & \tilde{s}_{c,i}+ (1-\epsilon_s)\left ( p_i\tilde{k}_i + \sum_{j\in \mathcal{N}^{out}_i}p_j\tilde{l}_{ij}\right )\le c_{1,i},\ \forall i \in \mathcal{V} \label{gpmodified1}\\ 
     &{\xi}_c ^{\frac{1}{2}}+\sum_{i\in \mathcal{C}}{p_i(\tilde{r}_{c,i} + \epsilon_r +\epsilon_3)}{\tilde{s}^{-1}_{c,i}}\le c_{2,m},\ \forall m\in \{1,\ldots, M\}\label{gpmodified3}\\
  &\left(\sum_{i \in \mathcal{V}}{p_i^2}{\tilde{s}^{-1}_{c,i}}\right)\left(\sum_{i \in \mathcal{V}}{(\tilde{r}_{c,i}+\epsilon_r + \epsilon_3)^{2}}{\tilde{s}^{-1}_{c,i}}\right) \le {\xi}_c,\label{gpmodified2}
\end{align}
where $\epsilon_s, \epsilon_r \in (0, 1)$ are {given} positive constants. Note that setting $\epsilon_s, \epsilon_r \rightarrow 0$ in \req{gpmodified1}, \req{gpmodified3} and \req{gpmodified2} corresponds to \req{gpthree}, \req{gpone} and \req{gptwo}, respectively.} 
%Suppose that \req{gpfour}, \req{gpfive} and \req{gpmodified1}--\req{gpmodified2} are feasible, and 
\textcolor{black}{Let $\tilde{k}^*_i, \tilde{l}^*_{ij}, \tilde{s}^*_{c,i} >0$ for all $i \in \mathcal{V}$, $j \in {\cal N}^{\rm out} _i$ and $\epsilon^*_1, \epsilon^*_2, \epsilon^*_3, {\xi}^*_{c} >0$ denote an any feasible solution to the posynomial constraints \req{gpfour}, \req{gpfive}, \req{gpmodified1}--\req{gpmodified2}. 
The corresponding control gains are denoted as $k^* _i = \bar{k}_i-\tilde{k}^*_i, \ l^*_{ij} =  \bar{l}_{ij} -\tilde{l}^* _{ij}$. 
Then, it follows that 
\begin{align}
    &\tilde{s}^* _{c,i}+ (1-\epsilon_s)\left ( p_i\tilde{k}_i + \sum_{j\in \mathcal{N}^{out}_i}p_j\tilde{l}_{ij}\right )\le c_{1,i}, \notag \\
    &\Longleftrightarrow \tilde{s}^* _{c,i} (1-\epsilon_s)^{-1} \leq p_i{k}^* _i + \sum_{j\in \mathcal{N}^{out}_i}p_j{l}^*_{ij} = c_{3,i}. \label{cond1}
\end{align}
Moreover, if $\eta_i = \epsilon_r/c_{3,i}$, it follows that  \begin{align}
    & \tilde{r}_{c,i} + c_{3,i} \eta_i =  \tilde{r}_{c,i} + \epsilon_r\ \Longleftrightarrow\ \tilde{r}_{c,i} (\tilde{r}_{c,i} + \epsilon_r)^{-1} + c_{3,i} \eta_i (\tilde{r}_{c,i} + \epsilon_r)^{-1}  =  1 \label{cond2}
\end{align}
Additionally, from \req{gpone} and \req{gptwo}, it follows that 
\begin{align}
& {\xi}_c ^{*\frac{1}{2}}+\sum_{i\in \mathcal{C}}{p_i(\tilde{r}_{c,i} + \epsilon_r +\epsilon^*_3)}{\tilde{s}^{*-1}_{c,i}}\le c_{2,m},\ \forall m\in \{1,\ldots, M\} \label{cond3}\\
&\left(\sum_{i \in \mathcal{V}}{p_i^2}{\tilde{s}^{*-1}_{c,i}}\right)\left(\sum_{i \in \mathcal{V}}{(\tilde{r}_{c,i}+\epsilon_r + \epsilon^*_3)^{2}}{\tilde{s}^{*-1}_{c,i}}\right) \le {\xi}^*_c \label{cond4}
\end{align}
Now, consider the feasibility problem provided in Proposition~3: find $\tilde{\sigma} _i$, $\tilde{s} _{e,i}$, $\eta _i$, $\tilde{r} _{e,i} > 0$ for all $i \in \mathcal{V}$ and $\epsilon_1, \epsilon_2, \epsilon_3, {\xi} _{e} > 0$ such that \req{propositionconstone}--\req{propositionconstfive} hold. 
Suppose that $\epsilon_r$ is chosen small enough such that $\epsilon_r < c_{3, i}$ for all $i\in \mathcal{V}$. Then, 
\req{cond1}, \req{cond2}, \req{cond3} and \req{cond4} imply that the posynomical constraints \req{propositionconstone}--\req{propositionconstfive} are all feasible with 
\begin{align}
    &\tilde{s}_{e,i} = \tilde{s}^* _{c,i},\ \tilde{\sigma}_i = 1-\epsilon_s,\ \tilde{r}_{e,i} = \tilde{r}_{c,i} + \epsilon_r,\  \eta_i = \epsilon_r / c_{3, i} \\
    &\xi_e = \xi^*_c,\ \epsilon_1 = \epsilon_s,\ \epsilon_2 = 1- \epsilon_r / c_{3,i},\ \epsilon_3 = \epsilon^*_3. \ 
\end{align}
Therefore, if we utilize the slightly tightened constraints in Proposition~2, and if it is feasible, we can guarantee the feasibility of the posynomical constraints in Proposition~3.}

\textcolor{black}{Even though Proposition~3 becomes feasible, it is still possible that, due to the conservativeness as described above, very small event-triggering gains $\sigma_i, \eta_i$ might be obtained. 
In such case, we could make use of the result of Theorem~1 in order to reduce the conservativeness. That is, if the resulting $\sigma_i, \eta_i$ according to Proposition~3 are very small for some $i$, we could try to increase these parameters (i.e., $\sigma_i \leftarrow \sigma_i + \epsilon_{\sigma}, \eta_i \leftarrow \eta_i + \epsilon_{\eta}$ for some $\epsilon_{\sigma}, \epsilon_{\eta}>0$ and the other parameters are fixed) and then check if \req{thetastar} in \rthm{mainresult} holds. If \req{thetastar} holds, then it follows that the control objective is achieved even with the modified event-triggering gains. 
Since the condition in Theorem~1 is less conservative than those in Proposition~3, we have the potential to enlarge the event-triggering gains. Note that \req{thetastar} can be checked via convex program, since all the control and the event-triggering gains are here given. While the above approach may be somewhat heuristic, it will be useful in practice for reducing the conservativeness of Proposition~3.}

\section{Numerical simulations}\label{numericalsimulationsec}
In this section, we demonstrate the performance of the proposed event-triggered controller through a numerical simulation. \textcolor{black}{The simulation have been conducted on MacOS Big Sur, 8-core Intel Core i9 2.4GHz, 32GB RAM using Python 3.} \textcolor{black}{Moreover, we used CVXPY for solving convex optimization problems. The code is available on Github: \textcolor{black}{\url{https://github.com/yugaro/etc-sis-solver}}.}

\smallskip
\textit{(Problem setup):} 
%\del{We apply the proposed approach against the epidemic spreading propagated over the network of 50 cities using an air transportation network in the United States (U.S.).} 
\textcolor{black}{We apply the proposed approach against the epidemic spreading propagated over an air transportation network consisting of 50 airports in the United States (U.S.).} The graph is constructed from the statistical data \citep{boarding2019all,openflights2019airport} of the number of passengers and flights. More specifically, we extract the data from \citep{boarding2019all} the top 50 U.S. airports according to the number of passengers in 2019, and from \citep{openflights2019airport} the number of the flights among the airports. 
%\del{The resulting graph consists of 50 nodes (i.e., ${\cal V} = \{1, \ldots, 50\}$) that represent the set of 50 cities in accordance with the selected airports, and the set of edges that represent the existence of the directed flights among them.} 
\textcolor{black}{The resulting graph consists of 50 nodes (i.e., ${\cal V} = \{1, \ldots, 50\}$) that represent the set of 50 airports, and the set of edges among different nodes that represent the existence of the directed flights among them.} 
%\del{\rfig{airtransportnetwork} depicts the resulting graph structure, where the size of each node is scaled by the eigenvector centrality \citep{ruhnau2000}, which is a measure of the influence of each node in the network. } 
\textcolor{black}{\rfig{airtransportnetwork} depicts the resulting graph structure, where the nodes in the network are divided into three groups denoted by ${\cal V}_1$, ${\cal V}_2$, ${\cal V}_3$ $\subseteq {\cal V}$, which are called Group~1, 2 and 3, respectively.}
%\onoue{and each edge is assigned to a weight in accordance with the number of passengers taking the corresponding directed flight, where the directions of the edges are not described.}
%\del{Moreover, as shown in \rfig{airtransportnetwork}, the network is divided into three groups denoted by ${\cal V}_1$, ${\cal V}_2$, ${\cal V}_3$ $\subseteq {\cal V}$, which are called Group~1, 2 and 3, respectively.} %This division is given by employing Community API (python-louvain)~\citep{communityapi}, a python-based package to detect dense communities based on the modularity index. 
%which divides the network  (see \citep{hoge}), 
%This division was provided by (approximately) maximizing the modularity index \citep{brandes2007modularity}, which is a measure of ...\hashi{この部分一言でもう少し簡潔に言えませんか？}\onoue{modularity index, which is a quality measure of clustering a network into several groups.}, and is used to define our control objective (see below).  
%which are labeled by Blue (B), White (W), and Red (R), respectively. 
%\onoue{This division was computed based on Community API (Python-Louvain) by modularity index, which is a quality measure of clustering a network into several groups,..}
%based on the modularity index \citep{brandes2007modularity}, which 
\begin{figure}[tbp]
  \centering
  \includegraphics[width = 12cm]{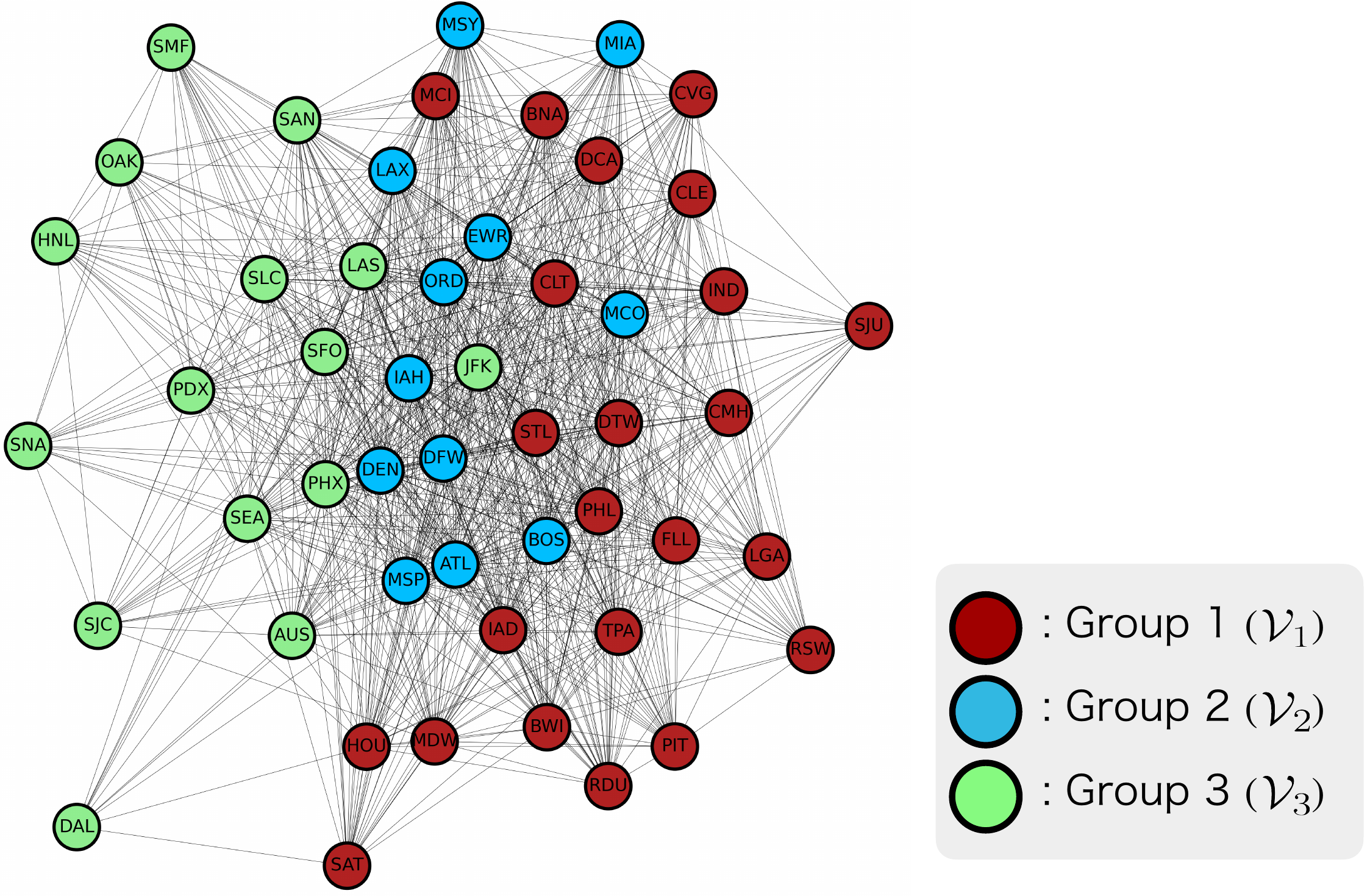}
  \caption{\textcolor{black}{The graph structure represents an air transport network consisting of 50 airports in the U.S.\citep{boarding2019all,openflights2019airport}}. The directions of the edges in the graph are omitted for brevity. \textcolor{black}{All the nodes are labeled by the world airport codes (IATA 3-letter codes)}, and divided by the three groups colored by red (Group~1), blue (Group~2), and green (Group~3).  %Group~2(blue)/Group~3(white). 
  }
  \label{airtransportnetwork}
\end{figure}
%\onoue{Moreover, based on the modularity, which is a graph property that measures the degree to which densely connected compartments within a graph can be partitioned into separate groups that interact more among themselves rather than the other groups, we divide the set of nodes into three groups $\mathcal{V}_1, \mathcal{V}_2, \mathcal{V}_3\subseteq \mathcal{V}$ with three colors: Group1 ($\mathcal{V}_1$) with red, Group2 ($\mathcal{V}_2$) with blue, and Group3 ($\mathcal{V}_3$) with yellow. Based on these clustering, we define our control objective given by \req{objectivegroup} as containing the average fraction of infected people $\frac{1}{|\mathcal{V}_m|}\sum_{i\in \mathcal{V}_m}x_i(t)$ in Group1 ($\mathcal{V}_1$), Group2 ($\mathcal{V}_2$), and Group3 ($\mathcal{V}_3$) below the corresponding thresholds $\bar{x}_1 = 8.0\times 10^{-2}$, $\bar{x}_2 = 1.0\times 10^{-1}$, and $\bar{x}_3 = 9.0\times 10^{-2}$, respectively. As described in \rsec{ControlObjective}, such control objective can be expressed by \req{objective} with the suitable parameters $w_m \in \{0, 1\}^n$, $\bar{d}_m$ for all $m \in \{1, 2, 3\}$.}

In this numerical simulation, we randomly choose the baseline recovery rate $\ubar{\delta}_i$ of each node $i\in \mathcal{V}$ from a uniform distribution on the interval 
% It is assumed that each baseline recovery rate $\ubar{\delta}_i$, $i \in {\cal V}$ in \req{controldynamics} is chosen randomly in 
$[8.0\times 10^{-2}, 1.0\times 10^{-1}]$. 
On the other hand, each baseline infection rate from the neighbor $\bar{\beta}_{ij}$ ($i \in {\cal V}$, $j \in {\cal N}^{\rm out}_{i}\backslash \{i\}$) is chosen on the interval $(0, 5.0\times 10^{-2}]$ in accordance with the data \citep{openflights2019airport}.
More specifically, each $\bar{\beta}_{ij}$ ($i \in {\cal V}$, $j \in {\cal N}^{\rm out}_{i}\backslash \{i\}$) is chosen to be proportional to the number of the direct flights from the airport corresponding to node $i \in \mathcal{V}$ to the one corresponding to node $j\in\mathcal{N}_i^{\rm out}\backslash \{i\}$, and is normalized such that the maximum value of $\bar{\beta}_{ij}$ ($i \in {\cal V}$, $j \in {\cal N}^{\rm out}_{i}\backslash \{i\}$) equals $5.0\times 10^{-2}$. \textcolor{black}{Moreover, each baseline infection rate for the node itself $\bar{\beta}_{ii}$, $i \in \mathcal{V}$ is chosen to be proportional to the population of the city that contains the airport corresponding to the node $i$, and is normalized such that the maximum value of $\bar{\beta}_{ii}$ ($i \in {\cal V}$) equals to $5.0\times 10^{-2}$.} 

The control objective in this numerical simulation is to achieve \req{objectivegroup} for all $m \in \{1, 2, 3\}$ with $\bar{x}_1 = 8.0\times 10^{-2}, \bar{x}_2 =  0.1$, $\bar{x}_3 = 9.0 \times 10^{-2}$. In other words, we aim at containing the average of the fraction of the infected people in Group~1, 2, and 3 within the corresponding thresholds $\bar{x}_1 = 8.0\times 10^{-2}, \bar{x}_2 =  0.1$, and $\bar{x}_3 = 9.0 \times 10^{-2}$, respectively. As described in \rsec{ControlObjective}, such control objective can be expressed by \req{objective} with appropriate selection of the parameters $w_m \in \{0, 1\}^n$ and  $\bar{d}_m$ for each $m \in \{1, 2, 3\}$. 
\begin{figure}[tbp]
  \centering
  \subfigure[Event-triggered controller.]{
    \centering
    \includegraphics[width = 0.48\hsize]{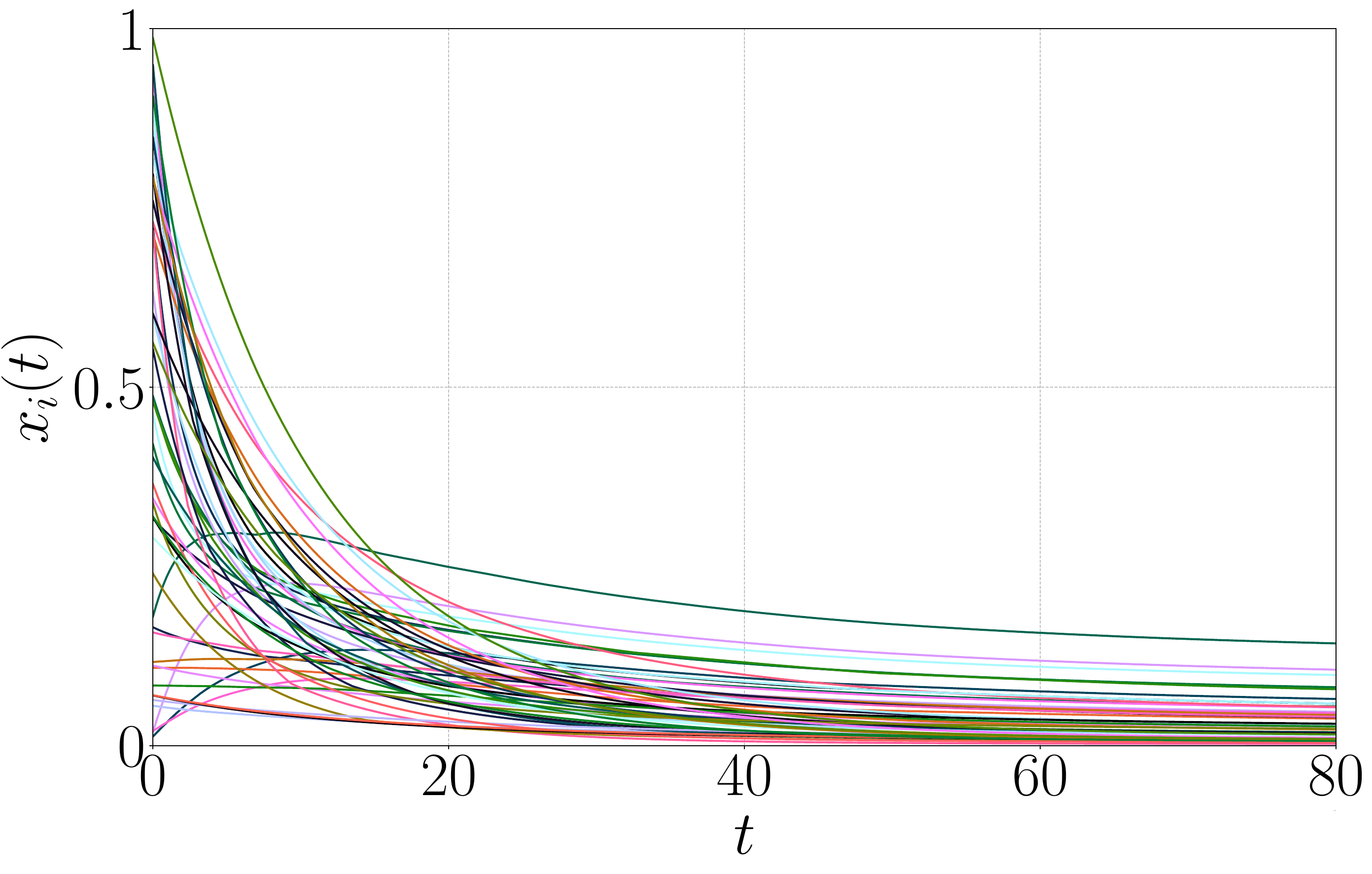}
    \label{statetrajectoryallcontrol}
  }
    \subfigure[Without control inputs.]{
    \centering
    \includegraphics[width = 0.48\hsize]{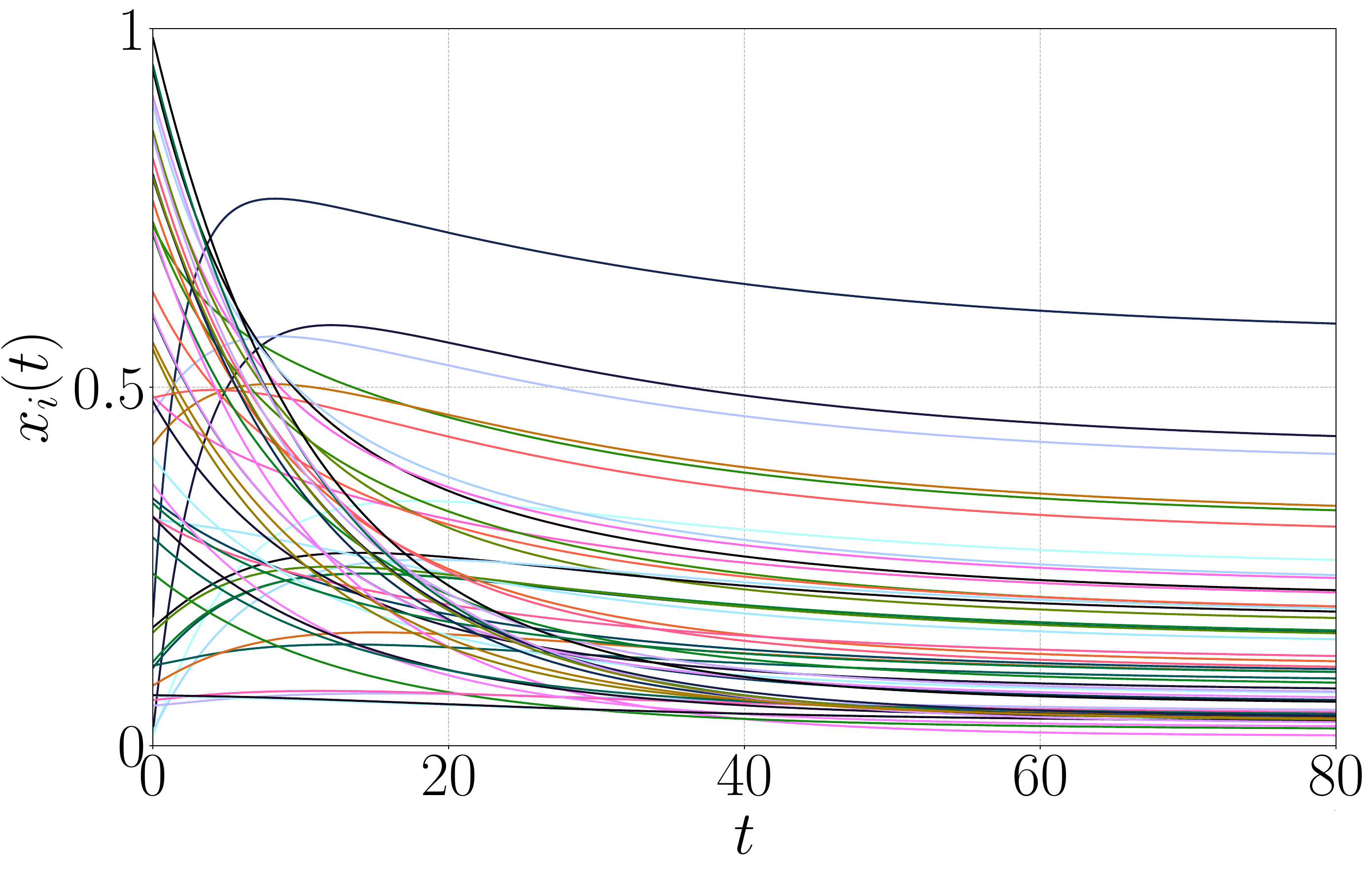}
    \label{statetrajectoryallnoinput}
    }
  \caption{State trajectories by applying the event-triggered controller designed by the emulation-based approach proposed in \rsec{eventtrigdesignsec}, and the state trajectories without control inputs, i.e., $u_i(t) = 0$ and $v_{ij}(t) = 0$ for all $i \in {\cal V}$, $j \in {\cal N}^{\rm out} _i$, and $ t \geq 0$.}
 \label{statetrajectoryall}
\end{figure}

%\onoue{\del{To define our control objective, we divide the set of nodes into three groups ${\cal V}_1$, ${\cal V}_2$, ${\cal V}_3$ $\subseteq {\cal V}$ according to the clustering, which are called Group~1, 2, and 3, respectively. Specifically, Group~1 consists of the nodes colored (B), Group~2 consists of the nodes colored (B)(Y)(C), and Group~3 consists of the nodes colored (B)(Y)(C)(R)(O) (see also \rfig{airtransportnetwork} for the illustration). Note that Group~3 corresponds to the set of all the nodes in the graph, i.e., ${\cal V}_3 = {\cal V}$.}} 
%Based on theses groups, the control objective is defined by \req{objectivegroup} for all $m \in \{1, 2, 3\}$ with $\bar{x}_1 = 0.10, \bar{x}_2 =  8.5\times 10^{-2}$, $\bar{x}_3 = 7.0 \times 10^{-2}$. In other words, we aim at containing the average of the fraction of the infected people in Group~1, 2, and 3 within the corresponding thresholds $\bar{x}_1 = 0.10, \bar{x}_2 =  8.5\times 10^{-2}$, and $\bar{x}_3 = 7.0 \times 10^{-2}$, respectively. As described in \rsec{ControlObjective}, such control objective can be expressed by \req{objective} with the suitable parameters $w_m \in \{0, 1\}^n$, $\bar{d}_m$ for all $m \in \{1, 2, 3\}$. 

The upper bounds of the control gains are given by $\bar{k}_i = 5.2\times 10^{-1}$, $\bar{l}_{ij} = 5.4\times 10^{-2}$ for all $i\in \mathcal{V}$ and $j\in \mathcal{N}_i^{\rm out}$. The Lyapunov parameter $p \in \mathbb{R}^n _{>0}$ in \req{eq11} is chosen by solving \req{designlyapunov}. When solving the geometric programming problem to find the optimal control gains from \rprop{gpproposition}, we define the cost function $g_c (\cdot)$ by \req{controlparamcostfunc} with $w_{k,i} = w_{l,ij} = 1$ for all $i \in \mathcal{V}, j \in \mathcal{N}^{\rm out} _i$. 
In addition, when solving the geometric programming problem to find the optimal event-triggering gains from \rprop{gppropositione}, %designing the event-triggering gains based on \rprop{gppropositione}, 
%solving the geometric programming to find the event-triggering gains (see \rprop{gppropositione}), 
we define the cost function $g_e(\cdot)$ by \req{eventtriggerparamcostfunc} with $w_{\sigma, i} = w_{\eta, i} =1$ for all $i \in \mathcal{V}$. 

\begin{comment}
\begin{align}\label{eventtriggerparamcostfunc}
g_e \left( Z_e \right) = \sum_{i\in \mathcal{V}} w_{\sigma, i} \tilde{\sigma}_i + \sum_{i\in \cal{C}} \frac{w_{\eta,i}}{\eta_i},
\end{align}
\textcolor{black}{where $w_{\sigma, i}, w_{\eta,i}>0$ for all $i \in \mathcal{V}$ are given weight parameters. Recall that $\tilde{\sigma} _i$ and $\eta_{i}$ for all $i\in{\cal V}$ are the variables satisfying $\tilde{\sigma} _i = 1- \sigma_i$ (see Appendix~B), and that $\sigma_i, \eta_{i}$ for all $i\in{\cal V}$ are the event-triggering gains. 
Hence, increasing $\sigma_i$ (resp. $\eta_{i}$) implies to reduce the cost of $\tilde{\sigma} _i$ (resp. $\eta_{i}^{-1}$). From \req{triggeringtime}, increasing $\sigma_i$ and $\eta_{i}$ allow us to reduce the number of the control updates. Therefore, minimizing \req{eventtriggerparamcostfunc} subject to the constraints \req{propositionconstone}--\req{propositionconstfive} implies to obtain large event-triggering gains so as to reduce the number of the control updates while achieving the control objective. For simplicity, here we set the weight parameters as $w_{\sigma, i} = w_{\eta, i} =1$ for all $i \in \mathcal{V}$.  } 
\end{comment}
%$\tilde{\sigma} _i$ 
%Since $\tilde{\sigma} _i = 1- \sigma_i$ for all $i\in{\cal V}$ (see Appendix~B), the cost function defined by \req{eventtriggerparamcostfunc} aims at obtaining large event-triggering gains so as to reduce the number of the control updates while achieving the control objective. 

\begin{figure}[tbp]
  \centering
  \subfigure[$m=1$]{
    \centering
    \includegraphics[width = 0.5\hsize]{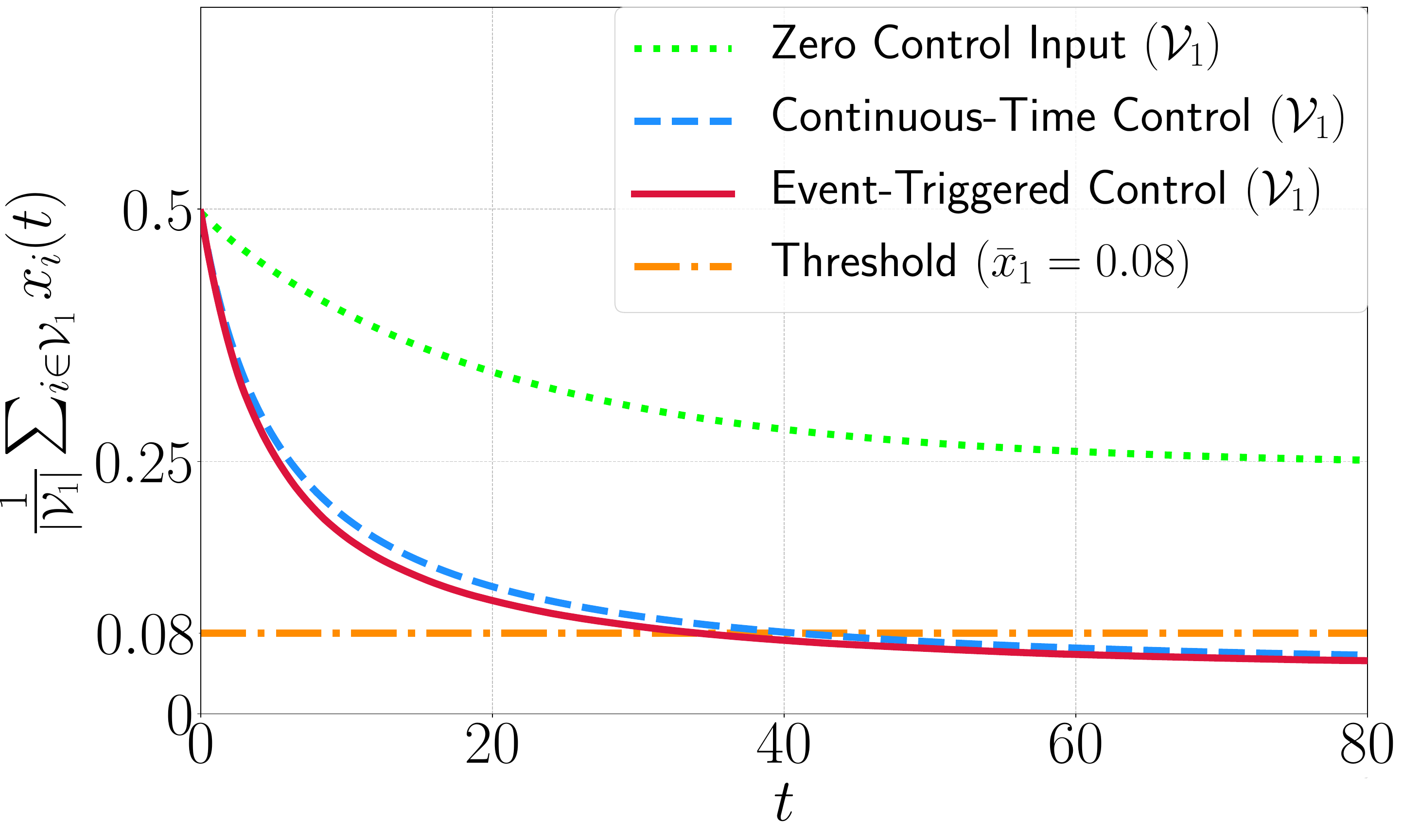}
    \label{statetrajectorycommunity1}
    }
  \subfigure[$m=2$]{
    \centering
    \includegraphics[width = 0.48\hsize]{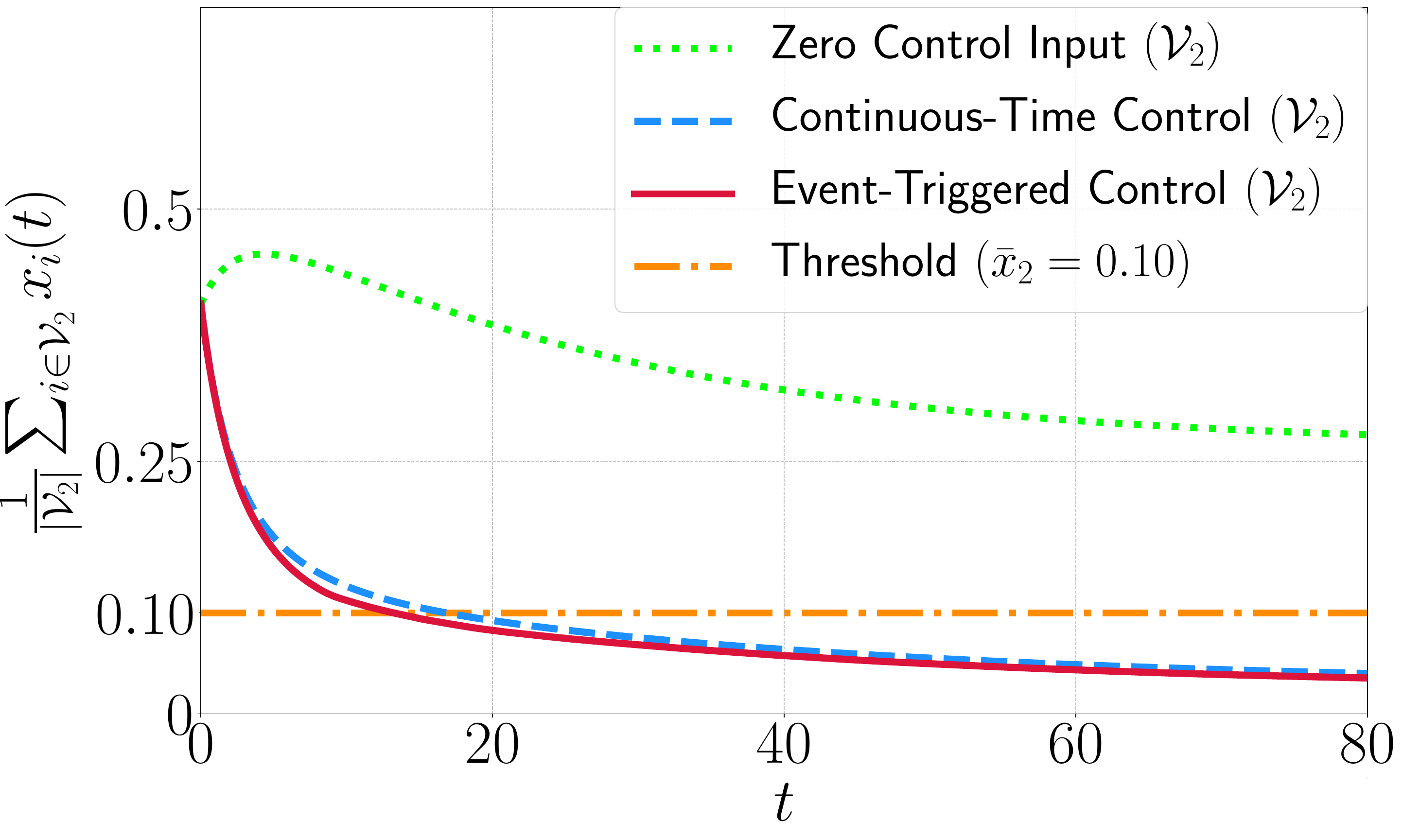}
    \label{statetrajectorycommunity2}
  }
  \subfigure[$m=3$]{
    \centering
    \includegraphics[width = 0.48\hsize]{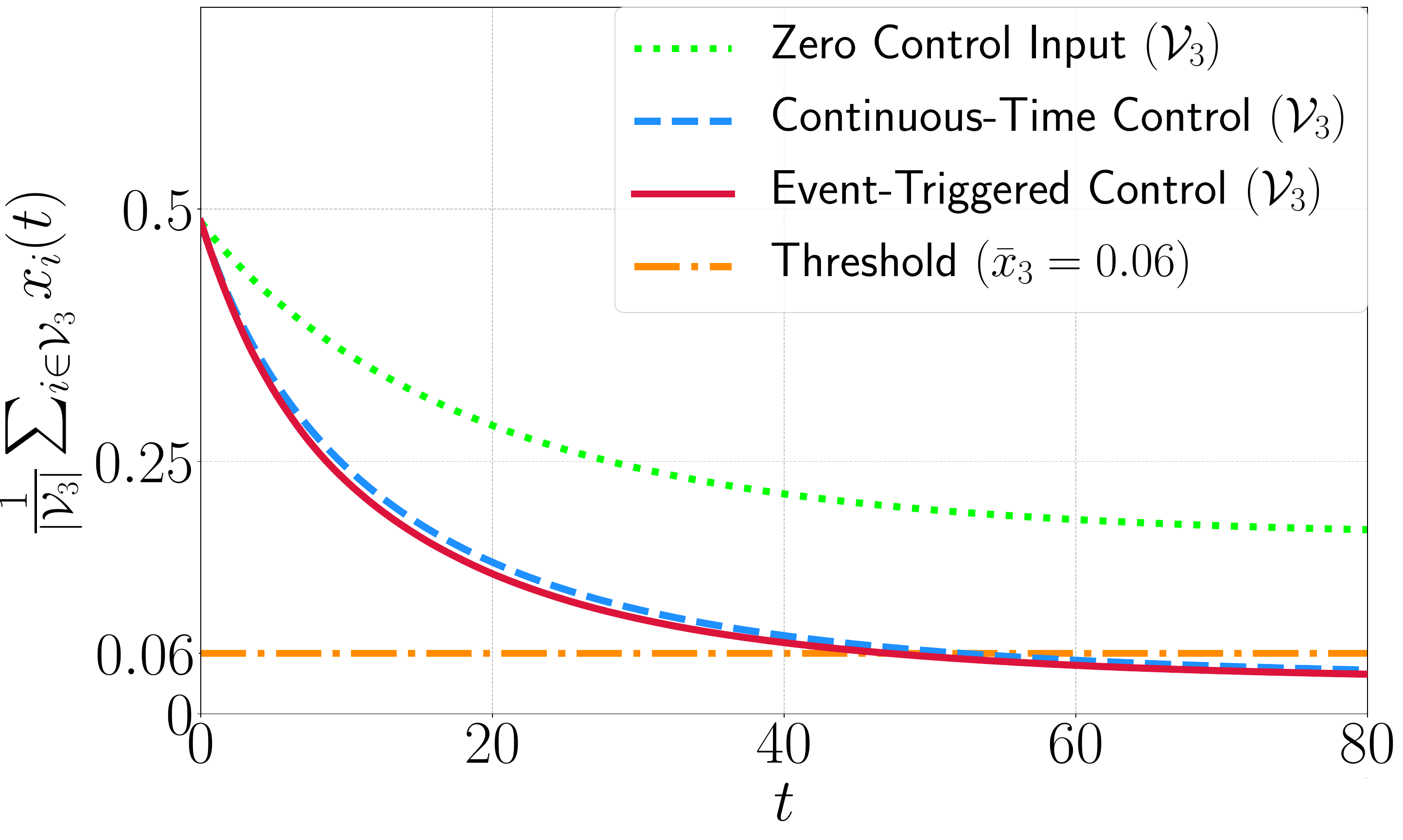}
    \label{statetrajectorycommunity3}
  }
  \caption{Trajectories of $|\mathcal{V}_m|^{-1}\sum_{i\in \mathcal{V}_m}x_i(t)$ for all $m \in \{1, 2, 3 \}$ under the proposed event-triggered controller (red solid line), without control inputs (green dotted line), and the continuous-time controller (blue dashed line) whose control gains are the same as the event-triggered controller. 
  The dash-dotted orange lines represent the thresholds $\bar{x}_1$, $\bar{x}_2$, and $\bar{x}_3$ for the control objective in the numerical simulation.}
  \label{statetrajectorycommunity}
\end{figure}

\begin{figure}[t]
  \centering
  \subfigure[Control inputs of $u_i(t)$ for $i\in\mathcal{V}_1$.]{
    \centering
    \includegraphics[width = 0.48\hsize]{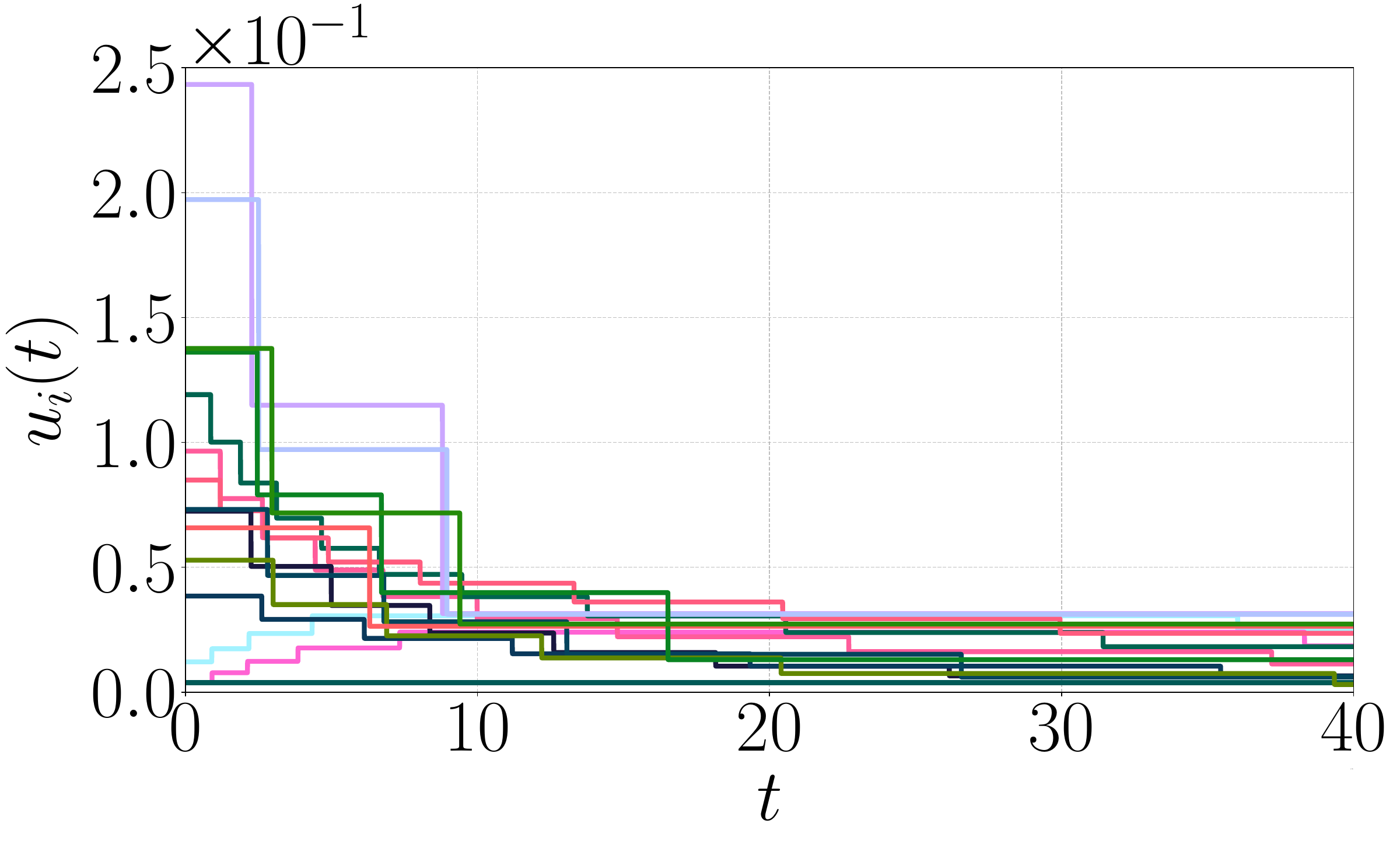}
    \label{controllertrajectoryrecovery}
    }
  \subfigure[Control inputs of $v_{ij}(t)$ for $i\in\mathcal{V}_1$.]{
    \centering
    \includegraphics[width = 0.48\hsize]{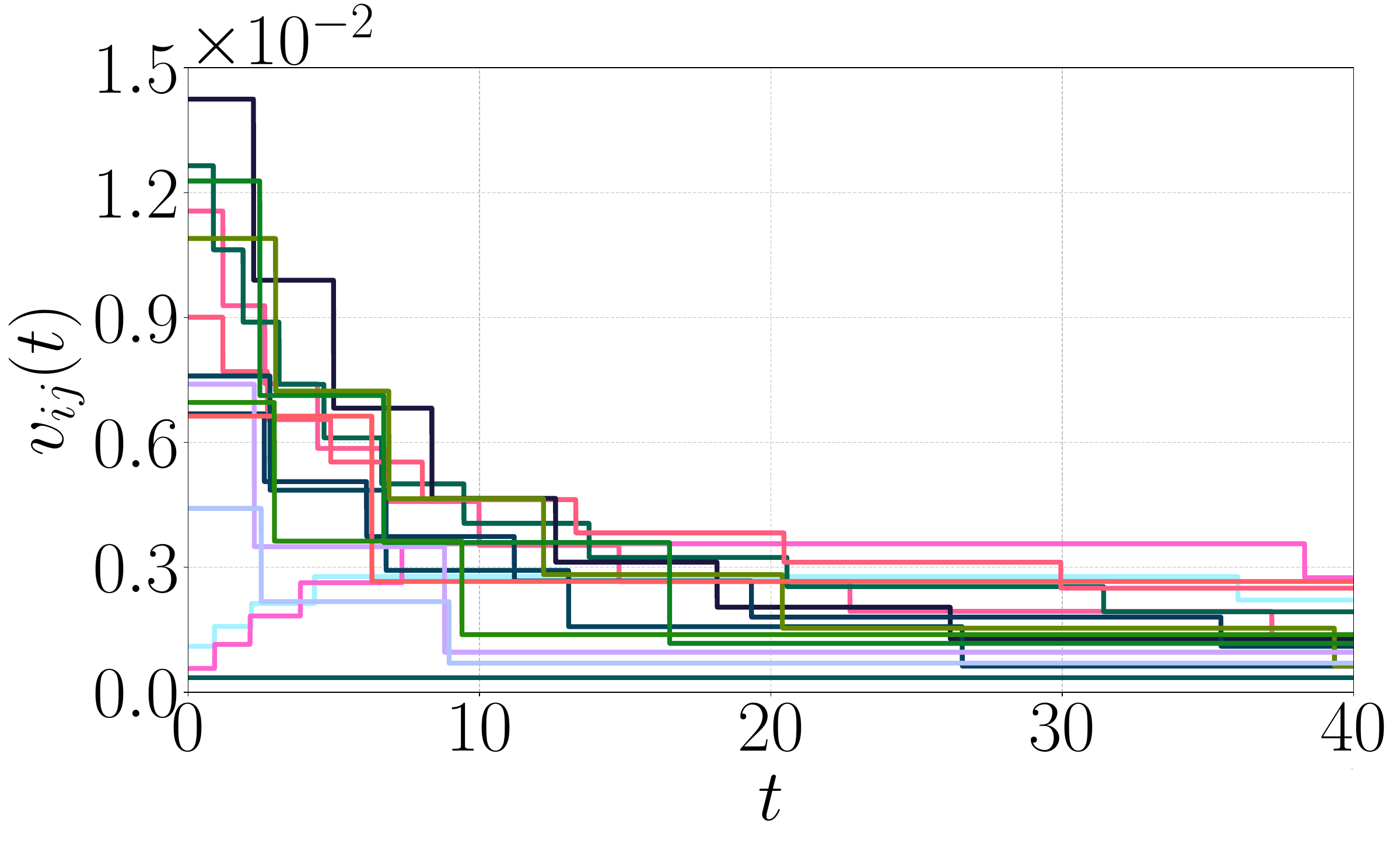}
    \label{controllertrajectoryinfection}
  }
  \subfigure[Inter-event times $t^i_{\ell + 1} - t^i_{\ell}$ for $i\in\mathcal{V}_1$.]{
    \centering
    \includegraphics[width = 0.48\hsize]{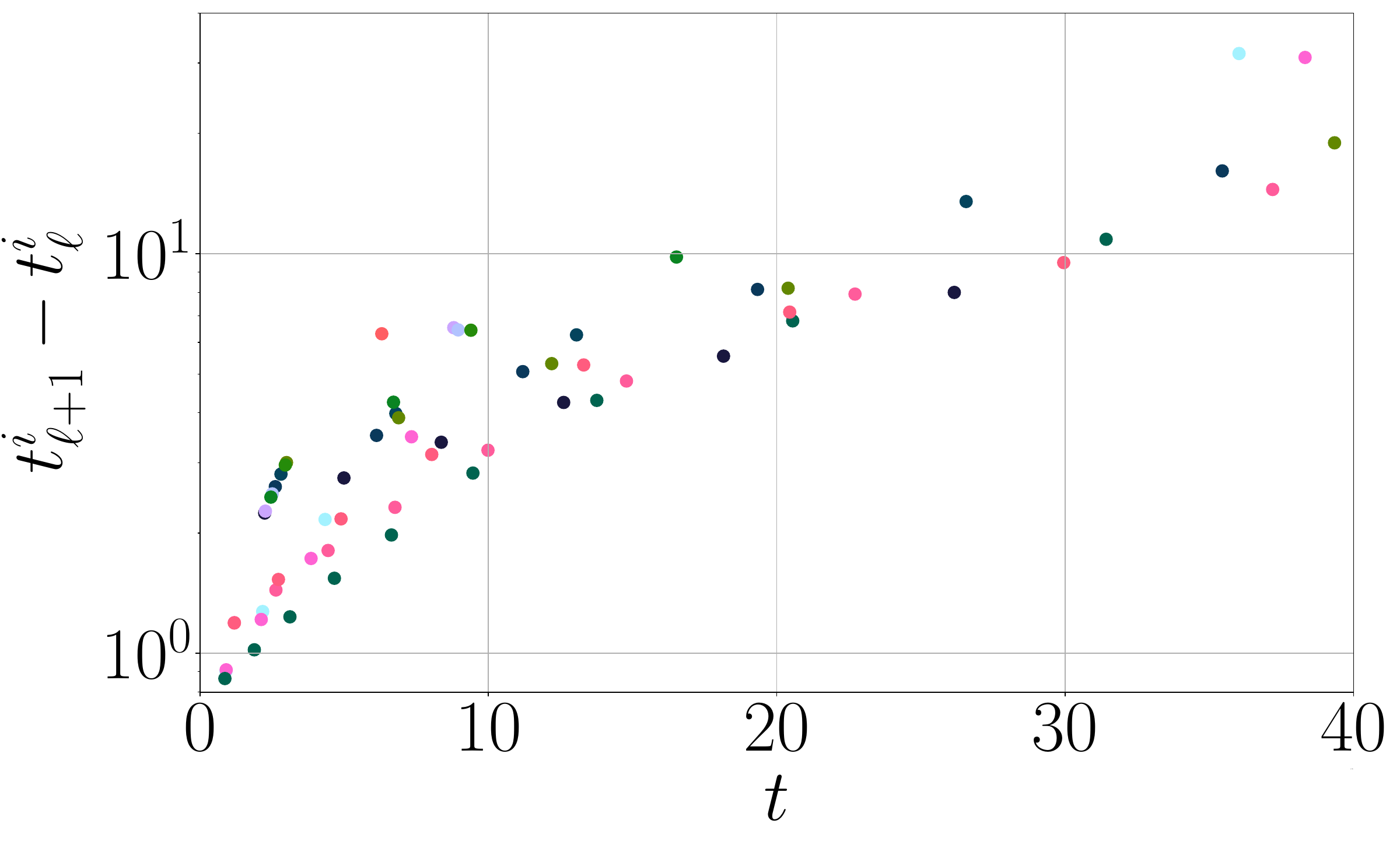}
    \label{controllertrajectorytriggering}
  }
  \caption{Control inputs $u_i(t)$, $v_{ij}(t)$ and the corresponding inter-event times $t_{\ell + 1}^i - t_{\ell}^i$, $\ell \in \mathbb{N}$ for $i\in\mathcal{V}_1$.}
 \label{controllertrajectory}
\end{figure}

\smallskip 
\textit{(Simulation results):} \rfig{statetrajectoryallcontrol} plots the state trajectories for all the nodes by applying the event-triggered controller designed by the emulation-based approach proposed in \rsec{eventtrigdesignsec}, where the initial state $x(0)$ is randomly given from a uniform distribution on the interval set $[0,1]^n$. 
%that has been designed by \rprop{gpproposition} and 3. 
For comparisons, we illustrate in \rfig{statetrajectoryallnoinput} the state trajectories without control inputs (i.e., the state trajectories with $u_i(t) = v_{ij}(t) = 0$, for all $i \in {\cal V}$, $j \in {\cal N}^{\rm out} _i$, and $t \geq 0$), where the initial state is the same as the event-triggered controller. 
\rfig{statetrajectoryallcontrol} shows that the fraction of infected people are contained effectively by applying the proposed event-triggered controller compared with the control-free case in \rfig{statetrajectoryallnoinput}. 
To verify that the control objective is achieved, we illustrate in \rfig{statetrajectorycommunity} the  trajectories of ${|\mathcal{V}_m|}^{-1} \sum_{i\in \mathcal{V}_m} x_i (t)$ for all $m \in \{1, 2, 3\}$ under the proposed event-triggered controller (red solid line), without control inputs (green dotted line), the continuous-time controller \req{uitimecon}, \req{vijtimecon} (see \rsec{designcontrolgainsec}; blue dashed line) whose control gains are the same as the event-triggered controller. \rfig{statetrajectorycommunity} shows that, by applying the event-triggered controller, the average of the fraction of infected people in each group can be appropriately contained such that the control objective is achieved (i.e., converge below the prescribed thresholds $\bar{x}_1, \bar{x}_2$, and $\bar{x}_3$), while at the same time preserving almost the same convergence performance as the continuous-time controller. \textcolor{black}{It can be also seen that the event-triggered controller provides a better control performance than the continuous-time controller. To describe why this has happened, recall that the event-triggered controller updates the control inputs {only when they are needed}, while the continuous controller updates them {continuously}. Moreover, the linear state-feedback controller is employed under both the continuous controller and the event-triggered controller (see \req{ui}, \req{vij}).
%Hence, when the state is decreasing, the control inputs under the event-triggered controller tend to be larger than the ones under the continuous-time controller. 
Thus, when the state (the fraction of the infected people) is decreasing, the control input under the continuous controller is decreasing for all times, while the control input under the event-triggered controller is kept constant during the inter-event times and it decreases only for the update times. Therefore, if the inter-event times are relatively long, the control input under the event-triggered controller tend to be larger than the one under the continuous-time controller. Therefore, the event-triggered controller tends to mitigate the infection further than the continuous-time controller (since it tends to utilize larger control inputs) and thus provides a better performance.
%Therefore, when the state is decreasing, the event-triggered controller tends to mitigate the infection further than the continuous-time controller (since it tends to utilize larger control inputs) and thus provides a better performance. 
Note that, when the state is \textit{increasing}, the reverse holds; the control inputs under the event-triggered controller tends to be smaller than the continuous-time controller and thus provide worse performance. In this numerical simulation, the event-triggered controller provides a better performance than the continuous time controller, since, as shown in \rfig{statetrajectoryall}, many states are \textit{decreasing} for most of the times.}
%To describe why this has happened, 
%From these figures, it can be 

%\rfig{controllertrajectoryrecovery} and \rfig{controllertrajectoryinfection} illustrate the trajectories of the control inputs of the recovery rate $u_i(t)$ and the infection rate $v_{ij}(t)$, respectively. To simplify the visualization, we only plot $v_{ij}(t)$ with the index $j$ given by $j={\rm argmax}_{j\in\mathcal{N}_i^{\rm out}}v_{ij}(0)$. 
%For the control input of the infection rate, where we plot $v_{ij}(t)$ corresponding only to the index $j={\rm argmax}_{j\in\mathcal{N}_i^{\rm out}}v_{ij}(0)$ for simplicity of visualization.

In \rfig{controllertrajectoryrecovery} and (b), we plot the trajectories of the control inputs $u_i(t)$, $v_{ij}(t)$ for $i \in {\cal V}_1$ under the proposed event-triggered controller. %For simplicity of illustration, the figure plots the control trajectories only for the nodes belonging to Group~1 (${\cal V}_1$). 
%for the nodes belonging to Group 1 (i.e., $i \in {\cal V}_1$). 
%Here, to make the simulation result visible, 
%Here, to increase visibility, we only plot $v_{ij}(t)$ only for the index $j$ given by $j={\rm argmax}_{j\in\mathcal{N}_i^{\rm out}}v_{ij}(0)$. 
In addition, \rfig{controllertrajectorytriggering} illustrates the corresponding inter-event times, i.e., $t_{\ell + 1}^i - t_{\ell}^i$ for $\ell \in \mathbb{N}$ and $i \in {\cal V}_1$. 
From the figures, we can observe that the control inputs are not updated continuously but are updated aperiodically based on the proposed event-triggered controller. We can furthermore observe the inter-event times tend to be larger as time evolves and the state trajectories converge, which implies that the control inputs are updated {only when} the fraction of the infected people in each node increases or decreases by the threshold designed by the proposed approach. In summary, we can confirm the validity of the proposed event-triggered controller by showing that, the control inputs for each node are updated only when they are needed while preserving almost the same convergence performance as the continuous-time controller. 

\section{Conclusions and future works}
In this paper, we proposed an event-triggered control-based framework for a deterministic SIS model. 
%The control objective is to contain the number of infected people below a given threshold. 
%toward the control objective of suppressing the number of infected people below a certain threshold. 
%Previous feedback or dynamic control strategies for epidemic spreading have been designed based on optimal control and feedback control, in which the policies are updated continuously or per unit of time. However, in social life and economy, such frequent updates of strategies are not necessarily suitable in practice, since even the small fluctuation of the processes forces us to update them. 
In the proposed framework, control inputs for each subpopulation are updated only when the fraction of the infected people increases or decreases by a prescribed threshold, aiming at reducing unnecessarily burden of updating the control inputs. 
%To overcome this challenge, we have devised our framework based on event-triggered control theories, in which the policies are updated not continuously but in response to the dynamic change in the number of infected people.
%As the first step toward forming our theoretical results, based on stability analysis, 
First, we analyzed the stability of the closed-loop system under the event-triggered controller. 
We in particular derived a sufficient condition for the event-triggered controller to achieve a prescribed control objective. 
%of containing the number of the infected people within a prescribe threshold. 
%, which can be checked in polynomial time using a convex program, for the event-triggered controller to achieve a prescribed control objective. 
We further showed that the derived conditions are characterized by convex programs, which can be thus efficiently solved in polynomial time. 
%for the epidemic processes under given event-triggered control-based strategies to asymptotically converge within the prescribed bound. 
%Based on the condition and the geometric programming, 
% which can be translated into the convex program, 
Then, we proposed an emulation-based approach towards the design of the event-triggered controller. We in particular showed that the problem of designing the control and the event-triggering gains can be solved by the geometric programmings, which can reduce to convex programs and can be efficiently solved in polynomial time. 
%reduce to convex programs. 
%formulated our controller synthesis to produce appropriate parameters of the event-triggered controller in an emulation-based approach.
%Because a geometric program equivalently reduces to a convex optimization problem, the proposed method allows us to efficiently design the controller.
Finally, we confirmed the validity of our proposed approach through numerical simulations of an epidemic spreading using an air transportation network.  
%an air transport network. %and verified our control objective can be achieved.

%Our future works involve the following aspects: 
The following aspects should be further pursued in our future works of research: 
\begin{itemize}
    \item \textcolor{black}{In the event-triggered control framework presented in this paper, the state for each node needs to be monitored continuously so as to determine the triggering time instants. Moreover, the inter-event times might still be very small especially when a huge outbreak happens. Hence, it is of great importance to provide some constraints on the \textit{lower bound} of the inter-event times, so as to further reduce the frequency of the control inputs. These issues could indeed be solved by employing the periodic event-triggered control \citep{heemels2013a}, in which the states are measured only periodically (not continuously as in the standard event-triggered control presented in this paper). Therefore, future work involves solving the above issues by formulating the periodic event-triggered control for mitigating the epidemic spreading.} 
    \item \textcolor{black}{In our current problem setup, we impose the constraints on the control \textit{gains} (see \req{constraintcontrolparameters}) in order to restrict the amount of the control inputs. However, in this formulation, the upper bound of the control input is utilized only when the state is $1$ (i.e., $u_i(t) = \bar{k}_i$ iff $x_i (t) =1$). Hence, it might be more reasonable and realistic to impose the constraints \textit{directly} on the control inputs, e.g., $u_i(t) \leq \bar{u}_i$, $v_{ij}(t) \leq \bar{v}_{ij}$ for all $t\geq 0$ rather than impose the constraints on the control gains. Hence, considering how to impose $u_i(t) \leq \bar{u}_i$, $v_{ij}(t) \leq \bar{v}_{ij}$ when designing the event-triggered controller for the epidemic spreading should be pursued in our future work.} 
    \item In addition to the above, there exist several research directions that should be further pursued in relation to the current COVID-19 pandemic. \textcolor{black}{For example, it is of both theoretical and practical interest to investigate the event-triggered control of more realistic epidemic models, such as the SIR models, the SEIR models, and various others (see, e.g., \citep{wang2014,Casella2021,DarabiSahneh2013,Giordano2020}).} Another important research direction is to develop a theoretical framework for the event-triggered control of epidemics over temporal networks \citep{ogura2016,pare2017epidemic} in order to account the intrinsic time-variability of contact networks in the human society. \textcolor{black}{Additionally, time delays arising in the feedback loop as well as parameter uncertainties in the epidemic models (see, e.g., \citep{Casella2021}) should be taken into account for further investigations in designing the event-triggered controller.} 
\end{itemize}

%In addition to the above, there exist several research directions that should be further pursued in relation to the current COVID-19 pandemic. \textcolor{black}{For example, it is of both theoretical and practical interest to investigate the event-triggered control of more realistic epidemic models, such as the SIR models, the SEIR models, and various others (see, e.g., \citep{wang2014,Casella2021,DarabiSahneh2013,Giordano2020}).} 
% \onoue{Although we have assumed in this paper that the SIS spreading processes can be dynamically captured in response to the (daily) reports of the active and total cases, it is shown that suppression strategies for the processes can be useful if strong enough and enacted early on (see, e.g., \citep{Casella2021}). In this direction, it is further of practical interest to devise effective and safe strategies to cope with the spreading processes of significant time delay based on decision-making systems as the event-triggered control.}
%\onoue{Much attention could also be directed to devise mitigation strategies based on the event-triggered control for spreading processes affected by time delay and uncertainty of (daily) test reports. \citep{Casella2021}}

\section*{Acknowledgement}
This work was supported in part by JSPS KAKENHI Grant 21H01353 and in part by JST ERATO HASUO Metamathematics for Systems Design Project (No. JPMJER1603). 

%\onoue{追加した方が良いとコメントがあった論文は以下です。}
%\onoue{https://ieeexplore.ieee.org/stamp/stamp.jsp?tp=\&arnumber=8723100}
%\onoue{https://ieeexplore.ieee.org/stamp/stamp.jsp?tp=\&arnumber=5934963}
% \masaki{引用文献の多様性のために下記のparaを書き足すことを提案します．}
% \add{There exist several research directions that should be further pursued in relation to the current COVID-19 pandemic. For example, it is of both theoretical and practical interest to investigate the event-triggered control of more realistic epidemic models having more than two compartments (see, e.g., \citep{DarabiSahneh2013,Giordano2020}). Another important research direction is to develop a theoretical framework for the control of epidemics over temporal networks \citep{ogura2016,Pare2018}\masaki{ogura2016はこっちにうつしました．} in order to account the intrinsic time-variability of contact networks in the human society.}

%\bibliography{sample,ogura}

\begin{thebibliography}{10}
\bibitem{RevModPhys}
R.~Pastor-Satorras, C.~Castellano, P.~Van~Mieghem, A.~Vespignani, Epidemic
  processes in complex networks, Reviews on Modern Physics 87 (2015) 925--979.

\bibitem{virusspread}
P.~{Van Mieghem}, J.~{Omic}, R.~{Kooij}, Virus spread in networks, IEEE/ACM
  Transactions on Networking 17~(1) (2009) 1--14.

\bibitem{malwareepidemic}
M.~{Garetto}, W.~{Gong}, D.~{Towsley}, Modeling malware spreading dynamics, in:
  Proceedins of the Twenty-second Annual Joint Conference of the IEEE Computer
  and Communications Societies, 2003, pp. 1869--1879.

\bibitem{socialepidemic}
K.~Lerman, R.~Ghosh, Information contagion: an empirical study of spread of
  news on digg and twitter social networks, in: Proceedings of 4th
  International Conference on Weblogs and Social Media (ICWSM), 2010.

\bibitem{cascadingfailure}
S.~Soltan, D.~Mazauric, G.~Zussman, Cascading failures in power grids: Analysis
  and algorithms, in: Proceedings of the 5th International Conference on Future
  Energy Systems, 2014, p. 195–206.

\bibitem{preciado2013traffic}
V.~M. Preciado, M.~Zargham, Traffic optimization to control epidemic outbreaks
  in metapopulation models, in: 2013 IEEE Global Conference on Signal and
  Information Processing, 2013, pp. 847--850.

\bibitem{preciado2014traffic}
V.~M. Preciado, M.~Zargham, D.~Sun, Traffic control for network protection
  against spreading processes, in: Proceedings of the 48th Annual Conference on
  Information Sciences and Systems, 2014, pp. 1--8.

\bibitem{epidemiccontrolsurvey}
C.~{Nowzari}, V.~M. {Preciado}, G.~J. {Pappas}, Analysis and control of
  epidemics: A survey of spreading processes on complex networks, IEEE Control
  Systems Magazine 36~(1) (2016) 26--46.

\bibitem{preciado2014}
V.~M. {Preciado}, M.~{Zargham}, C.~{Enyioha}, A.~{Jadbabaie}, G.~J. {Pappas},
  Optimal resource allocation for network protection against spreading
  processes, IEEE Transactions on Control of Network Systems 1~(1) (2014)
  99--108.

\bibitem{han2015}
S.~{Han}, V.~M. {Preciado}, C.~{Nowzari}, G.~J. {Pappas}, Data-driven network
  resource allocation for controlling spreading processes, IEEE Transactions on
  Network Science and Engineering 2~(4) (2015) 127--138.

\bibitem{nowzari2017}
C.~{Nowzari}, V.~M. {Preciado}, G.~J. {Pappas}, Optimal resource allocation for
  control of networked epidemic models, IEEE Transactions on Control of Network
  Systems 4~(2) (2017) 159--169.

\bibitem{mai2018}
V.~S. {Mai}, A.~{Battou}, K.~{Mills}, Distributed algorithm for suppressing
  epidemic spread in networks, IEEE Control Systems Letters 2~(3) (2018)
  555--560.

\bibitem{ogura2016b}
M.~Ogura, V.~M. Preciado, Epidemic processes over adaptive state-dependent
  networks, Phys. Rev. E 93 (2016).

\bibitem{kohler18}
J.~{Köhler}, C.~{Enyioha}, F.~{Allgöwer}, Dynamic resource allocation to
  control epidemic outbreaks a model predictive control approach, in: 2018
  Annual American Control Conference, 2018, pp. 1546--1551.

\bibitem{outputfeedback2018}
L.~A. Alarcón-Ramos, R.~B. Jaquez, A.~Schaum, Output-feedback control of virus
  spreading in complex networks with quarantine, Frontiers in Applied
  Mathematics and Statistics 4 (2018) 1--8.

\bibitem{watkins}
N.~J. {Watkins}, C.~{Nowzari}, G.~J. {Pappas}, Robust economic model predictive
  control of continuous-time epidemic processes, IEEE Transactions on Automatic
  Control 65~(3) (2020) 1116--1131.

\bibitem{liu2019}
J.~{Liu}, P.~E. {Paré}, A.~{Nedić}, C.~Y. {Tang}, C.~L. {Beck}, T.~{Başar},
  Analysis and control of a continuous-time bi-virus model, IEEE Transactions
  on Automatic Control 64~(12) (2019) 4891--4906.

\bibitem{gracy2020}
S.~{Gracy}, P.~E. {Pare}, H.~{Sandberg}, K.~H. {Johansson}, Analysis and
  distributed control of periodic epidemic processes, IEEE Transactions on
  Control of Network Systems (2020).

\bibitem{jhohanson}
A.~Janson, S.~Gracy, P.~E.Paré, H.~Sandberg, K.~H.Johansson, Networked
  multi-virus spread with a shared resource: Analysis and mitigation
  strategies, In arxiv: available at https://arxiv.org/abs/2011.07569.

\bibitem{sistutorial}
L.~J. Allen, Some discrete-time si, sir, and sis epidemic models, Math
  Bioscience 124~(1) (1994) 83--105.

\bibitem{Mei2017}
W.~Mei, S.~Mohagheghi, S.~Zampieri, F.~Bullo, {On the dynamics of deterministic
  epidemic propagation over networks}, Annual Reviews in Control 44 (2017)
  116--128.

\bibitem{heemels2012a}
W.~P. M.~H. Heemels, K.~H. Johansson, P.~Tabuada, An introduction to
  event-triggered and self-triggered control, in: Proceedings of the 51st IEEE
  Conference on Decision and Control, 2012, pp. 3270--3285.

\bibitem{corona}
{Causality levels in Akita Prefecture, Japan}:
  \url{https://www.pref.akita.lg.jp/pages/archive/51612#English}.

\bibitem{heemels2013a}
W.~P. M.~H. Heemels, M.~C.~F. Donkers, Model-based periodic event-triggered
  control for linear systems, Automatica 49~(3) (2013) 698--711.

\bibitem{heemels2013b}
W.~P. M.~H. {Heemels}, M.~C.~F. {Donkers}, A.~R. {Teel}, Periodic
  event-triggered control for linear systems, IEEE Transactions on Automatic
  Control 58~(4) (2013) 847--861.

\bibitem{boyd}
S.~Boyd, L.~Vandenberghe, Convex Optimization, Cambridge University Press,
  2004.

\bibitem{Ogura2019c}
M.~Ogura, M.~Kishida, J.~Lam, {Geometric programming for optimal positive
  linear systems}, IEEE Transactions on Automatic Control 65~(11) (2020)
  4648--4663.
  
\bibitem{dimos2012a}
D.~V. Dimagoronas, E.~Frazzoli, K.~H. Johansson, Distributed event-triggered
  control for multi-agent systems, IEEE Transactions on Automatic Control
  57~(5) (2012) 1291--1297.

\bibitem{dimos2013a}
G.~S. Seyboth, D.~V. Dimagoronas, K.~H. Johansson, Event-based broadcasting for
  multi-agent average consensus, Automatica 49~(1) (2013) 245--252.

\bibitem{fan2013}
Y.~Fan, G.~Feng, Y.~Wang, C.~Song, Distributed event-triggered control of
  multi-agent systems with combinational measurements, Automatica 49~(2) (2013)
  671 -- 675.

\bibitem{nowzari2019}
C.~Nowzari, E.~Garcia, J.~Cortés, Event-triggered communication and control of
  networked systems for multi-agent consensus, Automatica 105 (2019) 1 -- 27.

\bibitem{boyd2007tutorial}
S.~Boyd, S.-J. Kim, L.~Vandenberghe, A.~Hassibi, A tutorial on geometric
  programming, Optimization and engineering 8~(1) (2007) 67.

\bibitem{heemels2014a}
D.~P. {Borgers}, W.~P. M.~H. {Heemels}, Event-separation properties of
  event-triggered control systems, IEEE Transactions on Automatic Control
  59~(10) (2014) 2644--2656.

\bibitem{khalil}
H.~K. Khalil, Nonlinear Systems -\textit{third edition}-, Prentice Hall, 2002.

\bibitem{LAJMANOVICH}
A.~Lajmanovich, J.~A. Yorke, A deterministic model for gonorrhea in a
  nonhomogeneous population, Mathematical Biosciences 28~(3) (1976) 221 -- 236.

\bibitem{horn2012matrix}
R.~A. Horn, C.~R. Johnson, Matrix analysis, Cambridge university press, 2012.

\bibitem{boarding2019all}
{Federal Aviation Administration (Airport Planning \& Capacity Airports):}
  \url{https://www.faa.gov/airports/planning_capacity}.

\bibitem{openflights2019airport}
Airport, airline and route data: \url{https://openflights.org/data.html}.

\bibitem{wang2014}
J.~Wang, J.~Wang, M.~Liu, Y.~Li, Global stability analysis of an sir epidemic
  model with demographics and time delay on networks, Physica A 410 (2014)
  268--275.

\bibitem{Casella2021}
F.~Casella, {Can the COVID-19 Epidemic Be Controlled on the Basis of Daily Test
  Reports?}, IEEE Control Systems Letters 5~(3) (2021) 1079--1084.

\bibitem{DarabiSahneh2013}
F.~{Darabi Sahneh}, C.~Scoglio, P.~{Van Mieghem}, {Generalized epidemic
  mean-field model for spreading processes over multilayer complex networks},
  IEEE/ACM Transactions on Networking 21~(5) (2013) 1609--1620.

\bibitem{Giordano2020}
G.~Giordano, F.~Blanchini, R.~Bruno, P.~Colaneri, A.~{Di Filippo}, A.~{Di
  Matteo}, M.~Colaneri, {Modelling the COVID-19 epidemic and implementation of
  population-wide interventions in Italy}, Nature Medicine 26~(6) (2020)
  855--860.

\bibitem{ogura2016}
M.~{Ogura}, V.~M. {Preciado}, Stability of spreading processes over
  time-varying large-scale networks, IEEE Transactions on Network Science and
  Engineering 3~(1) (2016) 44--57.

\bibitem{pare2017epidemic}
P.~E. Par{\'e}, C.~L. Beck, A.~Nedi{\'c}, Epidemic processes over time-varying
  networks, IEEE Transactions on Control of Network Systems 5~(3) (2017)
  1322--1334.
\end{thebibliography}

\appendix 
\section{Proof of \rprop{gpproposition}}\label{proofprop}
\begin{comment}
Define $\tilde{s}_c >0$ and $\tilde{\cal W}_c \subseteq \mathbb{R}^n$ such that 
\begin{align}
\tilde{s}_c^\mathsf{T} &\leq {p}^\mathsf{T}(K+L) \\ 
\tilde{\cal W}_c &=\{x\in \mathbb{R}^n: {x}^\mathsf{T}\tilde{S}_c x - \tilde{r}^\mathsf{T} _c x \leq 0 \}, 
\end{align}
with $\tilde{S}_c = {\rm diag}(\tilde{s}_c)$.
\end{comment}
To prove \rprop{gpproposition}, let us first provide the following lemma, which gives a sufficient condition for the control objective \req{objective} to be achieved under the continuous-time controller. 
\begin{mylem}\label{emulationlemma}
\normalfont
Consider the SIS model \req{controldynamics}, the continuous-time controller \req{uitimecon}, \req{vijtimecon}, and the control objective \req{objective}. Assume that the control gains  satisfying \req{constraintcontrolparameters} are chosen such that the following conditions are satisfied: 
\begin{align}
\tilde{\theta}^* _c \le p_m^*\bar{d}_m, \label{thetaeq}
\end{align}
for all $m \in \{1, \ldots, M\}$, where $\tilde{\theta}^* _c \in \mathbb{R}$ is defined according to the following optimization problem: 
\begin{align}
\tilde{\theta}^* _c = &\ \underset{x\in {\cal W}_c}{\rm max}\ {p}^\mathsf{T} {x} \label{optimizationcontinuous} %\notag \\ 
%&{\rm subject\ to\ }\ {x}^\mathsf{T}\tilde{S}_c x - \tilde{r}^\mathsf{T} _c x \leq 0, \label{optimizationcontinuous}
\end{align}
where ${\cal W}_c =\{x\in \mathbb{R}^n: {x}^\mathsf{T}\tilde{S}_c x - (\tilde{r}_c + \epsilon \mathsf{1}_n)^\mathsf{T} x \leq 0 \}$ with $\epsilon>0$, $\tilde{S}_c = {\rm diag}(\tilde{s}_c)$ and $\tilde{s}_c \in \mathbb{R}^n$ is chosen such that the following inequality is satisfied:
\begin{align}\label{stilde}
0 < \tilde{s}_c^\mathsf{T} &\leq {p}^\mathsf{T}(K+L). 
\end{align}
Then, for any $x(0) \in [0, 1]^n$, the control objective \req{objective} is achieved by applying the continuous-time controller. \qedwhite 
\end{mylem}
%The proof is similar to the one of \rthm{mainresult} and is thus provided in the Appendix. 
%\section{Proof of \rlem{emulationlemma}}

\textit{(Proof of \rlem{emulationlemma}):} 
%\begin{proof}
The proof is given based on \rthm{mainresult}. 
Applying the continuous-time controller \req{uitimecon}, \req{vijtimecon} is equivalent to setting the event-triggering gains as $\sigma_i \rightarrow 0$ and $\eta_i \rightarrow 0$ for all $i \in \mathcal{V}$ (i.e., $G\rightarrow 0$ and $H \rightarrow 0$). Hence, by setting $G \rightarrow 0,  H \rightarrow 0$ in \req{xdoteq}, the upper bound of $\dot{x} (t)$ under the continuous-time controller is computed as 
\begin{align}
\dot{x}(t)\le&\left (\bar{B}-\ubar{D}\right) {x}(t)-(K+L) \tilde{ {x}}(t). 
\end{align}
The derivative of the Lyapunov function $V( {x}) = p^\mathsf{T} x$ (under the continuous-time controller) is then given by 
\begin{align}
\cfrac{{\rm d}}{{\rm d}t} {V}( {x}(t)) &= p^\mathsf{T} \dot{x}(t)\leq p^\mathsf{T} \left (\bar{B}-\ubar{D}\right)x(t) - p^\mathsf{T}(K+L) \tilde{ {x}}(t) \\
&\leq \tilde{r}^\mathsf{T} _c {x}(t)- {x}^\mathsf{T}(t)\tilde{S}_c  {x}(t), \label{upperboundcontinuous}
\end{align}
where we used $\tilde{r}^\mathsf{T} _c = p^\mathsf{T} \left (\bar{B}-\ubar{D}\right)$ (see the statement after \req{rtildeci}) and \req{stilde} to obtain \req{upperboundcontinuous}. 

Now, similarly to \req{thetaprop}, 
%from \req{optimizationcontinuous}, 
it follows that ${x}^\mathsf{T}\tilde{S}_c {x}-(\tilde{r}_c + \epsilon \mathsf{1}_n)^\mathsf{T} {x} <0 \implies p^\mathsf{T} x < \tilde{\theta}^* _c$ for all $x \in \mathbb{R}^n$.
%$p^\mathsf{T} x \geq \tilde{\theta}^* _c \implies {x}^\mathsf{T}\tilde{S}_c {x}-(\tilde{r}_c + \epsilon \mathsf{1} )^\mathsf{T} {x} \geq 0$ for all $x \in \mathbb{R}^n$. 
Hence, by taking the contrapositive, we obtain
\begin{equation}
V( {x}(t)) \geq \tilde{\theta}^* _c \implies \cfrac{{\rm d}}{{\rm d}t} {V}( {x}(t)) \leq \tilde{r}^\mathsf{T} _c {x}(t)- {x}^\mathsf{T}(t)\tilde{S}_c {x}(t) \leq - \epsilon \|x\|_1. \label{lyapunovdecreasecontinuous}
\end{equation}
 %\req{lyapunovdecreasecontinuous} implies that, as long as $V(x(t))$ is greater than $\tilde{\theta}^* _c$, it ensures that $V(x(t))$ is strictry decreasing in $t$. 
 Thus, by following the same procedure as  \req{lyapunovconverge}--\req{lyapunovconverge2}, we can show that there exists $t' \geq 0$ such that ${w}_m^T {x}(t)\le \bar{d}_m$ for all $t \geq t'$ and $m \in \{1, \ldots, M\}$. Therefore, if \req{thetaeq} holds for all $m \in \{1, \ldots, M\}$, the control objective \req{objective} is achieved by applying the continuous-time controller \req{uitimecon}, \req{vijtimecon}. \qedwhite 
 %\end{proof}

\smallskip
%We now prove \rprop{gpproposition}.
%From \rlem{emulationlemma}, the control objective \req{objective} is achieved by finding the set of control gains , such that \req{thetaeq} is satisfied for all $m \in \{1, \ldots, M\}$. Note that, in addition to the set of the control gains , the parameter $\tilde{s}_c > 0$ should be regarded as a decision variable that must satisfy the inequality \req{stilde} in the design problem. 
%Note also that $\tilde{r}_c = p^\mathsf{T} \left (\bar{B}-\ubar{D}\right)$ is constant, since the matrix $\bar{B}$, $\ubar{D}$ and the Lyapunov parameter $p$ are all given (therefore, the parameter $\tilde{r}_c$ is not regarded as a decision variable in the design problem).

%Note that the control gains  and $\tilde{s}_c >0_n$ must satisfy \req{constraintcontrolparameters} and \req{stilde}, respectively. 
%must satisfy \req{constraintcontrolparameters}, and that the parameter $\tilde{s}_c >0_n$ used in the optimization problem \req{optimizationcontinuous} must satisfy \req{stilde}. 
%Hence, \req{constraintcontrolparameters}, \req{thetaeq} and \req{stilde} should be given as the required conditions (or constraints) when designing the control gains . 
%and the parameter $\tilde{s}_c >0_n$ used in the optimization problem \req{optimizationcontinuous} must fulfill the conditions  \req{constraintcontrolparameters} and \req{stilde}, respectively. 

We now prove \rprop{gpproposition}. 

\smallskip
\textit{(Proof of \rprop{gpproposition})}: 
%\begin{proof}
%From \rlem{emulationlemma}, the control objective is achieved by finding the set of control gains , such that \req{thetaeq} is satisfied for all $m \in \{1, \ldots, M\}$. Additionally, the control gains  and the parameter $\tilde{s}_c >0_n$ used in the optimization problem \req{optimizationcontinuous} must fulfill the conditions  \req{constraintcontrolparameters} and \req{stilde}, respectively. Hence, the control objective is achieved by finding $k_i, l_{ij} >0$ and $\tilde{s}_c = [\tilde{s}_{c, 1}, \ldots, \tilde{s}_{c, n}] > 0$, such that the constraints \req{constraintcontrolparameters}, \req{thetaeq}, \req{stilde} are satisfied. 
From \rlem{emulationlemma}, it is shown that the control objective \req{objective} is achieved by finding control gains $k_i, l_{ij}$, $i\in {\cal V}$, $j \in {\cal N}^{\rm out} _i$ and a vector $\tilde{s}_c = [\tilde{s}_{c, 1}, \ldots, \tilde{s}_{c, n}]$, such that the following conditions are all satisfied: \req{constraintcontrolparameters} for all $i\in {\cal V}$ and $j \in {\cal N}^{\rm out} _i$, \req{thetaeq} for all $m \in \{1, \ldots, M\}$ and \req{stilde}. 
Thus, the decision variables in the design problem here are: $k_i, l_{ij}$, $i\in {\cal V}$, $j \in {\cal N}^{\rm out} _i$ and $\tilde{s}_c= [\tilde{s}_{c, 1}, \ldots, \tilde{s}_{c, n}]$. Moreover, the constraints are:  \req{constraintcontrolparameters} for all $i\in {\cal V}$, $j \in {\cal N}^{\rm out} _i$, \req{thetaeq} for all $m \in \{1, \ldots, M\}$ and \req{stilde}. %Note that the vector $\tilde{r}_c = p^\mathsf{T} \left (\bar{B}-\ubar{D}\right)$ used in \rlem{emulationlemma} is \textit{constant}; see \req{rtildeci}. 
In what follows, we translate the conditions \req{constraintcontrolparameters}, \req{thetaeq}, and \req{stilde} into the \textit{posynomial constraints}, so that the control gains can be designed by solving a geometric programming problem. 
%satisfying \req{constraintcontrolparameters} 
%$k_i \in (0, \bar{k}_i]$, $l_{ij} \in (0, \bar{l}_{ij}]$ and $\tilde{s}_c = [\tilde{s}_{c,1}, \ldots, \tilde{s}_{c,n}]^\mathsf{T} >0$, such that \req{stilde} and \req{thetaeq} are satisfied. 
First, the constraint for the $i$-th element of $\tilde{s}_c$ (i.e., $\tilde{s}_{c, i}$) in \req{stilde} 
%The constraint of \req{stilde} corresponding to the $i$-th ($i \in {\cal V}$) element of $\tilde{s}_c$, i.e., $\tilde{s}_{c, i}$, 
is given by 
%From \req{stilde}, the constraint corresponding to the $i$-th ($i\in \mathcal{V}$) element of $\tilde{s}_c$ is given by
\begin{align}
& \tilde{s}_{c, i} \leq p_i k_i + \sum_{j \in \mathcal{N}_i^{\rm out}} p_j l_{ij}. %\Longleftrightarrow & \tilde{s}_{c, i} + p_i (\bar{k}_i - k_i) + \sum_{j \in \mathcal{N}_i^{\rm out}} p_j (\bar{l}_{ij} - l_{ij}) \leq c_{1,i}
\label{stildeel}
\end{align}
 Summing $p_i \bar{k}_i + \sum_{j\in{\cal N}_i ^{\rm out}} p_j \bar{l}_{ij}$ in both sides of \req{stildeel}, we obtain 
\begin{align}\label{stildeeq}
\tilde{s}_{c, i} + p_i (\bar{k}_i - k_i) + \sum_{j \in \mathcal{N}_i^{\rm out}} p_j (\bar{l}_{ij} - l_{ij}) \leq c_{1,i}, 
\end{align}
where $c_{1,i}$ is the constant defined in \req{ctwoi}. 
Define the new decision variables $\widetilde{k}_i$, $\widetilde{l}_{ij} >0$ for all $j\in\mathcal{N}_i^{\rm out}$ by 
\begin{align}
\tilde{k}_i = \bar{k}_i - k_i,\ \ \tilde{l}_{ij} = \bar{l}_{ij} - l_{ij}. 
\end{align}
%for all $i\in\mathcal{V}_i$ and $j\in\mathcal{N}_i^{\rm out}$, 
Then, the constraint \req{stildeeq} results in \req{gpthree}, which is the posynomial constraint (with respect to the variables $\tilde{k}_i$, $\tilde{l}_{ij}$, $\tilde{s}_{c, i}$). With the new decision variables $\tilde{k}_i$, $\tilde{l}_{ij}$ ($j\in\mathcal{N}_i^{\rm out}$), the constraint \req{constraintcontrolparameters} is then rewritten as $\tilde{k}_i < \bar{k}_i$ and $\tilde{l}_{ij} < \bar{l}_{ij}$ for all $j\in\mathcal{N}_i^{\rm out}$. This leads to \req{gpfour}, \req{gpfive}, which are the posynomial constraints. 
%Even though \req{stildeel} is not the posynomial inequality in $k_i$ and $l_{ij}$, it can be transformed into the posynomial constraint by introducing the new positive variables $\tilde{k}_i > 0$ and $\tilde{l}_{ij} > 0$ as $\tilde{k}_i = \bar{k}_i - k_i$ and $\tilde{l}_{ij} = \bar{l}_{ij} - l_{ij}$, which results in \req{gpthree}.
Let us now translate \req{thetaeq} into the posynomial constraint. %and the optimization problem \req{optimizationcontinuous}. 
First, observe that 
%let ${{\mathcal{W}}_c}\subset \mathbb{R}^n$ be an ellipsoidal set given by 
\begin{align}\label{ellipsoidcdef}
{{\mathcal{W}}_c} &=\{x\in \mathbb{R}^n: x^\mathsf{T}\tilde{S}_c x-(\tilde{r}_c + \epsilon \mathsf{1}_n)^\mathsf{T}x \le 0\} \notag \\ 
                     &=\{x\in \mathbb{R}^n: (x - {g}_c )^\mathsf{T}\tilde{S}_c (x - {g}_c )  \le {\rho}_c^2\}, 
\end{align}
%which can be characterized by the center and the unit ball ($\{v\in \mathbb{R}^n:\|v\|\le 1\}$)  under an affine transformation as
where 
\begin{align}
{g}_c  = \left [\frac{\tilde{r}_{c,1}+\epsilon}{2\tilde{s}_{c,1}},\ldots, \frac{\tilde{r}_{c,n}+\epsilon}{2\tilde{s}_{c,n}}\right]^\mathsf{T}, \ \  
{\rho}_c = \sqrt{\sum_{i\in \mathcal{V}} \frac{(\tilde{r}_{c,i}+\epsilon)^2}{4{\tilde{s}_{c,i}}}}.
\end{align} 
Hence, we obtain 
\begin{align}\label{ellipsoidcdefaffine}
{\mathcal{W}}_c & = \left \{x\in \mathbb{R}^n: \left \|{{\tilde{S}^{1/2}_c}}(x - {g}_c )/{\rho}_c \right \|  \le 1\right \}, \notag \\ 
& = \left \{ {F}_c v + {g}_c  \in \mathbb{R}^n: \|v\| \le 1\right \}, 
\end{align}
where $\tilde{S}^{1/2}_c = {\rm diag}\left (\left[\sqrt{\tilde{s}_{c,1}}, \sqrt{\tilde{s}_{c,2}}, \ldots, \sqrt{\tilde{s}_{c,n}}\ \right]\right )$ and ${F}_c= {\rho}_c \tilde{S}^{-1/2}_c$. 
Thus, the optimization problem \req{optimizationcontinuous} can be analytically solved as: 
%Therefore, $\tilde{\theta}^* _c$ in \req{optimizationcontinuous} satisfies:
%the optimization problem \req{optimizationcontinuous} is analytically solved as 
\begin{align}
\tilde{\theta}^* _c &= \max_{x \in {{\mathcal{W}}_c}} p^\mathsf{T} x = \max_{\|v\|\leq 1} p^\mathsf{T} ({F}_c v + {g}_c ) =\|p^\mathsf{T} {F}_c\| + p^\mathsf{T}{{g}_c}  \notag\\ 
& = \frac{1}{2}\sqrt{\sum_{i \in \mathcal{V}}\frac{p_i^2}{\tilde{s}_{c,i}}}\sqrt{\sum_{i\in \mathcal{V}}\frac{(\tilde{r}_{c,i}+\epsilon)^2}{\tilde{s}_{c,i}}} + \frac{1}{2}\sum_{i\in \mathcal{V}} \frac{p_i(\tilde{r}_{c,i}+\epsilon)}{\tilde{s}_{c,i}}. \label{thetastaranalytic}
\end{align}

\begin{comment}
\begin{align}
\tilde{\theta}^* _c &= \underset{x\in [0, 1]^n}{\rm max}\ \{ {p}^\mathsf{T} {x} \in \mathbb{R} : {x}^\mathsf{T}\tilde{S}_c x - \tilde{r}^\mathsf{T} _c x \leq 0\} \notag \\
&\leq \underset{x\in \mathbb{R}^n}{\rm max}\ \{ {p}^\mathsf{T} {x} \in \mathbb{R} : {x}^\mathsf{T}\tilde{S}_c x - \tilde{r}^\mathsf{T} _c x \leq 0\} \notag \\
&= \max_{x \in {{\mathcal{W}}_c}} p^\mathsf{T} x = \max_{\|v\|\leq 1} p^\mathsf{T} ({F}_c v + {g}_c ) =\|p^\mathsf{T} {F}_c\| + p^\mathsf{T}{{g}_c}  \notag\\ 
& = \frac{1}{2}\sqrt{\sum_{i \in \mathcal{V}}\frac{p_i^2}{\tilde{s}_{c,i}}}\sqrt{\sum_{i=1}^n\frac{\tilde{r}_{c,i}^2}{\tilde{s}_{c,i}}} + \frac{1}{2}\sum_{i=1}^n \frac{p_i\tilde{r}_{c,i}}{\tilde{s}_{c,i}}
\end{align}
\end{comment}

Therefore, the constraint \req{thetaeq} is given by 
%from the assumption of \rlem{emulationlemma}, our control objective can be achieved if there exists feasible solutions in the following constraints.
\begin{align}\label{gaincalconst}
  &\sqrt{\sum_{i \in \mathcal{V}}\frac{p_i^2}{\tilde{s}_{c,i}}}\sqrt{\sum_{i\in \mathcal{V}}\frac{(\tilde{r}_{c,i}+\epsilon)^2}{\tilde{s}_{c,i}}} + \sum_{i\in \mathcal{V}} \frac{p_i(\tilde{r}_{c,i}+\epsilon)}{\tilde{s}_{c,i}} \le 2p_m^*\bar{d}_m \Longleftrightarrow \notag \\ 
& \sqrt{\sum_{i \in \mathcal{V}}\frac{p_i^2}{\tilde{s}_{c,i}}}\sqrt{\sum_{i\in \mathcal{V}}\frac{(\tilde{r}_{c,i}+\epsilon)^2}{\tilde{s}_{c,i}}} + \sum_{i\in\mathcal{C}}\frac{p_i\tilde{r}_{c,i}}{\tilde{s}_{c,i}} +\sum_{i\in \mathcal{V}} \frac{\epsilon p_i}{\tilde{s}_{c,i}} \le 2p_m^*\bar{d}_m + \sum_{i\notin\mathcal{C}}\frac{-p_i\tilde{r}_{c,i}}{\tilde{s}_{c,i}},
\end{align}
where the set ${\cal C}$ is defined in \req{positivec}. Note that \req{gaincalconst} is not yet a posynomial constraint, since the term $\sum_{i\notin\mathcal{C}}\frac{-p_i\tilde{r}_{c,i}}{\tilde{s}_{c,i}}$ in the right hand side involves the variable $\tilde{s}_{c,i}$. 
In order to convert \req{gaincalconst} into the posynomial constraint, notice that, from \req{stildeeq}, we have   $\tilde{s}_{c,i}\le c_{1,i}$ for all $i\in\mathcal{V}$.
Hence, we obtain $\frac{-p_i\tilde{r}_{c,i}}{c_{1,i}} \leq \frac{-p_i\tilde{r}_{c,i}}{\tilde{s}_{c,i}}$ for all $i \notin {\cal C}$. 
Thus, a sufficient condition to satisfy \req{gaincalconst} is given by 
\begin{align}
\sqrt{\sum_{i \in \mathcal{V}}\frac{p_i^2}{\tilde{s}_{c,i}}}\sqrt{\sum_{i\in \mathcal{V}}\frac{(\tilde{r}_{c,i}+\epsilon)^2}{\tilde{s}_{c,i}}} + \sum_{i\in\mathcal{C}}\frac{p_i\tilde{r}_{c,i}}{\tilde{s}_{c,i}} + \sum_{i\in \mathcal{V}} \frac{\epsilon p_i}{\tilde{s}_{c,i}} 
&\le 2p_m^*\bar{d}_m +\sum_{i\notin\mathcal{C}}\frac{-p_i\tilde{r}_{c,i}}{c_{1,i}} \notag \\ 
&\ \ \ \ = c_{2, m}, \label{gaincalconst3} %\label{gaincalconst3}
\end{align}
for all $m \in \{1, \ldots, M\}$.
%Notice that $\tilde{r}_{c,i}$ is a constant that take a negative (resp. positive) value correspoiding to the index $i\notin\mathcal{C}$ (resp. $i\in \mathcal{C}$) from \req{eq32}\req{positivec}. 
%Considering any posynomial constraints are composed of a sum of one or more monomials, the coefficients in \req{gaincalconst} must be positive, so it should be transformed as 
%\begin{align}\label{gaincalconstc}
%\sqrt{\sum_{i \in \mathcal{V}}\frac{p_i^2}{\tilde{s}_{c,i}}}\sqrt{\sum_{i=1}^n\frac{c_{1,i}^2}{\tilde{s}_{c,i}}} + \sum_{i\in\mathcal{C}}\frac{p_ic_{1,i}}{\tilde{s}_{c,i}} \le 2p_m^*\bar{d}_m + \sum_{i\notin\mathcal{C}}\frac{-p_ic_{1,i}}{\tilde{s}_{c,i}},
%\end{align}
%\req{gaincalconstc3} can be translated into the posynomial constraints as follows: defining 
Define a new variable ${\xi}_{c} >0$, such that the following constraint is satisfied: 
\begin{align}
\left({\sum_{i \in \mathcal{V}}\frac{p_i^2}{\tilde{s}_{c,i}}}\right )\left({\sum_{i\in \mathcal{V}}\frac{(\tilde{r}_{c,i}+\epsilon)^2}{\tilde{s}_{c,i}}}\right )\leq  {\xi}_{c}. \label{xieq}
\end{align}
%Using the variable ${\xi}_{c} >0$, 
%Then, \req{gaincalconst3} becomes 
Thus, we have
\begin{align}
{\xi}^{\frac{1}{2}} _c + \sum_{i\in\mathcal{C}}\frac{p_i (\tilde{r}_{c,i}+\epsilon)}{\tilde{s}_{c,i}} \le c_{2, m},  \label{gaincalconst4}
\end{align}
which leads to the posynomial constraints \req{gpone}, \req{gptwo}. 

In summary, the control objective is achieved by finding $\tilde{k}_i, \tilde{l}_{ij} >0$, $i \in \mathcal{V}$, $j \in {\cal N}^{\rm out} _i$, $\tilde{s}_c= [\tilde{s}_{c, 1}, \ldots, \tilde{s}_{c, n}] >0$ and $\epsilon_1, \epsilon_2, \epsilon_3, {\xi}_{c}>0$, such that the posynomial constraints \req{gpthree}--\req{gptwo} are satisfied. Since $\tilde{k}_i = \bar{k}_i - k_i$, $\tilde{l}_{ij} = \bar{l}_{ij} - l_{ij}$, the (optimal) control parameters are given by \req{kilij} after solving the geometric programming problem shown in \rprop{gpproposition}. 
%The proof is complete. 

%\end{proof}

\section{Proof of \rprop{gppropositione}} \label{appproof}
%\begin{comment}
%Similarly to \rlem{emulationlemma}, we can prove the following result:
Before proving \rprop{gppropositione}, let us show that the following lemma holds as a slight modification of \rthm{mainresult}: 
\begin{mylem}\label{emulationevlemma}
\normalfont
Consider the SIS model \req{controldynamics}, the event-triggered controller \req{ui}, \req{vij}, and the control objective \req{objective}. Assume that the control and the event-triggering gains satisfying \req{conditioneventtrig} and \req{constraintcontrolparameters} 
are chosen such that the following conditions are satisfied: 
%Suppose that the set of control gains  $k^* _i$, $l^* _{ij}$, $i \in \mathcal{V}$, $j \in {\cal N}_i^{\rm out}$ are designed according to \rprop{gpproposition}, and that the event-triggering gains $\sigma_i \in (0, 1)$, $\eta_i \in (0,1)$, $i \in {\cal V}$ are chosen such that the following conditions are satisfied: 
\begin{equation}
\tilde{\theta}^* _e \le p_m^*\bar{d}_m,  
\label{thetastarlemma}
\end{equation}
for all $m \in \{1, \ldots, M\}$, where $\tilde{\theta}^* _e \in \mathbb{R}$ is defined according to the following optimization problem: 
\begin{align}
\tilde{\theta}^* _e = &\ \underset{x\in {\cal W}_e}{\rm max}\ {p}^\mathsf{T} {x} \label{optimizationeventtrig} %\notag \\ 
%&{\rm subject\ to\ }\ {x}^\mathsf{T}\tilde{S}_c x - \tilde{r}^\mathsf{T} _c x \leq 0, \label{optimizationcontinuous}
\end{align}
where ${\cal W}_e =\{x\in \mathbb{R}^n: {x}^\mathsf{T}\tilde{S}_e x - (\tilde{r}_e + \epsilon \mathsf{1}_n)^\mathsf{T} x \leq 0 \}$ with $\epsilon>0$, $\tilde{S}_e = {\rm diag}(\tilde{s}_e)$, and the vectors $\tilde{r}_e, \tilde{s}_e \in \mathbb{R}^n$ are defined such that the following inequalities satisfied:
%$\tilde{s}_e = [\tilde{s}_{e,1}, \ldots, \tilde{s}_{e,n}]^\mathsf{T}$ and $\tilde{r}_e =[\tilde{r}_{e,1}, \ldots, \tilde{r}_{e, n}]^\mathsf{T} \in \mathbb{R}^n$ be the vectors %satisfying 
\begin{align}
&0 < \tilde{s}_e^\mathsf{T} \leq {p}^\mathsf{T}(K+L)(I_n-G) \label{setilde}\\
&\tilde{r}^\mathsf{T}_e  \geq {p}^\mathsf{T}\{\bar{B}-\ubar{D}+(K+L)H\}. \label{retilde}
\end{align}

%where $\tilde{S}_e = {\rm diag}(\tilde{s}_e)$, $\tilde{s}_e$, and $\tilde{r}_e$ are defined satisfying \req{stilde}, \req{rtilde}, respectively. 
Then, for any $x(0) \in [0, 1]^n$, the control objective \req{objective} is achieved by applying the event-triggered controller. \qedwhite 
\end{mylem}
%\rlem{emulationevlemma} differs from \rthm{mainresult}, in the sense that we make use of the inequalities in \req{setilde} and \req{retilde}, and that the matrix $\tilde{S}_e$ in the optimization problem \req{optimizationcontinuous} is diagonal. 
As shown in the proof below, 
\rlem{emulationevlemma} is a 
more conservative result than \rthm{mainresult}, in the sense that it provides \textit{sufficient conditions} to the ones provided in \rthm{mainresult}. However, \rlem{emulationevlemma} is useful for translating the conditions required to achieve the control objective into the posynomial constraints (see the proof of \rprop{gppropositione} below). 

%\begin{proof}
\textit{(Proof of \rlem{emulationevlemma}):} 
%Fix the control gains  as $k_i = k^* _i$ and $l_{ij} = l^* _{ij}$ for all $i \in \mathcal{V}$ and $j \in {\cal N}_i^{\rm out}$ (i.e., $K = K^*, L = L^*$). 
From \rthm{mainresult}, it is shown that the control objective \req{objective} is achieved by finding the the control and the event-triggering gains, such that \req{thetastar} is satisfied for all $m \in \{1, \ldots, M\}$. From \req{svec}, \req{rvec}, and \req{shat}, it follows that 
\begin{align}\label{constrelax}
&{x}^\mathsf{T}Qx - r^\mathsf{T} x \geq {x}^\mathsf{T} S x - r^\mathsf{T} x \nonumber\\
&= {x}^\mathsf{T}{\rm diag}(p^\mathsf{T}(K + L)(I_n -G))x - p^\mathsf{T}\{\bar{B} - \ubar{D} + (K + L) H\} x \nonumber\\ 
&\geq x^\mathsf{T} \tilde{S}_e x - \tilde{r}_e^\mathsf{T} x, 
\end{align}
where we used \req{setilde} and \req{retilde} to obtain \req{constrelax}.
%where we set $K = K^*$ and $L = L^*$ in \req{svec}, and $\tilde{S}_e = {\rm diag}(\tilde{s}_e)$, $\tilde{s}_e$, and $\tilde{r}_e$ are given such that \req{setilde} and \req{retilde} are satisfied. 
Hence, it follows that $x^\mathsf{T} \tilde{S}_e x - (\tilde{r}_e + \epsilon \mathsf{1}_n)^\mathsf{T} x >0 \implies {x}^\mathsf{T}Qx - (r+ \epsilon \mathsf{1}_n)^\mathsf{T} x >0$ for all $x \in \mathbb{R}^n$. 
Therefore, letting ${\cal W}_e =\{x\in \mathbb{R}^n: x^\mathsf{T}\tilde{S}_e x-(\tilde{r}_e+ \epsilon \mathsf{1}_n)^\mathsf{T}x \le 0\} $, we have ${\cal W} \subseteq {\cal W}_e$ (recall that ${\cal W}$ is defined in \req{calw}). Hence, $\theta^*$ in \req{optimization} satisfies 
\begin{align}
\theta^* = \underset{x \in {\cal W}}{\rm max}\ p^\mathsf{T} x &\leq \underset{x \in {\cal W}_e}{\rm max}\ p^\mathsf{T} x = \tilde{\theta}^* _e. 
\end{align}
Thus, $\tilde{\theta}^* _e \leq p^* _m \bar{d}_m$ for all $m\in \{1, \ldots, M\}$ implies $\theta^* \leq p^* _m \bar{d}_m$ for all $m\in \{1, \ldots, M\}$. 
Therefore, if the control and the event-triggering gains are chosen such that \req{optimizationeventtrig} is satisfied for all $m \in \{1, \ldots, M\}$, the control objective \req{objective} is achieved by applying the event-triggered controller. \qedwhite 
%Hence, if the event-triggering gains are chosen such that \req{thetastarlemma} is satisfied, the control objective is achieved by applying the event-triggered controller. \qedwhite 
%\end{proof}

Let us now prove \rprop{gppropositione}.

\smallskip
\textit{(Proof of \rprop{gppropositione}):} 
Fix the control gains  by $k_i = k^* _i$, $l_{ij} = l^* _{ij}$ for all $i \in \mathcal{V}$ and $j \in {\cal N}_i^{\rm out}$. 
From \rlem{emulationevlemma}, it is shown that the control objective \req{objective} is achieved by finding event-triggering gains $\sigma_i$, $\eta_i$, and vectors $\tilde{r}_e= [\tilde{r}_{e, 1}, \ldots, \tilde{r}_{e, n}]$, $\tilde{s}_e= [\tilde{s}_{e, 1}, \ldots, \tilde{s}_{e, n}]$, such that the following conditions are all satisfied: \req{conditioneventtrig} for all $i \in {\cal V}$, \req{thetastarlemma} for all $m \in \{1, \ldots, M\}$, \req{setilde} and \req{retilde}. 
Thus, the decision variables in the design problem here are:  $\sigma_i$, $\eta_i$ for all $i \in {\cal V}$, $\tilde{r}_e= [\tilde{r}_{e, 1}, \ldots, \tilde{r}_{e, n}]$ and $\tilde{s}_e= [\tilde{s}_{e, 1}, \ldots, \tilde{s}_{e, n}]$. Moreover, the constraints are: \req{conditioneventtrig} for all $i \in {\cal V}$, \req{thetastarlemma} for all $m \in \{1, \ldots, M\}$, \req{setilde} and \req{retilde}. In what follows, we translate the conditions \req{conditioneventtrig}, \req{thetastarlemma}, \req{setilde}, and \req{retilde} into the \textit{posynomial constraints}, so that the event-triggering gains can be designed by solving a geometric programming problem. First, the constraint for the $i$-th element of $\tilde{s}_e$ (i.e., $\tilde{s}_{e, i}$) in \req{setilde} is given by  \begin{align}\label{setildeel}
\tilde{s}_{e, i} \leq (p_i k_i^* + \sum_{j \in \mathcal{N}_i^{\rm out}} p_j l_{ij}^*)(1 - \sigma_ i) = c_{3, i} (1 - \sigma _i), 
\end{align}
where $c_{3, i} >0$ is the constant defined in \req{const3}. Thus, defining the new variables by $\tilde{\sigma}_i = 1- \sigma_i$ for all $i \in {\cal V}$, we obtain the following constraint: 
%\req{setildeel} results in 
\begin{align}
{\tilde{s}_{e, i}}{\tilde{\sigma}^{-1}_i}\leq  c_{3, i}, 
\end{align}
which yields the posynomial constraint \req{propositionconstone}. 
Moreover, the constraint for the $i$-th element of $\tilde{r}_e$ (i.e., $\tilde{r}_{e,i}$) in \req{retilde} is given by 
%the constraint corresponding to the $i$-th element of $\tilde{r}_e$, i.e., $\tilde{r}_{e,i}$ for all $i\in\mathcal{V}$, is given by
\begin{align}\label{retildeel}
\tilde{r}_{e,i} \ge \sum_{j\in \mathcal{N}^{\rm out}_i} p_j\bar \beta_{ij}-p_i\ubar{\delta}_i + (p_i k_i^* + \sum_{j\in \mathcal{N}^{\rm out}_i} p_j l_{ij}^*)\eta_i = \tilde{r}_{c, i} + c_{3, i}\eta_i.
\end{align}
where $\tilde{r}_{c, i}$ is the constant defined in \req{rtildeci}. If $i \in {\cal C}$, we have $\tilde{r}_{c, i} \geq 0$ and so $\tilde{r}_{e,i} \ge \tilde{r}_{c, i} + c_{3, i}\eta_i >0$. This implies that \req{retildeel} is the posynomial constraint. 
%with respect to $\tilde{r}_{e,i} >0$ and $\eta_i>0$. %with the variables $\tilde{r}_{e,i}$ and $\eta_i >0$. 
Thus, we obtain $\tilde{r}_{c, i} + c_{3, i}\eta_i \leq \tilde{r}_{e,i}$ for all $i \in {\cal C}$, which yields the posynomial constraint \req{propositionconstone}. 
If $i \notin {\cal C}$, we have $\tilde{r}_{c, i} + c_{3, i}\eta_i = 0$, since $\eta_i = \frac{{-\tilde{r}_{c,i}}}{c_{3, i}}$ (see \req{etanotinc}). Note that we have $\eta_i = \frac{{-\tilde{r}_{c,i}}}{c_{3, i}} \in (0, 1)$ from \ras{assumption}. Thus, the constraint \req{retildeel} for all $i \notin {\cal C}$ is trivially given by $\tilde{r}_{e,i} \geq 0$, $i \notin {\cal C}$. 
%In summary, the conditions for $\tilde{r}_{e,i}$ is 
%\begin{align}
%&\tilde{r}_{c, i} + c_{3, i}\eta_i \leq \tilde{r}_{e,i} \\ 
%&\tilde{r}_{e,i} \geq 0
%\end{align}

Using the new decision variable $\tilde{\sigma}_i$, the constraint of $\sigma_i$ (i.e., $\sigma_i \in (0, 1)$) is rewritten as $\tilde{\sigma}_i \in (0, 1)$. This leads to the posynomial constraint \req{propositionconstthree}. 
%Moreover, since $\sigma_i \in (0, 1)$, we have $\tilde{\sigma}_i = 1- \sigma_i \in (0, 1)$, which leads to the posynomial constraint \req{propositionconstthree}. 
The constraint for $\eta_i$ is given by \req{propositionconstfour}, which follows trivially.

Let us now translate \req{thetastarlemma} into a posynomial constraint. As with \req{thetastaranalytic}, the optimization \req{optimizationeventtrig} is analytically solved as 
%we obtain 

\begin{align}
\tilde{\theta}^* _e & = \underset{x \in {\cal W}_e}{\rm max}\ p^\mathsf{T} x = \max_{\|v\|\leq 1} p^\mathsf{T} ({F}_e v + {g}_e ) =\|p^\mathsf{T} {F}_e \| + p^\mathsf{T}{{g}_e}  \notag\\ 
& = \frac{1}{2}\sqrt{\sum_{i \in \mathcal{V}}\frac{p_i^2}{\tilde{s}_{e,i}}}\sqrt{\sum_{i\in \mathcal{V}}\frac{(\tilde{r}_{e,i}+\epsilon)^2}{\tilde{s}_{e,i}}} + \frac{1}{2}\sum_{i\in \mathcal{V}} \frac{p_i(\tilde{r}_{e,i}+\epsilon)}{\tilde{s}_{e,i}}, 
\end{align}
where ${F}_e= {\rho}_e \tilde{S}^{-1/2}_e$ with ${\rho}_e = \sqrt{\sum_{i\in \mathcal{V}} \frac{(\tilde{r}+\epsilon)^2_{e,i}}{4{\tilde{s}_{e,i}}}}$ and ${g}_e  = \left [\frac{\tilde{r}_{e,1}+\epsilon}{2\tilde{s}_{e,1}},\ldots, \frac{\tilde{r}_{e,n}+\epsilon}{2\tilde{s}_{e,n}}\right]^\mathsf{T}$. Hence, the constraint \req{thetastarlemma} is given by 
\begin{align}\label{constraintpd}
\frac{1}{2}\sqrt{\sum_{i \in \mathcal{V}}\frac{p_i^2}{\tilde{s}_{e,i}}}\sqrt{\sum_{i\in \mathcal{V}}\frac{(\tilde{r}_{e,i}+\epsilon)^2}{\tilde{s}_{e,i}}} + \frac{1}{2}\sum_{i\in \mathcal{V}} \frac{p_i(\tilde{r}_{e,i}+\epsilon)}{\tilde{s}_{e,i}} \leq p_m^*\bar{d}_m
\end{align}
Define the new variable ${\xi}_{e} >0$ such that the following constraint is satisfied: 
\begin{align}
\left({\sum_{i\in {\cal V}}\frac{p_i^2}{\tilde{s}_{e,i}}}\right )\left(\sum_{i\in {\cal V}}\frac{(\tilde{r}_{e,i}+\epsilon)^2}{\tilde{s}_{e,i}}\right )\leq  {\xi}_{e}. \label{xiconst2}
\end{align}
Then, the constraint \req{constraintpd} becomes
 \begin{align}\label{constraintpd3}
{\xi}^{\frac{1}{2}}_{e} + \sum_{i\in {\cal V}} \frac{p_i(\tilde{r}_{e,i}+\epsilon)}{\tilde{s}_{e,i}} \leq 2 p_m^*\bar{d}_m, 
\end{align}
%Note that \req{xiconst2} and \req{constraintpd3} are the posynomial constraints in accordance with 
which leads to the posynomial constraints \req{propositionconstfour},  \req{propositionconstfive}. 

In summary, the control objective is achieved by finding $\tilde{\sigma}_i$, $\tilde{s}_{e,i}$, $\eta_i$,  $\tilde{r}_{e,i} >0$ for all $i \in \mathcal{V}$ and $\epsilon_1, \epsilon_2, \epsilon_3, {\xi}_{e} >0$, such that \req{propositionconstone}-\req{propositionconstfive} are satisfied. Since $\tilde{\sigma}_i = 1- \sigma_i$ for all $i \in {\cal V}$, the (optimal) event-triggering gains are given by \req{sigmai} and \req{etanotinc} after solving the geometric programming problem in \rprop{gppropositione}. \qedwhite

\section{Proof of \rprop{feasibilityresult}}
Let $p^{(1)} = [p^{(1)}_{1}, p^{(1)}_{2},\ldots, p^{(1)}_n]$ and $p^{(2)} = [p^{(2)}_{1}, p^{(2)}_{2},\ldots, p^{(2)}_n]$ denote the solution to \req{designlyapunov} with $c_p= \gamma^{(1)}$ and $c_p = \gamma^{(2)}$, respectively, where $\gamma^{(1)}, \gamma^{(2)} >0$ are any positive constants. 
From \req{designlyapunov}, it can be easily shown that $p^{(2)} = \gamma p^{(1)}$, where $\gamma = \frac{\gamma^{(2)}}{\gamma^{(1)}}$. 
%\begin{comment}
Here, we will only prove for the geometric program in Proposition~2 (i.e., the feasibility of the geometric program in Proposition~2 with $p^{(1)}$ implies the one with $p^{(2)}$ and vice versa), since the proof for Proposition~3 is given in the same way with Proposition~2. 
%In what follows, it is shown that the feasibility of the geometric program in Proposition~2 with $p^{(1)}$ implies the feasibility of the geometric program in Proposition~2 with $p^{(2)}$ (and vice versa). 
Let $\tilde{k}^{(1)} _i, \tilde{l}^{(1)} _{ij}, \tilde{s}^{(1)} _{c,i} >0$ for all $i \in \mathcal{V}$, $j \in {\cal N}^{\mathrm{out}} _i$ and $\epsilon^{(1)}_1, \epsilon^{(1)}_2, \epsilon^{(1)}_3, {\xi}^{(1)} _{c}>0$ be an \textit{any} feasible solution that satisfies the posynomial constraints in Proposition~2 with $p^{(1)}$. In other words, we have
\begin{align}
  &\tilde{s}^{(1)}_{c,i}+ p^{(1)}_i\tilde{k}^{(1)}_i + \sum_{j\in \mathcal{N}^{out}_i}p^{(1)}_j\tilde{l}^{(1)}_{ij}\le c^{(1)}_{1,i},\ \forall i \in \mathcal{V}\label{one}\\
  &\tilde{k}^{(1)}_i + \epsilon^{(1)} _1 \leq \bar{k}^{(1)}_i,\ \forall i \in \mathcal{V} \label{two} \\ 
  &\tilde{l}^{(1)}_{ij} + \epsilon^{(1)} _2 \leq \bar{l}^{(1)}_{ij},\ \forall i \in \mathcal{V}, \forall j \in {\cal N}^{out} _i, \label{three}\\
  &({\xi}^{(1)}_c) ^{\frac{1}{2}}+\sum_{i\in \mathcal{C}}\frac{p^{(1)}_i(\tilde{r}^{(1)}_{c,i}+\epsilon^{(1)}_3)}{\tilde{s}_{c,i}}\le c^{(1)}_{2,m},\ \forall m\in \{1,\ldots, M\}\label{four}\\
  &\left(\sum_{i \in \mathcal{V}}\frac{(p^{(1)} _i)^2}{\tilde{s}^{(1)}_{c,i}}\right)\left(\sum_{i \in \mathcal{V}}\frac{(\tilde{r}^{(1)}_{c,i}+\epsilon^{(1)}_3)^{2}}{\tilde{s}^{(1)}_{c,i}}\right) \le {\xi}^{(1)}_c, \label{five}
\end{align}
where 
\begin{align}
&c^{(1)}_{1,i}= p^{(1)}_i\bar{k}_i+\sum_{j\in \mathcal{N}^{out}_i}p^{(1)}_j\bar{l}_{ij}, \ \ \tilde{r}^{(1)}_{c, i} = \sum_{j\in \mathcal{N}^{out}_i} p^{(1)} _j\overline{\beta}_{ij}-p^{(1)} _i\underline{\delta}_i, \notag \\ 
&c^{(1)}_{2,m}= 2p^{*(1)}_m \bar{d}_m - \sum_{i\notin \mathcal{C}}\cfrac{p^{(1)}_i \tilde{r}^{(1)}_{c, i}}{c^{(1)}_{1,i}}
\end{align}
 with $p_m^{*(1)}=\min_{i\in \mathrm{supp}( {w}_m)} p^{(1)}_i$.
Multiplying $\gamma$ in both sides of \req{one}, we obtain
\begin{align*}
    \gamma \tilde{s}^{(1)}_{c,i}+ \gamma p^{(1)}_i\tilde{k}^{(1)} _i + \sum_{j\in \mathcal{N}^{out}_i}\gamma p^{(1)}_j\tilde{l}^{(1)}_{ij} &=\gamma \tilde{s}^{(1)}_{c,i} + p^{(2)}_i\tilde{k}^{(1)}_i + \sum_{j\in \mathcal{N}^{out}_i} p^{(2)}_j\tilde{l}^{(1)}_{ij} \\ 
   & \le \gamma c^{(1)}_{1,i}. %=  c^{(2)}_{1,i},
\end{align*}
Letting $c^{(2)}_{1,i} = p^{(2)}_i\bar{k}_i+\sum_{j\in \mathcal{N}^{out}_i}p^{(2)}_j\bar{l}_{ij}$, we obtain 
\begin{align}
    \gamma \tilde{s}^{(1)}_{c,i} + p^{(2)}_i\tilde{k}^{(1)}_i + \sum_{j\in \mathcal{N}^{out}_i} p^{(2)}_j\tilde{l}^{(1)}_{ij} \leq c^{(2)}_{1,i}, \label{six}
\end{align}
where we used $\gamma c^{(1)}_{1,i} = c^{(2)}_{1,i}$. 
 Moreover, by multiplying $\gamma$ in both sides of \req{four}, it follows that 
\begin{align*}
\gamma ({\xi}^{(1)}_c) ^{\frac{1}{2}}+\sum_{i\in \mathcal{C}}{\gamma p^{(1)}_i(\gamma \tilde{r}^{(1)}_{c,i}+\gamma\epsilon^{(1)}_3)}\tilde{s}^{-1}_{c,i}\gamma^{-1} &= \gamma ({\xi}^{(1)}_c) ^{\frac{1}{2}}+\sum_{i\in \mathcal{C}}{ p^{(2)}_i( \tilde{r}^{(2)}_{c,i}+\gamma\epsilon^{(1)}_3)}\tilde{s}^{-1}_{c,i}\gamma^{-1} \notag \\ 
&\le \gamma c^{(1)}_{2,m}. %= c^{(2)}_{2,m}, 
\end{align*}
Letting $c^{(2)}_{2,m}= 2p^{*(2)}_m \bar{d}_m - \sum_{i\notin \mathcal{C}}\frac{p^{(2)}_i \tilde{r}^{(2)}_{c, i}}{c^{(2)}_{1,i}}$ with $p_m^{*(2)}=\min_{i\in \mathrm{supp}( {w}_m)} p^{(2)}_i$, we obtain
\begin{align}
    \gamma ({\xi}^{(1)}_c) ^{\frac{1}{2}}+\sum_{i\in \mathcal{C}}{ p^{(2)}_i( \tilde{r}^{(2)}_{c,i}+\gamma\epsilon^{(1)}_3)} (\gamma \tilde{s}^{(1)}_{c,i})^{-1} \leq c^{(2)}_{2,m},  \label{seven}
\end{align}
where we used $\gamma c^{(1)}_{2,m} = c^{(2)}_{2,m}$. 
 In addition, by multiplying $\gamma^2$ in both sides of \req{five} in the above, we obtain 
\begin{align}
    \left(\sum_{i \in \mathcal{V}}\frac{(p^{(2)} _i)^2}{\gamma \tilde{s}^{(1)}_{c,i}}\right)\left(\sum_{i \in \mathcal{V}}\frac{(\tilde{r}^{(2)}_{c,i}+\gamma\epsilon^{(1)}_3)^{2}}{\gamma \tilde{s}^{(1)}_{c,i}}\right) \le  \gamma^2 {\xi}^{(1)}_c \label{eight}
\end{align}

Now, consider the feasibility problem of the posynomial constraints \textit{with $p^{(2)}$}: find $\tilde{k}_i, \tilde{l}_{ij}, \tilde{s}_{c,i} >0$ for all $i \in \mathcal{V}$, $j \in {\cal N}^{out} _i$ and $\epsilon_1, \epsilon_2, \epsilon_3, {\xi}_{c}>0$ such that
\begin{align*}
  &\tilde{s}_{c,i}+ p^{(2)}_i\tilde{k}_i + \sum_{j\in \mathcal{N}^{out}_i}p^{(2)}_j\tilde{l}_{ij}\le c^{(2)}_{1,i},\ \forall i \in \mathcal{V}\\
  &\tilde{k}_i + \epsilon_1 \leq \bar{k}_i,\ \forall i \in \mathcal{V} \\ 
  &\tilde{l}_{ij} + \epsilon_2 \leq \bar{l}_{ij},\ \forall i \in \mathcal{V}, \forall j \in {\cal N}^{out} _i,\\
  &({\xi}_c) ^{\frac{1}{2}}+\sum_{i\in \mathcal{C}}{p^{(2)}_i(\tilde{r}^{(2)}_{c,i}+\epsilon_3)}\tilde{s}^{-1}_{c,i}\le c^{(2)}_{2,m},\ \forall m\in \{1,\ldots, M\}\\
  &\left(\sum_{i \in \mathcal{V}}{(p^{(2)} _i)^2}{\tilde{s}^{-1}_{c,i}}\right)\left(\sum_{i \in \mathcal{V}}{(\tilde{r}^{(2)}_{c,i}+\epsilon_3)^{2}}{\tilde{s}^{-1}_{c,i}}\right) \le {\xi}_c,
\end{align*}
\req{six}, \req{seven}, \req{eight} imply that the above problem is \textit{feasible} with 
\begin{align}
    &\tilde{k}_i = \tilde{k}^{(1)}_i, \ \tilde{l}_{ij}= \tilde{l}^{(1)}_{ij}, \ \tilde{s}_{c,i} = \gamma \tilde{s}^{(1)}_{c,i}, \\
    &\epsilon_1 =\epsilon^{(1)}_1, \  \epsilon_2 =\epsilon^{(1)}_2,\ \epsilon_3 =\gamma \epsilon^{(1)}_3,\ \xi_c = \gamma^2 \xi^{(1)}_c. 
\end{align}
Thus, the feasibility of the geometric program in Proposition~2 with $p^{(1)}$ implies the feasibility of the geometric program in Proposition~2 with $p^{(2)}$. 
It can be shown in the same way that the feasibility of geometric program with $p^{(2)}$ implies the feasibility of the geometric program with $p^{(1)}$. \qedwhite 

\end{document}